\documentclass[11pt,letterpaper, oneside]{amsart}
\usepackage{amsmath}
\usepackage{amsthm}
\usepackage{amsfonts}
\usepackage{amssymb,amscd,epsf,verbatim}
\usepackage{mathrsfs}
\usepackage{graphicx}
\usepackage{latexsym}
\usepackage{lscape}
\usepackage[colorlinks=true]{hyperref}
\hypersetup{colorlinks, citecolor=blue, filecolor=black, linkcolor=red, urlcolor=green}
\usepackage{epstopdf}
\usepackage{tikz}
\usetikzlibrary{calc}
\usetikzlibrary{matrix,arrows,decorations.pathmorphing}
\usepackage{tikz-cd}
\usepackage{color}
\usepackage{geometry}
\usepackage{multirow}
\usepackage{stmaryrd}
\usepackage{cite}

\usepackage{chngcntr}
\counterwithin{figure}{subsection}
\counterwithin{equation}{subsection}

\newcommand{\C}{\mathbb{C}}
\newcommand{\R}{\mathbb{R}}
\newcommand{\Z}{\mathbb{Z}}
\newcommand{\Q}{\mathbb{Q}}
\newcommand{\F}{\mathbb{F}}
\newcommand{\G}{\mathbb{G}}
\newcommand{\N}{\mathbb{N}}
\renewcommand{\P}{\mathbb{P}}
\newcommand{\A}{\mathbb{A}}
\newcommand{\cO}{\mathcal O}
\newcommand{\Spec}{\operatorname{Spec}}
\newcommand{\Sp}{\operatorname{Sp}}
\newcommand{\Spf}{\operatorname{Spf}}
\newcommand{\Spa}{\operatorname{Spa}}
\newcommand{\MaxSpec}{\operatorname{MaxSpec}}
\newcommand{\et}{\text{\'et}}
\newcommand{\ad}{\text{ad}}
\newcommand{\perf}{\text{perf}}

\DeclareMathOperator{\Hom}{Hom}
\DeclareMathOperator{\Ext}{Ext}
\DeclareMathOperator{\Pic}{Pic}

\DeclareMathOperator{\Gal}{Gal}
\DeclareMathOperator{\tIm}{Im}

\DeclareMathOperator{\GL}{GL}
\DeclareMathOperator{\An}{An}
\DeclareMathOperator{\Frob}{Frob}
\DeclareMathOperator{\Fil}{Fil}
\DeclareMathOperator{\Gr}{Gr}

\newcommand{\fE}{\mathfrak E}
\newcommand{\fA}{\mathfrak A}
\newcommand{\fL}{\mathfrak L}

\newcommand{\fT}{\mathfrak T}
\newcommand{\fX}{\mathfrak X}
\newcommand{\fG}{\mathfrak G}

\newcommand{\fp}{\mathfrak p}

\newcommand{\fTn}{\mathfrak T_{\alpha/p^n}}
\newcommand{\fEn}{\mathfrak E_{\alpha/p^n}}
\newcommand{\fAn}{\mathfrak A_{\alpha/p^n}}
\newcommand{\fLn}{\mathfrak L_{\alpha/p^n}}

\newcommand{\cL}{\mathcal{L}}
\newcommand{\cU}{\mathcal{U}}

\theoremstyle{theorem}
\newtheorem{theorem}{Theorem}[section]
\newtheorem{lemma}[theorem]{Lemma}
\newtheorem{proposition}[theorem]{Proposition}
\newtheorem{corollary}[theorem]{Corollary}

\newtheorem{conjecture}[theorem]{Conjecture}

\newtheorem{DefProp}[theorem]{Defintion/Proposition}
\newtheorem{question}[theorem]{Question}

\newtheorem{definition}[theorem]{Definition}

\theoremstyle{definition}

\newtheorem{notation}[theorem]{Notation}
\newtheorem{example}[theorem]{Example}
\newtheorem{remark}[theorem]{Remark}

\theoremstyle{remark}
\newtheorem*{theorem*}{Theorem}
\newtheorem*{lemma*}{Lemma}

\numberwithin{theorem}{subsection}

\title[Perfectoid covers of abelian varieties and weight-monodromy]{Perfectoid covers of abelian varieties and the weight-monodromy conjecture}
\author{Peter Wear}

\begin{document}
\maketitle

\begin{abstract}
Deligne's weight-monodromy conjecture gives control over the poles of local factors of $L$-functions of varieties at places of bad reduction. His proof in characteristic $p$ was a step in his proof of the generalized Weil conjectures. Scholze developed the theory of perfectoid spaces to transfer Deligne's proof to characteristic 0, proving the conjecture for complete intersections in toric varieties. Building on Scholze's techniques, we prove the weight-monodromy conjecture for complete intersections in abelian varieties.
\end{abstract}

\section{Introduction}\label{sec:intro}

\subsection{The weight-monodromy conjecture}

The goal of this paper is to prove new cases of Deligne's weight-monodromy conjecture, building off of techniques developed by Scholze. In the introduction, we motivate and state the conjecture, sketch Scholze's approach, and then give an outline of the paper. To simplify the exposition of the introduction we will work over $\Q$, but everything we state holds for any number field. 

Let $X$ be a variety over $\Q$ of dimension $d$. The $L$-function of $X$ is a Dirichlet series defined as a product of local zeta factors over each prime, and is a fundamental object of study in the Langlands program. There are many important conjectures about $L$-functions: they should come from automorphic representations, satisfy a functional equation, and contain lots of information about the variety $X$. For example, the Birch and Swinnerton-Dyer conjecture says that the rank of the group of rational points on an elliptic curve $E$ should be given by the order of the zero of the $L$-function $L(E,s)$ associated to $E$ at $s=1$. A stronger version of the conjecture gives an interpretation of the leading coefficient of the Taylor series at 1 in terms of fundamental invariants of $E$.

When $X$ has good reduction at a prime $p$, the Weil conjectures give us the local factor at $p$. The Riemann hypothesis implies that the poles of the $i$th local factors at such primes will have real part $i/2$. But there will generally be finitely many primes over which $X$ has bad reduction, so we'd like to understand the local factors at these primes as well. Deligne's weight-monodromy conjecture implies that even at primes of bad reduction, all of the poles of the $i$th piece of the local zeta function will have real part at most $i/2$. 

We now give an exposition of how these local zeta factors are constructed via \'etale cohomology, and give a precise statement of the weight-monodromy conjecture. Note that this is just for motivation: the reader just interested in the new arguments can skip straight to the statement of Conjecture \ref{conj:WMC}. Let $\ell$ be a prime different from $p$, let $i$ be an integer between $0$ and $2d$. The $i$th $\ell$-adic cohomology group $V:=H^i_{\et}(X_{\overline\Q_p},\overline\Q_\ell)$ is a finite-dimensional $\overline\Q_\ell$-vector space. The absolute Galois group $G_{\Q_p}:=\Gal(\overline\Q_p/\Q_p)$ acts on $X_{\overline\Q_p}$. This induces an action on $V$, giving a Galois representation \[\rho:G_{\Q_p}\rightarrow \GL(V).\] 

Roughly speaking, this representation is controlled by two parts: the action of Frobenius and the monodromy operator. Frobenius comes from $G_{\F_p}$, the Galois group of the residue field, and is used to define the weights of the representation, while the inertia in the Galois group leads to monodromy. As a first approximation, the weight-monodromy conjecture says that these two facets of the representation interact nicely.

The inertia subgroup $I_{\Q_p}$ of $G_{\Q_p}$ is defined via the short exact sequence \[1\rightarrow I_{\Q_p}\rightarrow G_{\Q_p}\rightarrow G_{\F_p}\rightarrow 1.\] Fix an element $\Phi\in G_{\Q_p}$ which maps to geometric Frobenius: the map $\overline\F_p\rightarrow\overline\F_p$ sending $x\mapsto x^{1/p}$. The local zeta factor of $X$ at $p$ is defined as $\operatorname{det}(1-p^{-s}\Phi|V^{I_{\Q_p}})$.

In the good reduction case, $X$ has a proper smooth model $\fX$ over $\Z_p$ with special fiber $\widetilde\fX/\F_p$. By proper smooth base change, there is a Galois-equivariant isomorphism 
\[H^i_{\et}(X_{\overline\Q_p},\overline\Q_{\ell})\cong H^i_{\et}(\widetilde\fX_{\overline\F_p},\overline\Q_\ell),\]
where the right side has an action of $G_{\F_p}$. In this case, the inertia group acts trivially on $V$, so the local zeta factor is now $\operatorname{det}(1-p^{-s}\Phi|V)$. The Weil conjectures tell us that this is a polynomial with variable $T:=p^{-s}$ whose roots are \emph{Weil numbers of weight $i$}: algebraic integers with absolute value $p^{i/2}$ for any embedding into $\C$. 

In the bad reduction case, the inertia group acts non-trivially on $V$. Grothendieck showed that the action of the pro-$\ell$ part of the tame inertia of $G_{\Q_p}$ leads to a nilpotent operator $N:V\rightarrow V(-1)$ called the \emph{monodromy operator} (here $V(-1)$ denotes a Tate twist of $V$). This operator gives rise to the \emph{monodromy filtration} on V: a finite increasing filtration $\{V_j\}_{j\in\Z}$ such that $N(V_j)\subset V_{j-2}$, and such that $N^j$ induces an isomorphism between the $(i-j)$th and $(i+j)$th graded pieces of the filtration.

The other half of the picture is the weight filtration. The image under $\rho$ of our lift $\Phi$ of geometric Frobenius is a matrix $\rho(\Phi)\in\GL(V)$, which we can use to give an generalized eigenspace decomposition \[V=\displaystyle\oplus^{2i}_{j=0} W'_j.\] Here the eigenvalues of $\Phi$ acting on $W'_j$ are Weil numbers of weight $j$. If we chose a different lift of Frobenius, we get a different decomposition of $V$. However, the weights and dimensions remain the same, as well as the \emph{weight filtration} defined by \[W_m:=\displaystyle\oplus_{j=0}^m W'_j.\] For any choice of $\Phi$, we have $N\Phi=p\Phi N$, so the monodromy operator interacts nicely with the weight filtration: $NW_m\subset W_{m-2}$.

\begin{conjecture}[Deligne, \cite{Hodge}]\label{conj:WMC}
For a proper smooth variety over a local field, the weight filtration is the same as the monodromy filtration.
\end{conjecture}

Deligne proved his conjecture over characteristic $p$ local fields (under a simplifying assumption which was later removed in \cite{Terasoma} and \cite{Ito}). In mixed characteristic, many special cases are known - see the introduction of \cite{PS} for an overview - but the general conjecture is open. In his thesis \cite{PS}, Scholze introduced perfectoid spaces as a bridge between characteristics 0 and $p$. He proved the weight-monodromy conjecture for complete intersections in toric varieties by using perfectoid spaces to reduce to the characteristic $p$ case that Deligne settled. The goal of this paper is to extend Scholze's techniques from toric varieties to abelian varieties, resulting in our main theorem.

\begin{theorem}\label{thm:WMC abelian varieties}
Let $k$ be a local field of residue characteristic $p$. Let $A$ be an abelian variety over $k$, let $Y$ be a geometrically connected, proper smooth variety which is a set-theoretic complete intersection in $A$. Then the weight-monodromy conjecture holds for $Y$.
\end{theorem}

When $Y$ has dimension at least 3, the varieties that are provably handled by Theorem \ref{thm:WMC abelian varieties} are different from those provably handled by Scholze's theorem: see Remark \ref{rem:lefschetz}. 

\subsection{Scholze's approach via perfectoid spaces}

Our proof of Theorem \ref{thm:WMC abelian varieties} follows the same path as Scholze's proof for toric varieties in \cite[Theorem 9.6]{PS}. We sketch his proof here in the special case of a hypersurface $Y$ in projective space $\P^n_{\Q_p}$. A more detailed and general version of this argument is given in Section \ref{sec:weight-monodromy}.

Let $K$ denote the perfectoid field $\widehat{\Q_p(p^{1/p^\infty})}$, take the base change to $Y_K$ and consider the corresponding $G_K$-representation on \'etale cohomology. This doesn't change the weights as any lift of geometric Frobenius in $G_K$ gives one in $G_{\Q_p}$. It preserves the monodromy operator as it doesn't affect the pro-$\ell$ part of the Galois group. We can therefore prove weight-monodromy after making this base change.

The \emph{tilt} $K^\flat:=\varprojlim_{x\mapsto x^p} K\cong\widehat{\F_p((t^{1/p^\infty}))}$ of $K$ is a perfectoid field of characteristic $p$. The starting point of the theory of perfectoid spaces is the Fontaine-Wintenberger isomorphism of Galois groups $G_K\cong G_{K^\flat}.$ Deligne proved that weight-monodromy holds for $G_{K^\flat}$-representations coming from the \'etale cohomology of varieties over $K^\flat$. This isomorphism of Galois groups allows us to move this result to the $G_K$-representations that we hope to study.

Consider the map of adic spaces $\varphi:\P_K^{n,\ad}\rightarrow\P_K^{n,\ad}$ defined on points by sending $(x_0:\cdots:x_n)$ to $(x_0^p:\cdots:x_n^p)$. Iterating this map to get an inverse system, we get a perfectoid space \[\P_K^{n,\perf}\sim\varprojlim_{\varphi} \P_K^{n,\ad}.\] 
The space $\P_K^{n,\perf}$ isn't an inverse limit of adic spaces, but we get inverse limits if we restrict to the underlying topological spaces or \'etale topoi. This allows topological and \'etale cohomological arguments to work in this setting. The tilt $(\P_K^{n,\perf})^\flat$ is given by the analogous construction over $K^\flat$. The story here is even simpler as the map $\varphi_{K^\flat}:\P_{K^\flat}^{n,\ad}\rightarrow\P_{K^\flat}^{n,\ad}$ induces a homeomorphism of topological spaces and an isomorphism of \'etale topoi. 

Putting all of this together, we get a projection map $\pi:\P_{K^\flat}^{n,\ad}\rightarrow\P_K^{n,\ad}$ on the underlying topological spaces and \'etale topoi. This map is Galois-equivariant under the Fontaine-Wintenberger isomorphism. This suggests the following method for proving weight-monodromy: take a variety over $K$, embed it in projective space, pull back along $\pi$ to get something over $K^\flat$, and apply Deligne's result. The issue is, the pullback won't be a variety: it'll be a transcendental fractally mess. Scholze fixed this problem for hypersurfaces by showing that this pullback can be approximated by a hypersurface with the same $\ell$-adic cohomology.

\subsection{Outline}

With the preceding sketch in hand, we now have enough context give a full outline of this paper. In Section \ref{sec:geometry background}, we give background on the various approaches to non-archimedean geometry that we will use throughout. We have already seen that adic spaces are needed to define perfectoid spaces, but much of our actual work will be done in the more classical categories of rigid analytic spaces and formal schemes. We recall Raynaud's description of rigid spaces as the generic fibers of formal schemes and his uniformization of abelian varieties over non-archimedean fields. We then explain how to construct perfectoid spaces as inverse limits of rigid spaces via what we call ``$F$-towers": certain inverse systems of formal schemes. 

In Section \ref{sec:perfectoid covers}, we use $F$-towers to construct perfectoid covers of abelian varieties (more generally, of abeloids - the nonarchimedean analogue of complex tori) giving an alternate proof of the following theorem. 

\begin{theorem}[Blakestad, Gvirtz, Heuer, Shchedrina, Shimizu, W., Yao]\label{intro:AWSthm}
Let $A$ be an abelian variety (or abeloid) over a perfectoid field $K$ of residue characteristic $p$ with value group contained in $\Q$. There is a perfectoid group $A_\infty$ such that $A_\infty\sim \varprojlim_{[p]} A$. 
\end{theorem}

We note that this theorem was already known when $A$ has good reduction, see Example \ref{eg:good reduction}. The proof in \cite{AWS} assumes that $K$ is algebraically closed but doesn't have the restriction on the value group. The cover $A_\infty$ of $A$ is our analogue of the perfectoid cover $\P^{n,\perf}_K$ of $\P^n_K$. In the case of projective space, the perfectoid cover can be built from explicit perfectoid $K$-algebras obtained by iterating a lift $\varphi$ of Frobenius. This isn't an option for abelian varieties: it's very hard to write down explicit formulas for them, and there isn't always a lift of Frobenius. The restriction on the value group allows us to construct formal models of $A$.

In Section \ref{sec:tilting abeloids}, we begin our study of the tilt $A_\infty^\flat$. For $\P_K^{n,\ad}$, there was a natural interpretation of the tilt as the perfectoid space $\P_{K^\flat}^{n,\perf}$ sitting above $\P_{K^\flat}^{n,\ad}$. This again could be checked by explicitly working with the perfectoid $K$-algebras involved. There isn't such an obvious candidate for an abelian variety over $K^\flat$ that gives rise to $A_\infty^\flat$. Roughly speaking, a perfectoid space and its tilt will have isomorphic special fibers. To understand $A_\infty^\flat$, we take the special fiber of the $F$-tower used to define $A_\infty$ and deform it to get an $F$-tower over an abeloid $A'/K^\flat$. This gives an alternate proof of the theorem

\begin{theorem}[Heuer, W.]\label{intro:abeloidtilt-intro}
Let $A$ and $A_\infty$ be as in Theorem \ref{intro:AWSthm}. Then after a pro-$p$ extension of $K$, there is a (non-unique) abeloid $A'$ over $K^\flat$ such that $A_\infty^\flat\sim\varprojlim_{[p]} A'.$
\end{theorem}

The proof in \cite{HW} relates the construction to Scholze's perfectoid Shimura varieties. It does not use formal models, which we will need for our application.

In Section \ref{sec:line bundles}, we explain how to transfer line bundles from $A$ to a suitable choice of $A'$ by constructing perfectoid covers of the associated $\G_m$-torsors. This allows us to algebraize the previous theorem: if we start with an abelian variety $A$, we can choose $A'$ to be an abelian variety as well. 

In Section \ref{sec:weight-monodromy}, we prove an analogue of Scholze's approximation lemma, which we use to prove Theorem \ref{thm:WMC abelian varieties}. The cohomological aspects of the proof are similar to Scholze's argument, with one addition necessary. As the map $[p]:A'\rightarrow A'$ is not a homeomorphism and doesn't induce an isomorphism of \'etale sites, there is no straightforward analogue of Scholze's map $\pi:\P^n_{K^\flat}\rightarrow\P^n_K$ at finite level. Our computations therefore have to pass through the perfectoid cover $A_\infty^\flat$ more explicitly.

\subsection{Acknowledgements}
This paper is based off of my PhD thesis at UCSD. I would like to thank my advisor, Kiran Kedlaya, for his advise and support. This project began at the 2017 Arizona Winter School. I thank the organizers for making this possible; Bhargav Bhatt and Matthew Morrow for suggesting the problem of constructing perfectoid covers of abelian varieties; and my groupmates Clifford Blakestad, Dami\'an Gvirtz, Ben Heuer, Daria Shchedrina, Koji Shimizu, and Zijian Yao. I thank Gabriel Dorfsman-Hopkins, Ben Heuer, Sean Howe, and Anwesh Ray for many helpful conversations related to this material. This paper was written with support from NSF grant DMS-1502651 and NSF RTG grant \#1840190.


\section{Background on non-archimedean geometry}\label{sec:geometry background}

In this section, we collect some results from non-archimedean geometry. Our overall goal is to take an abelian variety $A$ over a perfectoid field $K$ with residue characteristic $p$ and construct a perfectoid space $A_\infty\sim \varprojlim_{[p]} A$. To do this, we recall Raynaud's theory of formal models of rigid spaces and his uniformization of abelian varieties. We then axiomatize a method of Scholze for constructing perfectoid tilde-limits of inverse system of adic spaces via formal models. Finally, we recall a result of Shen on Galois descent of perfectoid tilde-limits that will allow us to assume that $A$ has a Raynaud uniformization. Additional exposition and references on these results can be found in \cite[Chapter 2]{Thesis}.


\subsection{Inverse limits of rigid spaces via formal models}

Let $K$ be a non-archimedean field with absolute value $|\cdot|$, residue characteristic $p$, and residue field $\kappa$. For any topological ring $R$, we write $R^\circ$ for the power-bounded elements and $R^{\circ\circ}$ for the topologically nilpotent, so $K^\circ$ is the ring of integers of $K$ and $K^{\circ}/K^{\circ\circ}\cong \kappa$. The categories of rigid spaces and adic spaces over $K$ do not admit inverse limits in general. To get around this issue, we work with formal models. We quickly recall the basic construction, see \cite[Chapters 7,8]{Bos} for details. 

\begin{definition}\label{def:formal scheme}
A $K^\circ$-algebra $R^+$ is \emph{admissible} if it has no $K^\circ$-torsion and it is of topologically finite type. Fix any pseudo-uniformizer $\varpi\in K^\circ$ with $|0|<|\varpi|<1$, then $A^\circ$ is $\varpi$-adically complete.
\end{definition}

Having no $K^\circ$-torsion is equivalent to being flat over $K^\circ$.

\begin{definition}\label{def:formal analytic cover}
A \emph{formal analytic cover} of a rigid space $X$ is a cover $\{U_i\}$ of $X$ by affinoid opens $U_i=\Sp(R_i)$ such that 

	\begin{enumerate}
	\item The cover is admissible in the $G$-topology of $X$ (cf. \cite[Chapter 5]{Bos}),
	\item\label{fas1} For each algebra $R_i$, the $K^\circ$-algebra $R_i^\circ$ is admissible as in Definition \ref{def:formal scheme}, and  
	\item For each pair of opens $U_i$ and $U_j$, the intersection $U_i\cap U_j\subset U_i$ is the preimage of a Zariski open subset of $\MaxSpec(\widetilde R_i)$ under the reduction map $R_i^\circ\rightarrow \tilde{R}_i:=R^\circ/R^{\circ\circ}$. 
	\end{enumerate}
\end{definition}

We can now state the precise construction that we'll need. This is essentially the content of \cite[Remark 3.2.5]{Lut}.  

\begin{proposition}\label{prop:fas to formal scheme}\hfill
\begin{enumerate}
\item The data of a formal analytic cover of $X$ defines an admissible formal scheme $\mathfrak X$ obtained by gluing the formal schemes $\Spf(R_i^\circ)$. The formal scheme $\mathfrak X$ is a \emph{formal model} of $X$. 
\item Given any cover of $\fX$ by formal affine opens, the generic fiber is a formal analytic cover of $X$.
\item Let $\{U_i\}$ be a formal analytic cover of $X$ and $\{V_k\}$ be a formal analytic cover of $Y$, let $\mathfrak X$ and $\mathfrak Y$ be the corresponding formal models. Then a morphism of rigid spaces $\varphi:X\rightarrow Y$ extends to a morphism of formal models $\overline{\varphi}:\mathfrak X\rightarrow \mathfrak Y$ if for each $U_i$ in the cover of $X$, there is some $V_i$ in the cover of $Y$ such that $\varphi(U_i)\subset V_i$ and $\varphi(U_i\cap U_j)\subset V_i\cap V_j$ for all $i,j$. 
\end{enumerate}
\end{proposition}

\begin{example}\label{eg:formal A1}
A formal analytic cover of rigid analytic affine space $\A_K^{1,\operatorname{an}}$ is given by fixing some $y\in K$ with $|y|=c<1$ and taking the unit disk $D(1,0)$ and the annuli $\An(c^{-n},c^{-n-1})$ for all integers $n\geq 0$, glued along the border circles. The reduction of the disk is $\A^1_{\kappa}$, and the reduction of each annulus is two copies of $\A^1_{\kappa}$ glued at the origin. These reductions are glued to form a infinite chain of $\P^1_{\kappa}$s, indexed by $\N$. 
\end{example}

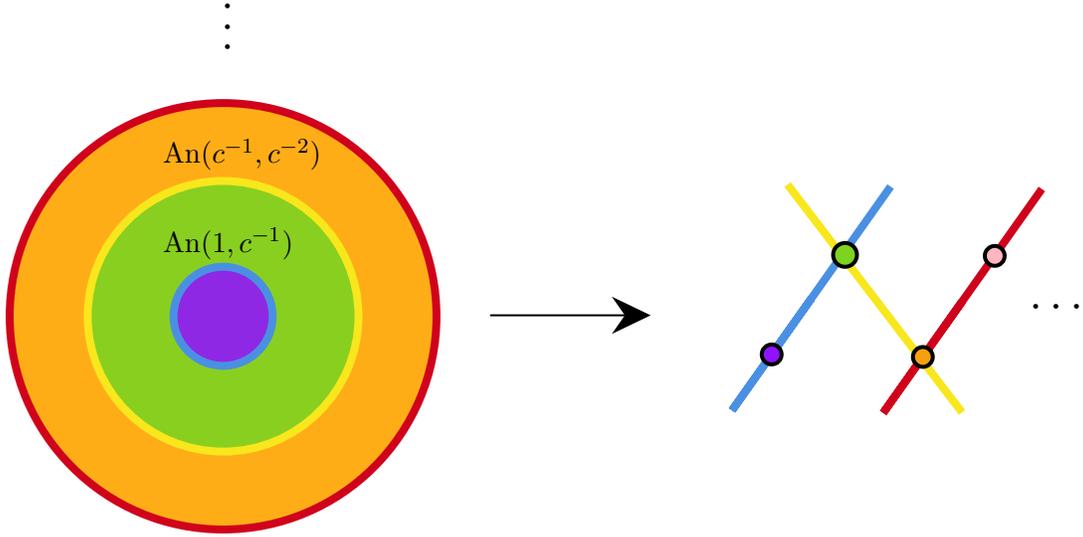
\begin{figure}

\tikzset{every picture/.style={line width=0.75pt}} 

\begin{tikzpicture}[x=0.75pt,y=0.75pt,yscale=-1,xscale=1]

\draw [color={rgb, 255:red, 248; green, 231; blue, 28 }  ,draw opacity=1 ][line width=3]    (441.5,94.69) -- (529.5,209.69) ;
\draw  [color={rgb, 255:red, 208; green, 2; blue, 27 }  ,draw opacity=1 ][fill={rgb, 255:red, 255; green, 165; blue, 0 }  ,fill opacity=0.91 ][line width=3]  (49.2,161.17) .. controls (49.2,101.73) and (97.39,53.55) .. (156.83,53.55) .. controls (216.27,53.55) and (264.45,101.73) .. (264.45,161.17) .. controls (264.45,220.61) and (216.27,268.8) .. (156.83,268.8) .. controls (97.39,268.8) and (49.2,220.61) .. (49.2,161.17) -- cycle ;
\draw  [color={rgb, 255:red, 248; green, 231; blue, 28 }  ,draw opacity=1 ][fill={rgb, 255:red, 126; green, 211; blue, 33 }  ,fill opacity=0.91 ][line width=3]  (88.54,161.17) .. controls (88.54,123.46) and (119.11,92.88) .. (156.83,92.88) .. controls (194.54,92.88) and (225.12,123.46) .. (225.12,161.17) .. controls (225.12,198.89) and (194.54,229.46) .. (156.83,229.46) .. controls (119.11,229.46) and (88.54,198.89) .. (88.54,161.17) -- cycle ;
\draw [color={rgb, 255:red, 74; green, 144; blue, 226 }  ,draw opacity=1 ][line width=3]    (493.5,95.69) -- (413.5,208.69) -- (453.5,152.19) ;
\draw  [color={rgb, 255:red, 0; green, 0; blue, 0 }  ,draw opacity=1 ][fill={rgb, 255:red, 126; green, 211; blue, 33 }  ,fill opacity=1 ][line width=1.5]  (464.5,130.19) .. controls (464.5,126.87) and (467.19,124.19) .. (470.5,124.19) .. controls (473.81,124.19) and (476.5,126.87) .. (476.5,130.19) .. controls (476.5,133.5) and (473.81,136.19) .. (470.5,136.19) .. controls (467.19,136.19) and (464.5,133.5) .. (464.5,130.19) -- cycle ;
\draw    (291.5,160.69) -- (369.5,160.69) ;
\draw [shift={(372.5,160.69)}, rotate = 180] [fill={rgb, 255:red, 0; green, 0; blue, 0 }  ][line width=0.08]  [draw opacity=0] (19.65,-9.44) -- (0,0) -- (19.65,9.44) -- (13.05,0) -- cycle    ;
\draw  [color={rgb, 255:red, 0; green, 0; blue, 0 }  ,draw opacity=1 ][fill={rgb, 255:red, 144; green, 19; blue, 254 }  ,fill opacity=1 ][line width=1.5]  (428.5,180.44) .. controls (428.5,177.68) and (430.74,175.44) .. (433.5,175.44) .. controls (436.26,175.44) and (438.5,177.68) .. (438.5,180.44) .. controls (438.5,183.2) and (436.26,185.44) .. (433.5,185.44) .. controls (430.74,185.44) and (428.5,183.2) .. (428.5,180.44) -- cycle ;
\draw  [color={rgb, 255:red, 74; green, 144; blue, 226 }  ,draw opacity=1 ][fill={rgb, 255:red, 144; green, 19; blue, 254 }  ,fill opacity=0.89 ][line width=3]  (131.83,161.17) .. controls (131.83,147.36) and (143.02,136.17) .. (156.83,136.17) .. controls (170.64,136.17) and (181.83,147.36) .. (181.83,161.17) .. controls (181.83,174.98) and (170.64,186.17) .. (156.83,186.17) .. controls (143.02,186.17) and (131.83,174.98) .. (131.83,161.17) -- cycle ;
\draw [color={rgb, 255:red, 208; green, 2; blue, 27 }  ,draw opacity=1 ][line width=3]    (569.75,96.94) -- (489.75,209.94) -- (529.75,153.44) ;
\draw  [color={rgb, 255:red, 0; green, 0; blue, 0 }  ,draw opacity=1 ][fill={rgb, 255:red, 246; green, 160; blue, 17 }  ,fill opacity=1 ][line width=1.5]  (504.75,181.69) .. controls (504.75,178.93) and (506.99,176.69) .. (509.75,176.69) .. controls (512.51,176.69) and (514.75,178.93) .. (514.75,181.69) .. controls (514.75,184.45) and (512.51,186.69) .. (509.75,186.69) .. controls (506.99,186.69) and (504.75,184.45) .. (504.75,181.69) -- cycle ;
\draw  [fill={rgb, 255:red, 255; green, 182; blue, 193 }  ,fill opacity=1 ][line width=1.5]  (541,130.69) .. controls (541,127.93) and (543.24,125.69) .. (546,125.69) .. controls (548.76,125.69) and (551,127.93) .. (551,130.69) .. controls (551,133.45) and (548.76,135.69) .. (546,135.69) .. controls (543.24,135.69) and (541,133.45) .. (541,130.69) -- cycle ;

\draw (125,115.09) node [anchor=north west][inner sep=0.75pt]    {$\An(1,c^{-1})$};
\draw (125,70.09) node [anchor=north west][inner sep=0.75pt]    {$\An(c^{-1},c^{-2})$};
\draw (165,0) node [anchor=north west][inner sep=0.75pt]  [font=\LARGE,rotate=-90]  {$\cdots $};
\draw (562,150.09) node [anchor=north west][inner sep=0.75pt]  [font=\LARGE]  {$\cdots $};

\end{tikzpicture}
\caption{A formal analytic cover of $\A^1_K$ with its special fiber}\label{fig:A1}
\end{figure}

\begin{example}\label{eg:formal Gm}
A key example for us will be the rigid analytification of the affine torus $\G_{m,K}$. This is just $\A^1_K$ with the origin removed, so a formal analytic cover is given by fixing some $y\in K$ with $|y|=c<1$ and taking the annuli $\An(c^{n+1},c^n)$ for all integers $n$, again glued along border circles. The reduction is now a two-sided infinite chain of $\P^1_\kappa$s, indexed by $\Z$.
\end{example}

In terms of adic spaces, the above construction corresponds to writing $X^{\ad}/K$ as the adic generic fiber of a formal scheme $\fX/K^\circ$. That is, $X^{\ad}=\fX^{\ad}_\eta:=\fX^{\ad}\times_{\Spa(K^\circ,K^\circ)}\Spa(K,K^\circ).$

In nice circumstances, the category of formal schemes admits inverse limits. 

\begin{lemma}\label{lemma:inverse limit formal schemes}
Let $\fX_i$ be an inverse system of formal schemes over $K^\circ$ with affine transition maps. Then the inverse limit $\fX=\varprojlim \fX_i$ exists in the category of formal schemes over $K^\circ$. If all the $\fX_i$ are flat over $K^\circ$, so is $\fX$.
\end{lemma}

\begin{proof}[Sketch]
See \cite[Tag 01YX]{Stacks} for details. As the transition maps are affine, we can reduce to the case where all the $\fX_i$ are affine. Say $\fX_i=\Spf(R_i)$, then we have $\fX=\Spf(\widehat{\varinjlim R_i})$, where we are taking the $\varpi$-adic completion of the direct limit. Flatness is equivalent to being torsion-free over $K^\circ$, so it is preserved by this process.  
\end{proof}

To translate these inverse limits to the generic fiber, we recall the definition of tilde-limits from \cite[Definition 7.14]{PS} along with some useful properties. Note that a slightly different and more general definition is given in \cite[Section 2.4]{SW}, but we will not need this.

\begin{definition}[{\cite[Definition 7.14]{PS}}]\label{def:tilde limit}
Let $X_i$ be a filtered inverse system of noetherian adic spaces. Let $X$ be a perfectoid space with a compatible family of morphisms $\phi_i:X\rightarrow X_i$. Then we write $X\sim\varprojlim X_i$ and call $X$ the \emph{tilde-limit} of the inverse system if the induced map of topological spaces $|X|\rightarrow \varprojlim |X_i|$ is a homeomorphism, and if for every $x\in X$ with $x_i:=\phi_i(x)\in X_i$, the map of residue fields \[\varinjlim k(x_i)\rightarrow k(x)\] is dense.
\end{definition}


\begin{proposition}[{\cite[Proposition 7.16]{PS}}]\label{prop:open immersion tilde}
Given $X\sim\varprojlim X_i$ as in Definition \ref{def:tilde limit}, let $Y_i\hookrightarrow X_i$ be an open immersion of noetherian adic spaces for some fixed $X_i$ in the inverse system. Let $Y_j:=Y_i\times_{X_i}X_j$ for $j\geq i$ and $Y:=Y_i\times_{X_i} X$. Then $Y\sim\varprojlim_{j\geq i} Y_j$.
\end{proposition}

To construct tilde-limits of inverse systems of rigid spaces, we construct a corresponding inverse system of well-behaved formal models of these spaces, use Lemma \ref{lemma:inverse limit formal schemes} to get the inverse limit in the category of formal schemes, then take the adic generic fiber of the limit.

\begin{lemma}\label{lemma:tilde generic fiber}
Let $\fX_i$ be a filtered inverse system of formal schemes over $K^\circ$ with affine transition maps, let $\fX$ be the inverse limit constructed in Lemma \ref{lemma:inverse limit formal schemes}. Let $X_i=(\fX_i)_\eta$ and $X=\fX_\eta$ be the adic generic fibers. Then $X\sim\varprojlim X_i$. \end{lemma}

\begin{proof}
As $\fX\cong\varprojlim \fX_i$, we also have $\fX\sim \varprojlim \fX_i$. Now we can apply Proposition \ref{prop:open immersion tilde} to the open immersion $\Spa(K,K^\circ)\mapsto \Spa(K^\circ,K^\circ)$, as the inverse system $\fX_i$ lives over $\Spa(K^\circ,K^\circ)$.
\end{proof}

For our purposes, a key result is that tilde-limits behave like inverse limits on \'etale cohomology. This will let us transfer cohomological results between characteristics by passing through perfectoid tilde-limits.

\begin{proposition}[{\cite[Theorem 7.17]{PS}}]\label{prop:tilde etale}
In the situation of Definition \ref{def:tilde limit}, the \'etale topos $X^{\sim}_{\et}$ of $X$ is the projective limit of the \'etale topoi $(X_i)^\sim_{\et}$. 
\end{proposition}

Unpacking this statement a bit more concretely, we have

\begin{corollary}[{\cite[Corollary 7.18]{PS}}]\label{cor:tilde etale}
In the situation of Definition \ref{def:tilde limit}, if $F_i$ is a sheaf of abelian groups on $X_i$, $F_j$ the preimage on $X_j$ for $j\geq i$, and $F$ the preimage on $X$, then we have bijections
\[\varinjlim_j H^n(X_{j,\et},F_j)\rightarrow H^n(X_{\et},F)\] for all $n\geq 0$.
\end{corollary}


\subsection{Raynaud extensions}\label{sub:Raynaud extensions}


Any abelian variety over $\C$ of dimension $d$ can be uniformized as a quotient of $\C^d$ by a lattice. Building off of work of Tate, Raynaud gave an analogous result in the non-archimedean case. Let $A$ be an abelian variety over a non-archimedean field $K$, considered as an adic space. In this section, we recall Raynaud's uniformization of $A$ as a quotient of a semi-abelian variety $E$ with good reduction by a lattice. We refer to \cite[Section 6]{Lut} for a nice exposition of the full construction; the original paper is \cite{BL}.

We first need a bit of background on rigid and formal tori. 

\begin{notation}\label{not:tori}
The $r$-dimensional split rigid torus is $(\G_m^{an})^r$ and will generally be denoted $T$. Taking products of Example \ref{eg:formal Gm}, we get an example of a formal model $\fT$ of $T$: a formal scheme over $K^\circ$ with $\fT_\eta\cong T$. In general, we will use normal letters for spaces over $K$, and fractur letters for formal models of spaces over $K$. We emphasize that $\fT$ is not a formal torus. The one-dimensional formal torus is $\overline{\G}_{m,K^\circ}=\Spf(K^\circ\langle T,T^{-1}\rangle)$. Its generic fiber is the unit circle $\An(1,1)$, so in general we have an embedding $(\overline T)_\eta\hookrightarrow T$. We denote formal tori by $\overline T$. Spaces with straight lines over them will be the ones that live ``naturally" over $K^\circ$. As we have seen, the special fiber $\widetilde{\fT}$ is very different from the special fiber $\widetilde{T}$. In the one-dimensional case, $\widetilde{\fT}$ is an infinite chain of $\P^1$s, while $\widetilde{T}=\Spec(\widetilde K[T,T^{-1}])=\G_{m,\widetilde K}.$ Special fibers will always be denoted with a tilde on top. 
\end{notation}

Let $\overline{B}$ be a formal abelian scheme over $K^\circ$, so 
\[\overline B:=\varinjlim_{n\in \N} \widetilde B_n\] where $\widetilde B_n=\overline B\times_{K^\circ} K^\circ/\varpi^{n+1}$ is an abelian scheme over $K^\circ/\varpi^{n+1}$. Fix a strict exact sequence of formal groups over $K^\circ$ 
\begin{equation}\label{eq:formal Raynaud sequence}
1\rightarrow\overline{T}\rightarrow\overline{E}\xrightarrow{\overline\pi}\overline{B}\rightarrow 1
\end{equation}
(strict means that the induced sequence of tangent bundles is also exact, but we will not be using this notion explicitly). 

Let $E$ be the pushforward of $\overline{E}_\eta$ by the embedding $\overline T_\eta\hookrightarrow T$. We get a commutative diagram of short exact sequences of rigid groups.

	\begin{center}
		\begin{equation}\label{eq:Raynaud diagram}
		\begin{tikzcd}
			1 \arrow[r] & \overline{T}_\eta \arrow[d, hook] \arrow[r] & \overline{E}_\eta \arrow[d, hook] \arrow[r, "\overline\pi_\eta"] & \overline{B}_\eta \arrow[d,equal] \arrow[r] & 1 \\
			1 \arrow[r] & T \arrow[r] & E \arrow[r, "\pi"] & B \arrow[r] & 1
		\end{tikzcd}
		\end{equation}
	\end{center}

\begin{definition}\label{def:Raynaud extension}
The rigid analytic group variety $E$ in Diagram \ref{eq:Raynaud diagram} is a \emph{Raynaud extension}. 
\end{definition}

Just as complex abelian varieties are uniformized by taking the quotient of $\C^r$ by a lattice, rigid abelian varieties are uniformized by taking the quotient of some Raynaud extension $E$ by a lattice. To make this more precise, we define lattices in $E$. This will require a bit more information about the geometry of the above short exact sequences. 

The map $\overline\pi$ in (\ref{eq:formal Raynaud sequence}) admits local sections $\overline s: \overline{V_i} \rightarrow \overline E$ for some open cover $\{\overline V_i\}$ of $\overline B$, so we can cover $\overline E$ by formal open subschemes of the form $\overline T\times \overline V_i$. These formal opens are glued on the overlaps by maps $$\overline\phi_{ij}:\overline T\times (\overline V_i\cap \overline V_j)\rightarrow\overline T\times (\overline V_i\cap \overline V_j)$$ such that $\overline\phi_{ij}(t,v) = (\overline\psi_{ij}(v)t,v)$ for some maps $\overline\psi_{ij}:(\overline V_i\cap \overline V_j)\rightarrow \overline T$. 

The local sections of $\overline \pi$ extend to local sections of $\pi$, so $E$ can be covered by opens isomorphic to $T\times V_i$ where $V_i:=(\overline V_i)_\eta$. The gluing data is equivalent to the maps $$\psi_{ij}:V_i\cap V_j=(\overline V_i\cap \overline V_j)_\eta\rightarrow \overline T_\eta.$$ 

The valuation on $K$ gives a group homomorphism $\ell:T(K)\cong(K^\times)^r\rightarrow \R^r$ sending the point $(x_1,\dots,x_r)$ to $(v(x_1),\dots,v(x_r))=(-\log(|x_1|),\dots,-\log(|x_r|))$. We extend the map $\ell:T(K)\rightarrow \R^r$ to a map $\ell:E(K)\rightarrow \R^r$ as follows. Locally on $K$-points, this is the map $T\times V_i\xrightarrow{pr_1} T\xrightarrow{\ell} \R^r$. On intersections the gluing maps come from the $\overline T_\eta$-action on $T$. The map $\ell$ is invariant for this action as $\ell$ sends $\overline T_\eta$ to the origin in $\R$. We can therefore can glue the local maps together to get a group homomorphism $\ell: E(K)\rightarrow \R^r$ as desired.

\begin{definition}\label{def:rigid lattice}
A \emph{lattice} in $E$ is a subgroup $M\subset E(K)$ such that $\ell$ restricts to a bijection from $M$ to a lattice $\Lambda$ in $\R^r$. The \emph{rank} of $M$ is the rank of $\Lambda$, $M$ has \emph{full rank} if it has rank $r$. 
\end{definition}

\begin{theorem}[{\cite[Theorem 6.4.8]{Lut}}]\label{thm:Raynaud uniformization}
Let $A$ be an abelian variety over $K$. After a suitable finite separable extension $L/K$ there is a unique Raynaud extension $E$ over $L$ and full rank lattice $M$ in $E$ such that $A_L\cong E/M$.
\end{theorem}

In the category of rigid spaces, $E$ is topologically simply connected (See \cite[Example 7.4.9(1)]{FvdP} or \cite[Theorem 1.2]{BL}), so this result is as strong as possible. 

More generally, define an abeloid variety to be a proper, smooth, connected, rigid group variety. This is the non-archimedean analogue of a complex torus.

\begin{theorem}[{\cite[Theorem 7.6.4]{Lut}}]\label{thm:Raynaud abeloid}
Any abeloid variety over $K$ admits a Raynaud uniformization after a suitable finite separable extension of $K$.
\end{theorem}

\begin{example}\label{eg:Tate curve}
Fix an element $q\in K^\times$ with $|q|=c<1$. The quotient $\G_{m,K}^{an}/q^\Z$ is the analytification of an elliptic curve with bad reduction, called a Tate curve. Extending the base field if needed, we assume that $q$ has a square root in $K$. We can then glue the annuli $\An(c^{1/2},1)$ and $\An(c,c^{1/2})$ along their boundary circles. The inner circles $\An(c/2,c/2)$ are glued via the identity, while the outer circles are glued along the map $\An(1,1)\rightarrow \An(c,c)$ given on points by multiplication by $q$. 

This is a formal analytic cover, so it gives rise to a formal model of the Tate curve. More generally, we get a formal model from any decomposition of the interval $[c,1]\subset \R$ into closed intervals with endpoints in the value group of $K$. See Figure \ref{fig:Tate model} for a picture of two such models.

\begin{figure}

\tikzset{every picture/.style={line width=0.75pt}} 

\begin{tikzpicture}[x=0.75pt,y=0.75pt,yscale=-.9,xscale=.9]

\draw  [color={rgb, 255:red, 208; green, 2; blue, 27 }  ,draw opacity=1 ][fill={rgb, 255:red, 255; green, 165; blue, 0 }  ,fill opacity=0.91 ][line width=3]  (204.2,635.17) .. controls (204.2,575.73) and (252.39,527.55) .. (311.83,527.55) .. controls (371.27,527.55) and (419.45,575.73) .. (419.45,635.17) .. controls (419.45,694.61) and (371.27,742.8) .. (311.83,742.8) .. controls (252.39,742.8) and (204.2,694.61) .. (204.2,635.17) -- cycle ;
\draw  [color={rgb, 255:red, 248; green, 231; blue, 28 }  ,draw opacity=1 ][fill={rgb, 255:red, 126; green, 211; blue, 33 }  ,fill opacity=0.91 ][line width=3]  (243.54,635.17) .. controls (243.54,597.46) and (274.11,566.88) .. (311.83,566.88) .. controls (349.54,566.88) and (380.12,597.46) .. (380.12,635.17) .. controls (380.12,672.89) and (349.54,703.46) .. (311.83,703.46) .. controls (274.11,703.46) and (243.54,672.89) .. (243.54,635.17) -- cycle ;
\draw    (437.5,628.69) -- (470.5,628.69) ;
\draw [shift={(473.5,628.69)}, rotate = 180] [fill={rgb, 255:red, 0; green, 0; blue, 0 }  ][line width=0.08]  [draw opacity=0] (12.5,-6.01) -- (0,0) -- (12.5,6.01) -- (8.3,0) -- cycle    ;
\draw  [color={rgb, 255:red, 74; green, 144; blue, 226 }  ,draw opacity=1 ][fill={rgb, 255:red, 255; green, 255; blue, 255 }  ,fill opacity=1 ][line width=3]  (286.83,635.17) .. controls (286.83,621.36) and (298.02,610.17) .. (311.83,610.17) .. controls (325.64,610.17) and (336.83,621.36) .. (336.83,635.17) .. controls (336.83,648.98) and (325.64,660.17) .. (311.83,660.17) .. controls (298.02,660.17) and (286.83,648.98) .. (286.83,635.17) -- cycle ;
\draw [color={rgb, 255:red, 248; green, 231; blue, 28 }  ,draw opacity=1 ][line width=3]    (491.5,643.69) .. controls (529.5,589.69) and (611.5,584.69) .. (650.5,644.69) ;
\draw [color={rgb, 255:red, 144; green, 19; blue, 254 }  ,draw opacity=1 ][line width=3]    (493.5,612.69) .. controls (521.5,688.69) and (621.5,688.69) .. (647.5,613.69) ;
\draw  [color={rgb, 255:red, 0; green, 0; blue, 0 }  ,draw opacity=1 ][fill={rgb, 255:red, 246; green, 160; blue, 17 }  ,fill opacity=1 ][line width=1.5]  (633.75,630.69) .. controls (633.75,627.93) and (635.99,625.69) .. (638.75,625.69) .. controls (641.51,625.69) and (643.75,627.93) .. (643.75,630.69) .. controls (643.75,633.45) and (641.51,635.69) .. (638.75,635.69) .. controls (635.99,635.69) and (633.75,633.45) .. (633.75,630.69) -- cycle ;
\draw  [color={rgb, 255:red, 0; green, 0; blue, 0 }  ,draw opacity=1 ][fill={rgb, 255:red, 126; green, 211; blue, 33 }  ,fill opacity=1 ][line width=1.5]  (497.5,631.19) .. controls (497.5,627.87) and (500.19,625.19) .. (503.5,625.19) .. controls (506.81,625.19) and (509.5,627.87) .. (509.5,631.19) .. controls (509.5,634.5) and (506.81,637.19) .. (503.5,637.19) .. controls (500.19,637.19) and (497.5,634.5) .. (497.5,631.19) -- cycle ;
\draw  [color={rgb, 255:red, 208; green, 2; blue, 27 }  ,draw opacity=1 ][fill={rgb, 255:red, 255; green, 165; blue, 0 }  ,fill opacity=0.91 ][line width=3]  (202.2,334.17) .. controls (202.2,274.73) and (250.39,226.55) .. (309.83,226.55) .. controls (369.27,226.55) and (417.45,274.73) .. (417.45,334.17) .. controls (417.45,393.61) and (369.27,441.8) .. (309.83,441.8) .. controls (250.39,441.8) and (202.2,393.61) .. (202.2,334.17) -- cycle ;
\draw  [color={rgb, 255:red, 248; green, 231; blue, 28 }  ,draw opacity=1 ][fill={rgb, 255:red, 126; green, 211; blue, 33 }  ,fill opacity=0.91 ][line width=3]  (241.54,334.17) .. controls (241.54,296.46) and (272.11,265.88) .. (309.83,265.88) .. controls (347.54,265.88) and (378.12,296.46) .. (378.12,334.17) .. controls (378.12,371.89) and (347.54,402.46) .. (309.83,402.46) .. controls (272.11,402.46) and (241.54,371.89) .. (241.54,334.17) -- cycle ;
\draw  [color={rgb, 255:red, 74; green, 144; blue, 226 }  ,draw opacity=1 ][fill={rgb, 255:red, 255; green, 255; blue, 255 }  ,fill opacity=1 ][line width=3]  (284.83,334.17) .. controls (284.83,320.36) and (296.02,309.17) .. (309.83,309.17) .. controls (323.64,309.17) and (334.83,320.36) .. (334.83,334.17) .. controls (334.83,347.98) and (323.64,359.17) .. (309.83,359.17) .. controls (296.02,359.17) and (284.83,347.98) .. (284.83,334.17) -- cycle ;

\draw  [line width=1.5]  (273.46,334.17) .. controls (273.46,314.08) and (289.74,297.8) .. (309.83,297.8) .. controls (329.92,297.8) and (346.2,314.08) .. (346.2,334.17) .. controls (346.2,354.26) and (329.92,370.54) .. (309.83,370.54) .. controls (289.74,370.54) and (273.46,354.26) .. (273.46,334.17) -- cycle ;
\draw  [line width=1.5]  (258.49,334.17) .. controls (258.49,305.82) and (281.47,282.83) .. (309.83,282.83) .. controls (338.18,282.83) and (361.17,305.82) .. (361.17,334.17) .. controls (361.17,362.53) and (338.18,385.51) .. (309.83,385.51) .. controls (281.47,385.51) and (258.49,362.53) .. (258.49,334.17) -- cycle ;
\draw  [line width=1.5]  (228.75,334.17) .. controls (228.75,289.39) and (265.05,253.09) .. (309.83,253.09) .. controls (354.61,253.09) and (390.91,289.39) .. (390.91,334.17) .. controls (390.91,378.95) and (354.61,415.25) .. (309.83,415.25) .. controls (265.05,415.25) and (228.75,378.95) .. (228.75,334.17) -- cycle ;
\draw  [line width=1.5]  (212.56,334.17) .. controls (212.56,280.45) and (256.11,236.91) .. (309.83,236.91) .. controls (363.55,236.91) and (407.09,280.45) .. (407.09,334.17) .. controls (407.09,387.89) and (363.55,431.44) .. (309.83,431.44) .. controls (256.11,431.44) and (212.56,387.89) .. (212.56,334.17) -- cycle ;
\draw    (437.5,327.69) -- (470.5,327.69) ;
\draw [shift={(473.5,327.69)}, rotate = 180] [fill={rgb, 255:red, 0; green, 0; blue, 0 }  ][line width=0.08]  [draw opacity=0] (12.5,-6.01) -- (0,0) -- (12.5,6.01) -- (8.3,0) -- cycle    ;
\draw [color={rgb, 255:red, 248; green, 231; blue, 28 }  ,draw opacity=1 ][line width=2.25]    (510.5,266.75) -- (630,266.75) ;
\draw [line width=2.25]    (588.5,246.69) -- (640.5,347.75) ;
\draw [line width=2.25]    (589,407.25) -- (640.5,304.69) ;
\draw [color={rgb, 255:red, 144; green, 19; blue, 254 }  ,draw opacity=1 ][line width=2.25]    (508.5,385.69) -- (629.5,385.75) ;
\draw [line width=2.25]    (548.5,408.19) -- (498.5,306.25) ;
\draw [line width=2.25]    (499.5,347.19) -- (550,246.25) ;
\draw  [color={rgb, 255:red, 0; green, 0; blue, 0 }  ,draw opacity=1 ][fill={rgb, 255:red, 126; green, 211; blue, 33 }  ,fill opacity=1 ][line width=1.5]  (533,386.19) .. controls (533,382.87) and (535.69,380.19) .. (539,380.19) .. controls (542.31,380.19) and (545,382.87) .. (545,386.19) .. controls (545,389.5) and (542.31,392.19) .. (539,392.19) .. controls (535.69,392.19) and (533,389.5) .. (533,386.19) -- cycle ;
\draw  [color={rgb, 255:red, 0; green, 0; blue, 0 }  ,draw opacity=1 ][fill={rgb, 255:red, 126; green, 211; blue, 33 }  ,fill opacity=1 ][line width=1.5]  (534,268.69) .. controls (534,265.37) and (536.69,262.69) .. (540,262.69) .. controls (543.31,262.69) and (546,265.37) .. (546,268.69) .. controls (546,272) and (543.31,274.69) .. (540,274.69) .. controls (536.69,274.69) and (534,272) .. (534,268.69) -- cycle ;
\draw  [color={rgb, 255:red, 0; green, 0; blue, 0 }  ,draw opacity=1 ][fill={rgb, 255:red, 126; green, 211; blue, 33 }  ,fill opacity=1 ][line width=1.5]  (503.5,326.19) .. controls (503.5,322.87) and (506.19,320.19) .. (509.5,320.19) .. controls (512.81,320.19) and (515.5,322.87) .. (515.5,326.19) .. controls (515.5,329.5) and (512.81,332.19) .. (509.5,332.19) .. controls (506.19,332.19) and (503.5,329.5) .. (503.5,326.19) -- cycle ;
\draw  [color={rgb, 255:red, 0; green, 0; blue, 0 }  ,draw opacity=1 ][fill={rgb, 255:red, 246; green, 160; blue, 17 }  ,fill opacity=1 ][line width=1.5]  (594.75,268.19) .. controls (594.75,265.43) and (596.99,263.19) .. (599.75,263.19) .. controls (602.51,263.19) and (604.75,265.43) .. (604.75,268.19) .. controls (604.75,270.95) and (602.51,273.19) .. (599.75,273.19) .. controls (596.99,273.19) and (594.75,270.95) .. (594.75,268.19) -- cycle ;
\draw  [color={rgb, 255:red, 0; green, 0; blue, 0 }  ,draw opacity=1 ][fill={rgb, 255:red, 246; green, 160; blue, 17 }  ,fill opacity=1 ][line width=1.5]  (624.25,326.69) .. controls (624.25,323.93) and (626.49,321.69) .. (629.25,321.69) .. controls (632.01,321.69) and (634.25,323.93) .. (634.25,326.69) .. controls (634.25,329.45) and (632.01,331.69) .. (629.25,331.69) .. controls (626.49,331.69) and (624.25,329.45) .. (624.25,326.69) -- cycle ;
\draw  [color={rgb, 255:red, 0; green, 0; blue, 0 }  ,draw opacity=1 ][fill={rgb, 255:red, 246; green, 160; blue, 17 }  ,fill opacity=1 ][line width=1.5]  (594.25,386.19) .. controls (594.25,383.43) and (596.49,381.19) .. (599.25,381.19) .. controls (602.01,381.19) and (604.25,383.43) .. (604.25,386.19) .. controls (604.25,388.95) and (602.01,391.19) .. (599.25,391.19) .. controls (596.49,391.19) and (594.25,388.95) .. (594.25,386.19) -- cycle ;

\end{tikzpicture}
\caption[Two formal analytic covers of a Tate curve]{Gluing the red and blue circles gives a formal analytic cover of the Tate curve. The special fiber of the corresponding formal model is two $\P^1$s intersecting twice. Refining the formal analytic cover corresponds to blowing up the special fiber.}\label{fig:Tate model}
\end{figure}
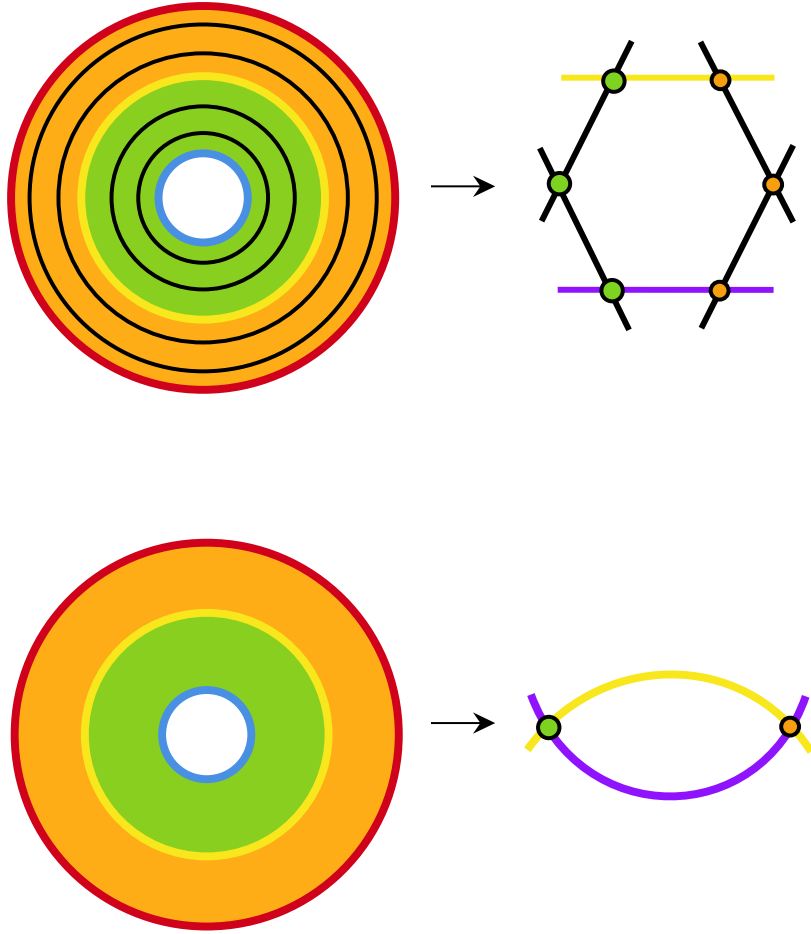

\end{example}

\subsection{Background on perfectoid spaces}\label{sec:perfectoid background}


The original reference for this material is \cite{PS}. Other good introductions include \cite{AWS_book} and \cite{Bhatt_notes}. Fix a prime $p$. 

\begin{definition}\label{def:perfectoid field}
Let $K$ be a non-archimedean field with residue characteristic $p$, complete with respect to the norm $|\cdot|$. Fix a \emph{pseudo-uniformizer} $\varpi\in K$ with $|p|<|\varpi|<|1|$. The field $K$ is \emph{perfectoid} if $|\cdot|$ is not discrete and the Frobenius map on $K^\circ/\varpi$ is surjective. 
\end{definition}

\begin{example}\label{eg:pth roots perfectoid field}
Examples of perfectoid fields include $\widehat{\Q_p(p^{1/p^\infty})}$ completed with respect to the $p$-adic valuation and $\widehat{\F_p((t^{1/p^\infty}))}$ completed with respect to the $t$-adic valuation. Any algebraically closed non-archimedean field is perfectoid. In characteristic $p$, perfectoid fields are exactly the non-discrete, perfect fields.
\end{example}

Given a perfectoid field, the \emph{tilt} is defined as $K^\flat:=\varprojlim_{x\mapsto x^p} K$, so elements of $K^\flat$ are sequences $(x_0,x_1,\dots)$ of elements of $K$ with $x_n^p=x_{n-1}$. This is a characteristic $p$ perfectoid field, where the multiplication law is given by component-wise multiplication, and the addition law is given by $(x_n)+(y_n)=(z_n)$, where \[z_n=\displaystyle\lim_{m\rightarrow\infty} (x_{m+n}+y_{m+n})^{p^m}.\] As perfectoid fields in characteristic $p$ are perfect, we have $K^\flat=K$ for such fields.

There is a projection map $\sharp:K^\flat\rightarrow K$ sending $(x_n)\mapsto x_0$: this is multiplicative but not additive. Its image is the elements of $K$ which have $p^n$th roots for all $n$. The absolute value on $K^\flat$ is given by $|x|_{K^\flat}=|x^\sharp|_{K}.$ The fields $K$ and $K^\flat$ have the same value group. We can therefore choose our pseudo-uniformizer $\varpi$ of $K$ to be in the image of $\sharp$ and choose our pseudo-uniformizer of $K^\flat$ to be an element $\varpi^\flat$ of $K^\flat$ with $(\varpi^\flat)^\sharp=\varpi$. A key relation between $K$ and $K^\flat$ is the isomorphism 
\begin{equation}\label{eq:field special fiber} 
K^\circ/\varpi\cong K^{\flat\circ}/\varpi^\flat
\end{equation}

\begin{example}\label{eg:pth roots field tilt}
The tilt of $K=\widehat{\Q_p(p^{1/p^\infty})}$ is isomorphic to $\widehat{\F_p((t^{1/p^\infty}))}$, where the element $(p,p^{1/p},\dots)\in K^\flat$ is identified with the pseudo-uniformizer $t$. Then $t^\sharp=p$, and we have \[K^{\flat\circ}/\varpi^\flat\cong \widehat{\F_p[[t^{1/p^\infty}]]}/t\cong \F_p[t^{\frac{1}{p^\infty}}]/t\cong K^\circ/\varpi\] as expected.
\end{example}

For the purposes of this paper, we will not need the full generality of perfectoid rings and adic spaces. Our goal is to work with explicit limits of well-behaved spaces coming from rigid geometry. We can therefore fix a perfectoid field $K$ and restrict our attention to the following nice situation.

\begin{definition}\label{def:perfectoid space}
A \emph{perfectoid $K$-algebra} is a complete $K$-algebra $R$ such that $R^\circ$ is open and bounded, and the Frobenius map on $R^\circ/\varpi$ is surjective. An \emph{affinoid perfectoid} space is an adic space $\Spa(R,R^+)$ with $R$ perfectoid. A \emph{perfectoid space} over $K$ is an adic space which is locally affinoid perfectoid.  
\end{definition}

Perfectoid $K$-algebras and affinoid perfectoid spaces can again be tilted, with 
\[(R,R^+)\mapsto (R^\flat,R^{\flat +}):=(\varprojlim_{x\mapsto x^p} R,\varprojlim_{x\mapsto x^p} R^+).\] The map $\sharp:R^\flat\rightarrow R$ and topology on $R^\flat$ are defined analogously to the case of fields, and there is again a crucial isomorphism $R^+/\varpi\cong R^{\flat +}/\varpi^\flat$. This tilting process glues, so perfectoid spaces in general tilt.

\begin{example}\label{eg:perfectoid tate algebra}
The \emph{perfectoid Tate algebra} $K\langle T_1^{1/p^\infty},\dots,T_n^{1/p^\infty}\rangle$ is a perfectoid $K$-algebra with power bounded elements $K^\circ\langle T_1^{1/p^\infty},\dots,T_n^{1/p^\infty}\rangle$ and special fiber $K^\circ/\varpi[T_1^{1/p^\infty},\dots,T_n^{1/p^\infty}].$ The affinoid perfectoid space $\Spa(K\langle T_1^{1/p^\infty},\dots,T_n^{1/p^\infty}\rangle,K^\circ\langle T_1^{1/p^\infty},\dots,T_n^{1/p^\infty}\rangle)$ is called the $n$-dimensional perfectoid closed disk. The tilt of the perfectoid Tate algebra over $K$ is the perfectoid Tate algebra over $K^\flat$.
\end{example} 

\begin{theorem}[{\cite{PS}}]\label{thm:tilting}
Tilting induces an equivalence between the following pairs of categories:
\begin{enumerate}
\item Perfectoid $K$-algebras and perfectoid $K^\flat$-algebras;
\item Affinoid perfectoid spaces over $K$ and $K^\flat$; 
\item Finite \'etale algebras over a perfectoid $K$-algebra $R$ and finite \'etale algebras over $R^\flat$;
\item Perfectoid spaces over $\Spa(K,K^\circ)$ and $\Spa(K^\flat,K^{\flat\circ})$;
\item (Finite) \'etale covers of a perfectoid space $X$ and of its tilt $X^\flat$.
\end{enumerate}
\end{theorem}


\subsection{Constructing perfectoid covers via $F$-towers}\label{sub:pF towers}

Our overall goal is to transfer information about the \'etale cohomology of varieties from characteristic $p$ to characteristic 0. Theorem \ref{thm:tilting} lets us make this transfer for perfectoid spaces, so we need to relate the cohomology of varieties to that of perfectoid spaces. Proposition \ref{prop:tilde etale} says that tilde-limits act like inverse limits on \'etale topoi, and therefore \'etale cohomology. Following Scholze, our approach is to construct inverse systems of varieties with a perfectoid tilde-limit. In \cite{PS}, Scholze does this for inverse systems built by iterating Frobenius lifts on toric varieties. As we typically don't have Frobenius lifts for abelian varieties, we need a more general setup. Roughly speaking, we need to make sure that the rings at the top of the tower will have $p^n$th power roots mod $\varpi$ for all $n$. At each level of the tower, we add more $p$th roots by requiring that our maps factor through relative Frobenius. 

The following formalism is modeled off of Scholze's construction of perfectoid covers of Shimura varieties in \cite[III.2]{Torsion}.

\begin{definition}\label{def:pF towers}
Let $X$ be a rigid analytic space over $K$. An $F$-tower over $X$ consists of an inverse system of formal schemes $\{\fX_n\}_{n\geq 1}$ satisfying the following conditions:
\begin{enumerate}
\item The generic fiber of $\fX_1$ is $X$,
\item The formal schemes $\fX_n$ are all flat over $K^\circ$,
\item\label{pf3} The transition morphisms $\fX_n\rightarrow\fX_{n-1}$ are affine and on the mod $\varpi$ special fiber factor through the relative Frobenius map \[F_{\widetilde{\fX}_n/(K^\circ/\varpi)}:\widetilde{\fX}_n\rightarrow\widetilde{\fX}_n^{(p)}.\] To ease notation, we write $\widetilde F_n$ for this relative Frobenius map from now on.
\end{enumerate}

We write $\{X_n\}$ for the adic generic fiber of the $F$-tower. If the formal schemes are all admissible (as in Definition \ref{def:formal scheme}), this is an inverse system of rigid spaces over $X$, all considered as adic spaces.
\end{definition}

\begin{proposition}\label{prop:pF tower}
Given an $F$-tower $\{\fX_n\}$ of a rigid space $X$, there is a unique perfectoid space $X_\infty\sim\varprojlim X_n$. We have $X_\infty=(\varprojlim \fX_n)_\eta$.
\end{proposition}

\begin{proof}
The inverse limit $\fX_\infty:=\varprojlim \fX_n$ exists in the category of formal schemes by Lemma \ref{lemma:inverse limit formal schemes}, and the generic fiber $X_\infty$ is a tilde-limit for $\{X_n\}$ by Lemma \ref{lemma:tilde generic fiber}. 

We check perfectoidness of $X_\infty$ locally. It is enough to cover $\fX_\infty$ by affine formal schemes $\Spf(R_\infty^+)$ for flat $K^\circ$ algebras $R_\infty^+$ such that the absolute Frobenius map $R_\infty^+/\varpi \rightarrow R_\infty^+/\varpi$ is surjective. Take a formal affine open $\Spf(R^+_1)\subset\fX_1$. As the transition maps of our $F$-tower are affine and all the formal schemes are flat, the preimage in $\fX_n$ (respectively, $\fX_\infty$) is a flat formal affine open $\Spf(R^+_n)$ (respectively, $\Spf(R^+_\infty)$). We have the following diagram defining relative Frobenius in terms of absolute Frobenius
\begin{equation}\label{eq:rel Frob}
\begin{tikzcd}
R^+_\infty/\varpi  & & \\
& (R^+_\infty/\varpi)^{(p)} \arrow[ul, "F_{rel}"]  & R^+_\infty/\varpi \arrow[l]  \arrow[ull, bend right=20, "\Frob"] \\
& K^\circ/\varpi \arrow[u]\arrow[uul, bend left=20] & \arrow[l, "\Frob"] \arrow[u]  K^\circ/\varpi 
\end{tikzcd}
\end{equation}
where the square is cartesian.

By condition \ref{def:pF towers}(\ref{pf3}), we have a diagram
\begin{center}
\begin{tikzcd}
&  & (R^+_3/\varpi)^{(p)} \arrow[dl, "\widetilde F_3"] & &  (R^+_2/\varpi)^{(p)} \arrow[dl, "\widetilde F_2"] & \\
\cdots & R^+_3/\varpi \arrow[l] & & R^+_2/\varpi \arrow[ll] \arrow[ul, "\widetilde V_3"] & & R^+_1/\varpi.\arrow[ll] \arrow[ul, "\widetilde V_2"] 
\end{tikzcd}
\end{center}
where the maps $\widetilde F_n$ are relative Frobenius maps and the direct limit of the $R_n^+/\varpi$ is $R_\infty^+/\varpi$. By diagram (\ref{eq:rel Frob}), any element of $R_n^+/\varpi$ in the image of $\widetilde F_n$ is also in the image of the absolute Frobenius map $R_n^+/\varpi\rightarrow R_n^+/\varpi$, so it has a $p$th root. Every element of $R_\infty^+/\varpi$ comes from finite level, so they all factor through some relative Frobenius map. They therefore all have $p$th roots and so absolute Frobenius is surjective as desired.

By \cite[Proposition 2.4.5]{SW}, any inverse system has at most one perfectoid tilde-limit (even though it could have multiple tilde-limits), so $X_\infty$ is the unique perfectoid tilde-limit of the $F$-tower $\{\fX_n\}$.
\end{proof}

\begin{example}\label{eg:pF P1}
Let $\P^1_K$ be the analytification of the projective line. We get a formal model by gluing $\Spf(K^\circ\langle T\rangle)$ to $\Spf(K^\circ\langle T^{-1}\rangle)$ along the boundary circles. The map of formal models sending $T\mapsto T^p$ and $T^{-1}\mapsto T^{-p}$ reduces to relative Frobenius mod $\varpi$ and satisfies all the conditions of Definition \ref{def:pF towers}. We can therefore iterate this map to get an $F$-tower of $\P^1_K$. This corresponds to gluing two perfectoid closed unit disks $\Spa(K\langle T^{1/p^\infty}\rangle,K^\circ\langle T^{1/p^\infty}\rangle)$ and $\Spa(K\langle T^{-1/p^\infty}\rangle,K^\circ\langle T^{-1/p^\infty}\rangle)$ (as in Example \ref{eg:perfectoid tate algebra}) along the perfectoid unit circle. On points, this is iterating the map $(z_0:z_1)\mapsto (z_0^p:z_1^p)$. This construction works more generally for toric varieties, as described in \cite[Section 8]{PS}.
\end{example}

\begin{notation}
In this paper, we will be applying this construction to construct $F$-towers of various rigid analytic groups $G$. In this case, on the generic fiber our inverse system will always come from iterating the multiplication map $[p]:G\rightarrow G$. We therefore denote the $n$th transition map of the tower as $[\fp]_n:\fG_{n+1} \rightarrow \fG_n$. Note however that the formal models will generally not be group schemes.
\end{notation}

We can apply this formalism when $G$ is an abelian variety with good reduction.

\begin{example}[{\cite[Lemme A.16]{P-S}}]\label{eg:good reduction}
Let $B$ be (the analytification of) an abelian variety with good reduction over $K$, so there is a formal abelian scheme $\overline B/K^\circ$ with rigid generic fiber $B$. Iterating the multiplication map $\overline{[p]}:\overline B\rightarrow \overline B$ gives an $F$-tower. In particular, the mod $\varpi$ special fiber is an abelian scheme in characteristic $p$, so the map $\widetilde{[p]}$ factors as Frobenius and Verschiebung. We therefore get a perfectoid cover $B_\infty\sim\varprojlim_{[p]} B$.
\end{example}

We note that this construction is functorial, and therefore that $G_\infty$ is a perfectoid group.

\begin{proposition}[{\cite[Lemma 2.10]{AWS}}]\label{prop:G_inf a group}
Let $G$ be a rigid commutative group with a perfectoid cover $G_\infty\sim\varprojlim_{[p]} G$. Then there is a unique way to endow $G_\infty$ with the structure of a group object in the category of perfectoid spaces such that all projections $G_\infty\rightarrow G$ are group homomorphisms. Given another rigid group $H$ with perfectoid tilde-limit $H_\infty\sim\varprojlim_{[p]}H$ and a group homomorphism $H\rightarrow G$, there is a unique group homomorphism $H_\infty\rightarrow G_\infty$ commuting with all projection maps.
\end{proposition}

\begin{proof}
This follows directly from functoriality of the construction and uniqueness of perfectoid tilde-limits.
\end{proof}

To transfer information at finite level between characteristics 0 and $p$, we need to understand how $F$-towers interact with tilting. By Theorem \ref{thm:tilting}, we have the following:

\begin{lemma}\label{cor:pF special fiber unique}
Let $\{\fX_n\}$ be an $F$-tower for a rigid space $X$, let $X_\infty$ be the perfectoid cover of $X$ given by Proposition \ref{prop:pF tower}. Then $X_\infty$ can be reconstructed from the mod $\varpi$ special fiber $\widetilde{\fX}_\infty$ of the $F$-tower. 
\end{lemma}

\begin{proof}
Part of the tilting equivalence \cite[Theorem 5.2]{PS} gives an equivalence of categories between perfectoid $K$-algebras and perfectoid $K^{\circ a}/\varpi$-algebras, where the $a$ means that we are working in the world of almost mathematics as in \cite{GR}. The actual $K^\circ/\varpi$-algebras used to build $\widetilde{\fX}_\infty$ give $K^{\circ a}/\varpi$-algebras, which lift to $K$-algebras under the equivalence of categories. This process glues by \cite[Proposition 6.17]{PS}, so we get the desired reconstruction.
\end{proof}

This implies that different abelian varieties will have the same perfectoid cover, a fact which we will exploit in Section \ref{sec:line bundles} when we tilt line bundles. This is explored in more detail in \cite{Heuer_uniform}.

\begin{example}\label{eg:good reduction same fiber}
Let $B$ and $B'$ be abelian varieties with good reduction over $K$, let $\overline B$ and $\overline B'$ be the corresponding formal abelian schemes over $K^\circ$, let $B_\infty$ and $B'_\infty$ be the perfectoid covers constructed in Example \ref{eg:good reduction}. If $\overline B\times_{K^\circ}K^\circ/\varpi\simeq\overline B'\times_{K^\circ}K^\circ/\varpi$, then we have $B_\infty\simeq B'_\infty$. 
\end{example}

This allows us to understand the tilt of a perfectoid cover of $X$ by constructing a suitable $F$-tower over $K^\flat$.

\begin{lemma}\label{lemma:pF tilt}
Let $\{\fX_n\}$ be an $F$-tower for a rigid space $X$, let  $\{\widetilde\fX_n\}$ denote the mod $\varpi$ special fiber of the tower, let $X_\infty$ be the perfectoid cover of $X$ given by Proposition \ref{prop:pF tower}. Let $\{\fX'_n\}$ be an $F$-tower for a rigid space $X'$ over $K^\flat$, let  $\{\widetilde\fX'_n\}$ denote the mod $\varpi^\flat$ special fiber of the tower, let $X'_\infty$ be the perfectoid cover of $X'$ given by Proposition \ref{prop:pF tower}. If the isomorphism $K^\circ/\varpi\cong K^{\flat\circ}/\varpi^\flat$ induces an isomorphism of the special fibers of the two $F$-towers, then $(X_\infty)^\flat\cong(X'_\infty)$.
\end{lemma}

\begin{proof}
This follows directly from Theorem \ref{thm:tilting} and Lemma \ref{cor:pF special fiber unique}. See \cite[Section III.2.3]{Torsion} for an example of this method in the case of Shimura varieties.
\end{proof}

\begin{question}
It is not clear how common it is for rigid spaces $X$ to admit $F$-towers. It is also not clear, given an $F$-tower over $K$, if it should be possible in general to construct an $F$-tower over $K^\flat$ with isomorphic special fiber. In this paper, we answer these questions for abelian varieties (and some other group schemes along the way). Answering these questions for more general varieties should make it possible to prove more cases of weight-monodromy. 
\end{question}

We'll need two more results later on:

\begin{lemma}[{\cite[Lemma 2.8]{AWS}}]\label{lemma:product tilde}
Let $(A_i,A_i^+)$ and $(B_i,B_i^+)$ be direct systems of affinoids over $(K,K^\circ)$ with compatible rings of definition carrying the $\varpi$-adic topology. If there are perfectoid tilde-limits $\Spa(A,A^+)\sim\varprojlim\Spa(A_i,A_i^+)$ and $\Spa(B,B^+)\sim\varprojlim\Spa(B_i,B_i^+)$, then \[\Spa(A,A^+)\times_{\Spa(K,K^\circ)} \Spa(B,B^+)\sim \varprojlim (\Spa(A_i,A_i^+)\times_{\Spa(K,K^\circ)} \Spa(B_i,B_i^+))\] is also a perfectoid tilde-limit.
\end{lemma}

\begin{proposition}[{\cite[Corollary 2.3.5]{Shen}}]\label{prop:Galois descent}
Let $L/K$ be a finite Galois extension of perfectoid fields with Galois group $G$. Let $X_i$ be a filtered inverse system of quasi-compact adic spaces over $\Spa(K, K^\circ)$ with finite transition maps, and let $X_{L,i}$ be the base change of $X_i$ to $\Spa(L, L^\circ)$. Assume that there is a perfectoid space $X_{L,\infty}$ over $\Spa(L, L^\circ)$ such that $X_{L, \infty}\sim \varprojlim X_{L, i}$. Then there is a perfectoid space $X_\infty$ over $\Spa(K, K^\circ)$ such that $X_\infty\sim\varprojlim X_i$.
\end{proposition}

\section{Perfectoid covers of abeloids}\label{sec:perfectoid covers}

Let $K$ be a perfectoid field with value group $\Gamma=|K^\times|$ such that $\Gamma\subset \Q$. For example, this holds whenever $K$ is a perfectoid subfield of $\C_p$. In this section, we give an alternate proof of the main theorem of \cite{AWS}, constructing for any abeloid $A$ over $K$ a perfectoid space $A_\infty\sim\varprojlim_{[p]} A$. We give a sketch of the proof here.

After a finite extension of $K$, we can uniformize $A$ as a quotient $E/M$ for $E$ a Raynaud extension by Theorem \ref{thm:Raynaud uniformization}. By Proposition \ref{prop:Galois descent}, we can construct our perfectoid cover after this extension, then descend to get the desired cover. We construct an $F$-tower as in Definition \ref{def:pF towers} for the torus part $T$ of $E$, then use this to construct an $F$-tower for $E$, giving a perfectoid space $E_\infty\sim\varprojlim_{[p]} E$. We choose formal models that all have compatible actions of $M$, which allows us to take the quotient by this action to get formal models of $A$, which we show form an $F$-tower for $A$.

To construct our models for a split torus $T$ of rank $r$, we use formal analytic covers made up of products of annuli, extending Example \ref{eg:formal Gm}. The condition $\Gamma\subset \Q$ ensures that by choosing sufficiently small annuli, we can extend the action of $M$ on $T$ to the formal analytic cover. To move from models of $T$ to models of $E$, we take the pushout as in Diagram \ref{eq:Raynaud diagram}. When we do this, the lattice action is retained, allowing us to quotient out by the action of $M$ to get the desired models of $A$. 


We remark that one can use techniques from tropical geometry - the theory of polytopal domains in rigid tori as described in \cite{Gub} - to get more general models of tori (and therefore abelian varieties). See \cite[Chapter 3]{Thesis} for more references and details on this construction. One can combine the theory of polytope domains with the proof of \cite{AWS} to construct perfectoid covers of abelian varieties over any perfectoid field $K$. This will not result in formal models of the $A$, which will be necessary for Section \ref{sec:tilting abeloids}, so we do not go into more detail here. 

\subsection{Formal models of tori}\label{sub:torus models}

Let $T$ be a split rigid analytic torus of rank $r$ over $K$, so $T\cong (\G_{m,an})^r$. Let $M$ be a lattice in $T$ with rank $r$ as in Definition \ref{def:rigid lattice}: that is, the log map $\ell:T(K)\rightarrow \R^r$ sends $M$ bijectively to a lattice $\Lambda$. Fix coordinates $x_1,\dots,x_r$ on $T=\mathbb G_{m,an}^r$ and $u_1,\dots,u_r$ on $\R^r$ so that the log map sends $x_i$ to $u_i$.

In this subsection, we construct an $F$-tower $\{\fT_n\}$  of $T$ such that the translation action of $M$ on $T$ extends to an action on each $\fT_n$, and such that these actions are compatible with the morphisms $[\fp]_n$. This will allow us to construct an $F$-tower for $T/M$, though we will wait until next section to do this in the more general setting of Raynaud extensions. 

\begin{definition}\label{def:hypercube models}
For any $y\in K^\times$, let $c:=|y|$ and $\alpha=:=\ell(y):=-\log(c)>0$. We get a formal model $\fT_\alpha$ of $\G_{m,K}^r$ by taking $r$ products of the formal model from Example \ref{eg:formal Gm}. The corresponding formal analytic cover corresponds to the decomposition of $\R^r$ into hypercubes with side length $\alpha$. 

More precisely, let $\Delta_\alpha=[0,\alpha]^r\subset \R^r$. The preimage $\ell^{-1}(\Delta_\alpha)$ is $\An(c,1)^r$. For any tuple ${\bf e}=(e_1,\dots,e_r)\in\Z^r$, the translation ${\bf e}\Delta_\alpha\subset\R^r$ has preimage $U_{{\bf e}\Delta_\alpha}:=\An(c^{e_1+1},c^{e_1})\times\cdots\times\An(c^{e_r+1},c^{e_r}).$ These poly-annuli form the formal analytic cover of $\G_{m,K}^r$ giving rise to the formal model $\fT_\alpha$.
\end{definition}

In our case, we want formal models of $T$ such that the translation morphisms $\tau_m:T\rightarrow T$ extend to the formal model for all lattice elements $m\in M$. On $\R^r$, these maps induce the translation maps $+\lambda:\R^r\rightarrow \R^r$ for $\lambda\in\Lambda$. We therefore restrict our attention to hypercube decompositions of $\R^r$ which are preserved by these translation actions.

\begin{definition}\label{def:lattice division}
Given a lattice $\Lambda\subset\R^r$, we say that $\alpha\in \R$ divides $\Lambda$ if $\Lambda\subset \alpha\Z^r$. 
\end{definition}

If $\Lambda\subset\Q^r$, we can always find $\alpha$ dividing $\Lambda$ by letting $\alpha$ be the reciprocal of the gcd of the denominators of any set of generators of $\Lambda$. See Figure \ref{fig:decomp} for a visualization of this process.

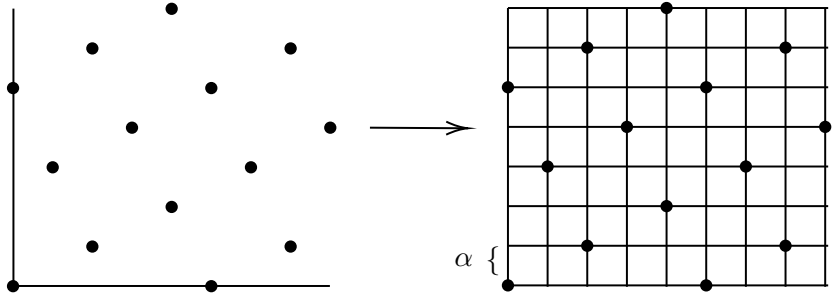
\begin{figure}

\tikzset{every picture/.style={line width=0.75pt}} 

\begin{tikzpicture}[x=0.75pt,y=0.75pt,yscale=-1,xscale=1]

\draw  [fill={rgb, 255:red, 0; green, 0; blue, 0 }  ,fill opacity=1 ] (77.41,230) .. controls (77.41,228.57) and (78.57,227.41) .. (80,227.41) .. controls (81.43,227.41) and (82.59,228.57) .. (82.59,230) .. controls (82.59,231.43) and (81.43,232.59) .. (80,232.59) .. controls (78.57,232.59) and (77.41,231.43) .. (77.41,230) -- cycle ;
\draw  [fill={rgb, 255:red, 0; green, 0; blue, 0 }  ,fill opacity=1 ] (77.41,130) .. controls (77.41,128.57) and (78.57,127.41) .. (80,127.41) .. controls (81.43,127.41) and (82.59,128.57) .. (82.59,130) .. controls (82.59,131.43) and (81.43,132.59) .. (80,132.59) .. controls (78.57,132.59) and (77.41,131.43) .. (77.41,130) -- cycle ;
\draw  [fill={rgb, 255:red, 0; green, 0; blue, 0 }  ,fill opacity=1 ] (237.41,150) .. controls (237.41,148.57) and (238.57,147.41) .. (240,147.41) .. controls (241.43,147.41) and (242.59,148.57) .. (242.59,150) .. controls (242.59,151.43) and (241.43,152.59) .. (240,152.59) .. controls (238.57,152.59) and (237.41,151.43) .. (237.41,150) -- cycle ;
\draw  [fill={rgb, 255:red, 0; green, 0; blue, 0 }  ,fill opacity=1 ] (217.41,210) .. controls (217.41,208.57) and (218.57,207.41) .. (220,207.41) .. controls (221.43,207.41) and (222.59,208.57) .. (222.59,210) .. controls (222.59,211.43) and (221.43,212.59) .. (220,212.59) .. controls (218.57,212.59) and (217.41,211.43) .. (217.41,210) -- cycle ;
\draw  [fill={rgb, 255:red, 0; green, 0; blue, 0 }  ,fill opacity=1 ] (177.41,230) .. controls (177.41,228.57) and (178.57,227.41) .. (180,227.41) .. controls (181.43,227.41) and (182.59,228.57) .. (182.59,230) .. controls (182.59,231.43) and (181.43,232.59) .. (180,232.59) .. controls (178.57,232.59) and (177.41,231.43) .. (177.41,230) -- cycle ;
\draw  [fill={rgb, 255:red, 0; green, 0; blue, 0 }  ,fill opacity=1 ] (217.41,110) .. controls (217.41,108.57) and (218.57,107.41) .. (220,107.41) .. controls (221.43,107.41) and (222.59,108.57) .. (222.59,110) .. controls (222.59,111.43) and (221.43,112.59) .. (220,112.59) .. controls (218.57,112.59) and (217.41,111.43) .. (217.41,110) -- cycle ;
\draw  [fill={rgb, 255:red, 0; green, 0; blue, 0 }  ,fill opacity=1 ] (197.41,170) .. controls (197.41,168.57) and (198.57,167.41) .. (200,167.41) .. controls (201.43,167.41) and (202.59,168.57) .. (202.59,170) .. controls (202.59,171.43) and (201.43,172.59) .. (200,172.59) .. controls (198.57,172.59) and (197.41,171.43) .. (197.41,170) -- cycle ;
\draw  [fill={rgb, 255:red, 0; green, 0; blue, 0 }  ,fill opacity=1 ] (177.41,130) .. controls (177.41,128.57) and (178.57,127.41) .. (180,127.41) .. controls (181.43,127.41) and (182.59,128.57) .. (182.59,130) .. controls (182.59,131.43) and (181.43,132.59) .. (180,132.59) .. controls (178.57,132.59) and (177.41,131.43) .. (177.41,130) -- cycle ;
\draw  [fill={rgb, 255:red, 0; green, 0; blue, 0 }  ,fill opacity=1 ] (157.41,190) .. controls (157.41,188.57) and (158.57,187.41) .. (160,187.41) .. controls (161.43,187.41) and (162.59,188.57) .. (162.59,190) .. controls (162.59,191.43) and (161.43,192.59) .. (160,192.59) .. controls (158.57,192.59) and (157.41,191.43) .. (157.41,190) -- cycle ;
\draw  [fill={rgb, 255:red, 0; green, 0; blue, 0 }  ,fill opacity=1 ] (157.41,90) .. controls (157.41,88.57) and (158.57,87.41) .. (160,87.41) .. controls (161.43,87.41) and (162.59,88.57) .. (162.59,90) .. controls (162.59,91.43) and (161.43,92.59) .. (160,92.59) .. controls (158.57,92.59) and (157.41,91.43) .. (157.41,90) -- cycle ;
\draw  [fill={rgb, 255:red, 0; green, 0; blue, 0 }  ,fill opacity=1 ] (137.41,150) .. controls (137.41,148.57) and (138.57,147.41) .. (140,147.41) .. controls (141.43,147.41) and (142.59,148.57) .. (142.59,150) .. controls (142.59,151.43) and (141.43,152.59) .. (140,152.59) .. controls (138.57,152.59) and (137.41,151.43) .. (137.41,150) -- cycle ;
\draw  [fill={rgb, 255:red, 0; green, 0; blue, 0 }  ,fill opacity=1 ] (117.41,210) .. controls (117.41,208.57) and (118.57,207.41) .. (120,207.41) .. controls (121.43,207.41) and (122.59,208.57) .. (122.59,210) .. controls (122.59,211.43) and (121.43,212.59) .. (120,212.59) .. controls (118.57,212.59) and (117.41,211.43) .. (117.41,210) -- cycle ;
\draw  [fill={rgb, 255:red, 0; green, 0; blue, 0 }  ,fill opacity=1 ] (117.41,110) .. controls (117.41,108.57) and (118.57,107.41) .. (120,107.41) .. controls (121.43,107.41) and (122.59,108.57) .. (122.59,110) .. controls (122.59,111.43) and (121.43,112.59) .. (120,112.59) .. controls (118.57,112.59) and (117.41,111.43) .. (117.41,110) -- cycle ;
\draw  [fill={rgb, 255:red, 0; green, 0; blue, 0 }  ,fill opacity=1 ] (97.41,170) .. controls (97.41,168.57) and (98.57,167.41) .. (100,167.41) .. controls (101.43,167.41) and (102.59,168.57) .. (102.59,170) .. controls (102.59,171.43) and (101.43,172.59) .. (100,172.59) .. controls (98.57,172.59) and (97.41,171.43) .. (97.41,170) -- cycle ;
\draw  [draw opacity=0] (329.6,89.6) -- (491.1,89.6) -- (491.1,230.29) -- (329.6,230.29) -- cycle ; \draw   (329.6,89.6) -- (329.6,230.29)(349.6,89.6) -- (349.6,230.29)(369.6,89.6) -- (369.6,230.29)(389.6,89.6) -- (389.6,230.29)(409.6,89.6) -- (409.6,230.29)(429.6,89.6) -- (429.6,230.29)(449.6,89.6) -- (449.6,230.29)(469.6,89.6) -- (469.6,230.29)(489.6,89.6) -- (489.6,230.29) ; \draw   (329.6,89.6) -- (491.1,89.6)(329.6,109.6) -- (491.1,109.6)(329.6,129.6) -- (491.1,129.6)(329.6,149.6) -- (491.1,149.6)(329.6,169.6) -- (491.1,169.6)(329.6,189.6) -- (491.1,189.6)(329.6,209.6) -- (491.1,209.6)(329.6,229.6) -- (491.1,229.6) ; \draw    ;
\draw  [fill={rgb, 255:red, 0; green, 0; blue, 0 }  ,fill opacity=1 ] (327.01,229.6) .. controls (327.01,228.17) and (328.17,227.01) .. (329.6,227.01) .. controls (331.03,227.01) and (332.19,228.17) .. (332.19,229.6) .. controls (332.19,231.03) and (331.03,232.19) .. (329.6,232.19) .. controls (328.17,232.19) and (327.01,231.03) .. (327.01,229.6) -- cycle ;
\draw  [fill={rgb, 255:red, 0; green, 0; blue, 0 }  ,fill opacity=1 ] (327.01,129.6) .. controls (327.01,128.17) and (328.17,127.01) .. (329.6,127.01) .. controls (331.03,127.01) and (332.19,128.17) .. (332.19,129.6) .. controls (332.19,131.03) and (331.03,132.19) .. (329.6,132.19) .. controls (328.17,132.19) and (327.01,131.03) .. (327.01,129.6) -- cycle ;
\draw  [fill={rgb, 255:red, 0; green, 0; blue, 0 }  ,fill opacity=1 ] (487.01,149.6) .. controls (487.01,148.17) and (488.17,147.01) .. (489.6,147.01) .. controls (491.03,147.01) and (492.19,148.17) .. (492.19,149.6) .. controls (492.19,151.03) and (491.03,152.19) .. (489.6,152.19) .. controls (488.17,152.19) and (487.01,151.03) .. (487.01,149.6) -- cycle ;
\draw  [fill={rgb, 255:red, 0; green, 0; blue, 0 }  ,fill opacity=1 ] (467.01,209.6) .. controls (467.01,208.17) and (468.17,207.01) .. (469.6,207.01) .. controls (471.03,207.01) and (472.19,208.17) .. (472.19,209.6) .. controls (472.19,211.03) and (471.03,212.19) .. (469.6,212.19) .. controls (468.17,212.19) and (467.01,211.03) .. (467.01,209.6) -- cycle ;
\draw  [fill={rgb, 255:red, 0; green, 0; blue, 0 }  ,fill opacity=1 ] (427.01,229.6) .. controls (427.01,228.17) and (428.17,227.01) .. (429.6,227.01) .. controls (431.03,227.01) and (432.19,228.17) .. (432.19,229.6) .. controls (432.19,231.03) and (431.03,232.19) .. (429.6,232.19) .. controls (428.17,232.19) and (427.01,231.03) .. (427.01,229.6) -- cycle ;
\draw  [fill={rgb, 255:red, 0; green, 0; blue, 0 }  ,fill opacity=1 ] (467.01,109.6) .. controls (467.01,108.17) and (468.17,107.01) .. (469.6,107.01) .. controls (471.03,107.01) and (472.19,108.17) .. (472.19,109.6) .. controls (472.19,111.03) and (471.03,112.19) .. (469.6,112.19) .. controls (468.17,112.19) and (467.01,111.03) .. (467.01,109.6) -- cycle ;
\draw  [fill={rgb, 255:red, 0; green, 0; blue, 0 }  ,fill opacity=1 ] (447.01,169.6) .. controls (447.01,168.17) and (448.17,167.01) .. (449.6,167.01) .. controls (451.03,167.01) and (452.19,168.17) .. (452.19,169.6) .. controls (452.19,171.03) and (451.03,172.19) .. (449.6,172.19) .. controls (448.17,172.19) and (447.01,171.03) .. (447.01,169.6) -- cycle ;
\draw  [fill={rgb, 255:red, 0; green, 0; blue, 0 }  ,fill opacity=1 ] (427.01,129.6) .. controls (427.01,128.17) and (428.17,127.01) .. (429.6,127.01) .. controls (431.03,127.01) and (432.19,128.17) .. (432.19,129.6) .. controls (432.19,131.03) and (431.03,132.19) .. (429.6,132.19) .. controls (428.17,132.19) and (427.01,131.03) .. (427.01,129.6) -- cycle ;
\draw  [fill={rgb, 255:red, 0; green, 0; blue, 0 }  ,fill opacity=1 ] (407.01,189.6) .. controls (407.01,188.17) and (408.17,187.01) .. (409.6,187.01) .. controls (411.03,187.01) and (412.19,188.17) .. (412.19,189.6) .. controls (412.19,191.03) and (411.03,192.19) .. (409.6,192.19) .. controls (408.17,192.19) and (407.01,191.03) .. (407.01,189.6) -- cycle ;
\draw  [fill={rgb, 255:red, 0; green, 0; blue, 0 }  ,fill opacity=1 ] (407.01,89.6) .. controls (407.01,88.17) and (408.17,87.01) .. (409.6,87.01) .. controls (411.03,87.01) and (412.19,88.17) .. (412.19,89.6) .. controls (412.19,91.03) and (411.03,92.19) .. (409.6,92.19) .. controls (408.17,92.19) and (407.01,91.03) .. (407.01,89.6) -- cycle ;
\draw  [fill={rgb, 255:red, 0; green, 0; blue, 0 }  ,fill opacity=1 ] (387.01,149.6) .. controls (387.01,148.17) and (388.17,147.01) .. (389.6,147.01) .. controls (391.03,147.01) and (392.19,148.17) .. (392.19,149.6) .. controls (392.19,151.03) and (391.03,152.19) .. (389.6,152.19) .. controls (388.17,152.19) and (387.01,151.03) .. (387.01,149.6) -- cycle ;
\draw  [fill={rgb, 255:red, 0; green, 0; blue, 0 }  ,fill opacity=1 ] (367.01,209.6) .. controls (367.01,208.17) and (368.17,207.01) .. (369.6,207.01) .. controls (371.03,207.01) and (372.19,208.17) .. (372.19,209.6) .. controls (372.19,211.03) and (371.03,212.19) .. (369.6,212.19) .. controls (368.17,212.19) and (367.01,211.03) .. (367.01,209.6) -- cycle ;
\draw  [fill={rgb, 255:red, 0; green, 0; blue, 0 }  ,fill opacity=1 ] (367.01,109.6) .. controls (367.01,108.17) and (368.17,107.01) .. (369.6,107.01) .. controls (371.03,107.01) and (372.19,108.17) .. (372.19,109.6) .. controls (372.19,111.03) and (371.03,112.19) .. (369.6,112.19) .. controls (368.17,112.19) and (367.01,111.03) .. (367.01,109.6) -- cycle ;
\draw  [fill={rgb, 255:red, 0; green, 0; blue, 0 }  ,fill opacity=1 ] (347.01,169.6) .. controls (347.01,168.17) and (348.17,167.01) .. (349.6,167.01) .. controls (351.03,167.01) and (352.19,168.17) .. (352.19,169.6) .. controls (352.19,171.03) and (351.03,172.19) .. (349.6,172.19) .. controls (348.17,172.19) and (347.01,171.03) .. (347.01,169.6) -- cycle ;
\draw   (239.8,229.9) -- (80.2,229.9) -- (80.2,89.9) ;
\draw    (260,150) -- (307.4,150.29) ;
\draw [shift={(309.4,150.3)}, rotate = 180.35] [color={rgb, 255:red, 0; green, 0; blue, 0 }  ][line width=0.75]    (10.93,-3.29) .. controls (6.95,-1.4) and (3.31,-0.3) .. (0,0) .. controls (3.31,0.3) and (6.95,1.4) .. (10.93,3.29)   ;

\draw (301.19,212) node [anchor=north west][inner sep=0.75pt]    {$\alpha $};
\draw (317,208.4) node [anchor=north west][inner sep=0.75pt]    {$\{$};

\end{tikzpicture}
\caption[A hypercube decomposition in $\R^2$]{Constructing a hypercube decomposition dividing a lattice in $\R^2$.}\label{fig:decomp}
\end{figure}

We conclude that:

\begin{proposition}\label{prop:formal analytic T from M} 
The action of $M$ on $T$ extends to the formal model $\fT_\alpha$ whenever $\alpha$ divides the lattice $\ell(M)=\Lambda\subset\R^r$.
\end{proposition}

Any such model leads to a model of $T/M$ by identifying any open $U_{\Delta}$ in our cover with the opens $U_{\lambda+\Delta}$ for all $\lambda\in\Lambda$. See Figure \ref{fig:TM model} for an example.

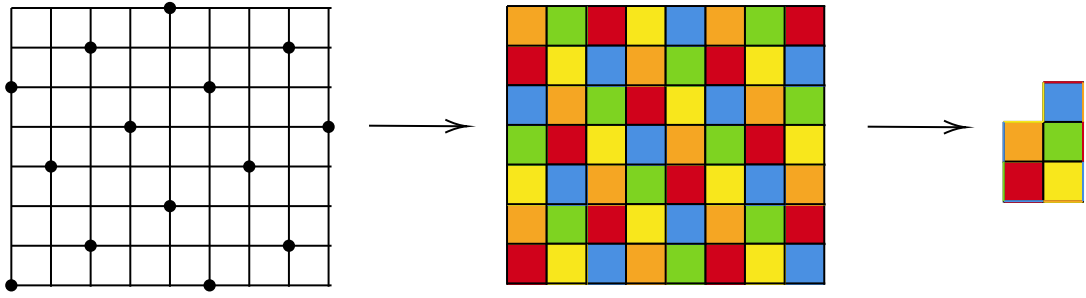
\begin{figure}

\tikzset{every picture/.style={line width=0.75pt}} 

\begin{tikzpicture}[x=0.75pt,y=0.75pt,yscale=-1,xscale=1]

\draw  [fill={rgb, 255:red, 0; green, 0; blue, 0 }  ,fill opacity=1 ] (438.01,221.6) .. controls (438.01,220.17) and (439.17,219.01) .. (440.6,219.01) .. controls (442.03,219.01) and (443.19,220.17) .. (443.19,221.6) .. controls (443.19,223.03) and (442.03,224.19) .. (440.6,224.19) .. controls (439.17,224.19) and (438.01,223.03) .. (438.01,221.6) -- cycle ;
\draw  [fill={rgb, 255:red, 0; green, 0; blue, 0 }  ,fill opacity=1 ] (438.01,121.6) .. controls (438.01,120.17) and (439.17,119.01) .. (440.6,119.01) .. controls (442.03,119.01) and (443.19,120.17) .. (443.19,121.6) .. controls (443.19,123.03) and (442.03,124.19) .. (440.6,124.19) .. controls (439.17,124.19) and (438.01,123.03) .. (438.01,121.6) -- cycle ;
\draw  [fill={rgb, 255:red, 0; green, 0; blue, 0 }  ,fill opacity=1 ] (418.01,181.6) .. controls (418.01,180.17) and (419.17,179.01) .. (420.6,179.01) .. controls (422.03,179.01) and (423.19,180.17) .. (423.19,181.6) .. controls (423.19,183.03) and (422.03,184.19) .. (420.6,184.19) .. controls (419.17,184.19) and (418.01,183.03) .. (418.01,181.6) -- cycle ;
\draw  [fill={rgb, 255:red, 0; green, 0; blue, 0 }  ,fill opacity=1 ] (398.01,141.6) .. controls (398.01,140.17) and (399.17,139.01) .. (400.6,139.01) .. controls (402.03,139.01) and (403.19,140.17) .. (403.19,141.6) .. controls (403.19,143.03) and (402.03,144.19) .. (400.6,144.19) .. controls (399.17,144.19) and (398.01,143.03) .. (398.01,141.6) -- cycle ;
\draw  [fill={rgb, 255:red, 0; green, 0; blue, 0 }  ,fill opacity=1 ] (378.01,201.6) .. controls (378.01,200.17) and (379.17,199.01) .. (380.6,199.01) .. controls (382.03,199.01) and (383.19,200.17) .. (383.19,201.6) .. controls (383.19,203.03) and (382.03,204.19) .. (380.6,204.19) .. controls (379.17,204.19) and (378.01,203.03) .. (378.01,201.6) -- cycle ;
\draw  [fill={rgb, 255:red, 0; green, 0; blue, 0 }  ,fill opacity=1 ] (358.01,161.6) .. controls (358.01,160.17) and (359.17,159.01) .. (360.6,159.01) .. controls (362.03,159.01) and (363.19,160.17) .. (363.19,161.6) .. controls (363.19,163.03) and (362.03,164.19) .. (360.6,164.19) .. controls (359.17,164.19) and (358.01,163.03) .. (358.01,161.6) -- cycle ;
\draw  [fill={rgb, 255:red, 0; green, 0; blue, 0 }  ,fill opacity=1 ] (338.01,221.6) .. controls (338.01,220.17) and (339.17,219.01) .. (340.6,219.01) .. controls (342.03,219.01) and (343.19,220.17) .. (343.19,221.6) .. controls (343.19,223.03) and (342.03,224.19) .. (340.6,224.19) .. controls (339.17,224.19) and (338.01,223.03) .. (338.01,221.6) -- cycle ;
\draw  [fill={rgb, 255:red, 0; green, 0; blue, 0 }  ,fill opacity=1 ] (338.01,121.6) .. controls (338.01,120.17) and (339.17,119.01) .. (340.6,119.01) .. controls (342.03,119.01) and (343.19,120.17) .. (343.19,121.6) .. controls (343.19,123.03) and (342.03,124.19) .. (340.6,124.19) .. controls (339.17,124.19) and (338.01,123.03) .. (338.01,121.6) -- cycle ;
\draw  [fill={rgb, 255:red, 0; green, 0; blue, 0 }  ,fill opacity=1 ] (318.01,181.6) .. controls (318.01,180.17) and (319.17,179.01) .. (320.6,179.01) .. controls (322.03,179.01) and (323.19,180.17) .. (323.19,181.6) .. controls (323.19,183.03) and (322.03,184.19) .. (320.6,184.19) .. controls (319.17,184.19) and (318.01,183.03) .. (318.01,181.6) -- cycle ;
\draw    (231,162) -- (278.4,162.29) ;
\draw [shift={(280.4,162.3)}, rotate = 180.35] [color={rgb, 255:red, 0; green, 0; blue, 0 }  ][line width=0.75]    (10.93,-3.29) .. controls (6.95,-1.4) and (3.31,-0.3) .. (0,0) .. controls (3.31,0.3) and (6.95,1.4) .. (10.93,3.29)   ;
\draw  [draw opacity=0] (50.6,102.6) -- (212.1,102.6) -- (212.1,243.29) -- (50.6,243.29) -- cycle ; \draw   (50.6,102.6) -- (50.6,243.29)(70.6,102.6) -- (70.6,243.29)(90.6,102.6) -- (90.6,243.29)(110.6,102.6) -- (110.6,243.29)(130.6,102.6) -- (130.6,243.29)(150.6,102.6) -- (150.6,243.29)(170.6,102.6) -- (170.6,243.29)(190.6,102.6) -- (190.6,243.29)(210.6,102.6) -- (210.6,243.29) ; \draw   (50.6,102.6) -- (212.1,102.6)(50.6,122.6) -- (212.1,122.6)(50.6,142.6) -- (212.1,142.6)(50.6,162.6) -- (212.1,162.6)(50.6,182.6) -- (212.1,182.6)(50.6,202.6) -- (212.1,202.6)(50.6,222.6) -- (212.1,222.6)(50.6,242.6) -- (212.1,242.6) ; \draw    ;
\draw  [fill={rgb, 255:red, 0; green, 0; blue, 0 }  ,fill opacity=1 ] (48.01,242.6) .. controls (48.01,241.17) and (49.17,240.01) .. (50.6,240.01) .. controls (52.03,240.01) and (53.19,241.17) .. (53.19,242.6) .. controls (53.19,244.03) and (52.03,245.19) .. (50.6,245.19) .. controls (49.17,245.19) and (48.01,244.03) .. (48.01,242.6) -- cycle ;
\draw  [fill={rgb, 255:red, 0; green, 0; blue, 0 }  ,fill opacity=1 ] (48.01,142.6) .. controls (48.01,141.17) and (49.17,140.01) .. (50.6,140.01) .. controls (52.03,140.01) and (53.19,141.17) .. (53.19,142.6) .. controls (53.19,144.03) and (52.03,145.19) .. (50.6,145.19) .. controls (49.17,145.19) and (48.01,144.03) .. (48.01,142.6) -- cycle ;
\draw  [fill={rgb, 255:red, 0; green, 0; blue, 0 }  ,fill opacity=1 ] (208.01,162.6) .. controls (208.01,161.17) and (209.17,160.01) .. (210.6,160.01) .. controls (212.03,160.01) and (213.19,161.17) .. (213.19,162.6) .. controls (213.19,164.03) and (212.03,165.19) .. (210.6,165.19) .. controls (209.17,165.19) and (208.01,164.03) .. (208.01,162.6) -- cycle ;
\draw  [fill={rgb, 255:red, 0; green, 0; blue, 0 }  ,fill opacity=1 ] (188.01,222.6) .. controls (188.01,221.17) and (189.17,220.01) .. (190.6,220.01) .. controls (192.03,220.01) and (193.19,221.17) .. (193.19,222.6) .. controls (193.19,224.03) and (192.03,225.19) .. (190.6,225.19) .. controls (189.17,225.19) and (188.01,224.03) .. (188.01,222.6) -- cycle ;
\draw  [fill={rgb, 255:red, 0; green, 0; blue, 0 }  ,fill opacity=1 ] (148.01,242.6) .. controls (148.01,241.17) and (149.17,240.01) .. (150.6,240.01) .. controls (152.03,240.01) and (153.19,241.17) .. (153.19,242.6) .. controls (153.19,244.03) and (152.03,245.19) .. (150.6,245.19) .. controls (149.17,245.19) and (148.01,244.03) .. (148.01,242.6) -- cycle ;
\draw  [fill={rgb, 255:red, 0; green, 0; blue, 0 }  ,fill opacity=1 ] (188.01,122.6) .. controls (188.01,121.17) and (189.17,120.01) .. (190.6,120.01) .. controls (192.03,120.01) and (193.19,121.17) .. (193.19,122.6) .. controls (193.19,124.03) and (192.03,125.19) .. (190.6,125.19) .. controls (189.17,125.19) and (188.01,124.03) .. (188.01,122.6) -- cycle ;
\draw  [fill={rgb, 255:red, 0; green, 0; blue, 0 }  ,fill opacity=1 ] (168.01,182.6) .. controls (168.01,181.17) and (169.17,180.01) .. (170.6,180.01) .. controls (172.03,180.01) and (173.19,181.17) .. (173.19,182.6) .. controls (173.19,184.03) and (172.03,185.19) .. (170.6,185.19) .. controls (169.17,185.19) and (168.01,184.03) .. (168.01,182.6) -- cycle ;
\draw  [fill={rgb, 255:red, 0; green, 0; blue, 0 }  ,fill opacity=1 ] (148.01,142.6) .. controls (148.01,141.17) and (149.17,140.01) .. (150.6,140.01) .. controls (152.03,140.01) and (153.19,141.17) .. (153.19,142.6) .. controls (153.19,144.03) and (152.03,145.19) .. (150.6,145.19) .. controls (149.17,145.19) and (148.01,144.03) .. (148.01,142.6) -- cycle ;
\draw  [fill={rgb, 255:red, 0; green, 0; blue, 0 }  ,fill opacity=1 ] (128.01,202.6) .. controls (128.01,201.17) and (129.17,200.01) .. (130.6,200.01) .. controls (132.03,200.01) and (133.19,201.17) .. (133.19,202.6) .. controls (133.19,204.03) and (132.03,205.19) .. (130.6,205.19) .. controls (129.17,205.19) and (128.01,204.03) .. (128.01,202.6) -- cycle ;
\draw  [fill={rgb, 255:red, 0; green, 0; blue, 0 }  ,fill opacity=1 ] (128.01,102.6) .. controls (128.01,101.17) and (129.17,100.01) .. (130.6,100.01) .. controls (132.03,100.01) and (133.19,101.17) .. (133.19,102.6) .. controls (133.19,104.03) and (132.03,105.19) .. (130.6,105.19) .. controls (129.17,105.19) and (128.01,104.03) .. (128.01,102.6) -- cycle ;
\draw  [fill={rgb, 255:red, 0; green, 0; blue, 0 }  ,fill opacity=1 ] (108.01,162.6) .. controls (108.01,161.17) and (109.17,160.01) .. (110.6,160.01) .. controls (112.03,160.01) and (113.19,161.17) .. (113.19,162.6) .. controls (113.19,164.03) and (112.03,165.19) .. (110.6,165.19) .. controls (109.17,165.19) and (108.01,164.03) .. (108.01,162.6) -- cycle ;
\draw  [fill={rgb, 255:red, 0; green, 0; blue, 0 }  ,fill opacity=1 ] (88.01,222.6) .. controls (88.01,221.17) and (89.17,220.01) .. (90.6,220.01) .. controls (92.03,220.01) and (93.19,221.17) .. (93.19,222.6) .. controls (93.19,224.03) and (92.03,225.19) .. (90.6,225.19) .. controls (89.17,225.19) and (88.01,224.03) .. (88.01,222.6) -- cycle ;
\draw  [fill={rgb, 255:red, 0; green, 0; blue, 0 }  ,fill opacity=1 ] (88.01,122.6) .. controls (88.01,121.17) and (89.17,120.01) .. (90.6,120.01) .. controls (92.03,120.01) and (93.19,121.17) .. (93.19,122.6) .. controls (93.19,124.03) and (92.03,125.19) .. (90.6,125.19) .. controls (89.17,125.19) and (88.01,124.03) .. (88.01,122.6) -- cycle ;
\draw  [fill={rgb, 255:red, 0; green, 0; blue, 0 }  ,fill opacity=1 ] (68.01,182.6) .. controls (68.01,181.17) and (69.17,180.01) .. (70.6,180.01) .. controls (72.03,180.01) and (73.19,181.17) .. (73.19,182.6) .. controls (73.19,184.03) and (72.03,185.19) .. (70.6,185.19) .. controls (69.17,185.19) and (68.01,184.03) .. (68.01,182.6) -- cycle ;
\draw  [draw opacity=0][fill={rgb, 255:red, 208; green, 2; blue, 27 }  ,fill opacity=1 ] (300.6,221.6) -- (321.1,221.6) -- (321.1,242.29) -- (300.6,242.29) -- cycle ; \draw   (300.6,221.6) -- (300.6,242.29)(320.6,221.6) -- (320.6,242.29) ; \draw   (300.6,221.6) -- (321.1,221.6)(300.6,241.6) -- (321.1,241.6) ; \draw    ;
\draw  [draw opacity=0][fill={rgb, 255:red, 208; green, 2; blue, 27 }  ,fill opacity=1 ] (320.6,161.6) -- (341.1,161.6) -- (341.1,182.29) -- (320.6,182.29) -- cycle ; \draw   (320.6,161.6) -- (320.6,182.29)(340.6,161.6) -- (340.6,182.29) ; \draw   (320.6,161.6) -- (341.1,161.6)(320.6,181.6) -- (341.1,181.6) ; \draw    ;
\draw  [draw opacity=0][fill={rgb, 255:red, 208; green, 2; blue, 27 }  ,fill opacity=1 ] (340.6,101.6) -- (361.1,101.6) -- (361.1,122.29) -- (340.6,122.29) -- cycle ; \draw   (340.6,101.6) -- (340.6,122.29)(360.6,101.6) -- (360.6,122.29) ; \draw   (340.6,101.6) -- (361.1,101.6)(340.6,121.6) -- (361.1,121.6) ; \draw    ;
\draw  [draw opacity=0][fill={rgb, 255:red, 208; green, 2; blue, 27 }  ,fill opacity=1 ] (340.6,201.6) -- (361.1,201.6) -- (361.1,222.29) -- (340.6,222.29) -- cycle ; \draw   (340.6,201.6) -- (340.6,222.29)(360.6,201.6) -- (360.6,222.29) ; \draw   (340.6,201.6) -- (361.1,201.6)(340.6,221.6) -- (361.1,221.6) ; \draw    ;
\draw  [draw opacity=0][fill={rgb, 255:red, 208; green, 2; blue, 27 }  ,fill opacity=1 ] (360.6,141.6) -- (381.1,141.6) -- (381.1,162.29) -- (360.6,162.29) -- cycle ; \draw   (360.6,141.6) -- (360.6,162.29)(380.6,141.6) -- (380.6,162.29) ; \draw   (360.6,141.6) -- (381.1,141.6)(360.6,161.6) -- (381.1,161.6) ; \draw    ;
\draw  [draw opacity=0][fill={rgb, 255:red, 208; green, 2; blue, 27 }  ,fill opacity=1 ] (380.6,181.6) -- (401.1,181.6) -- (401.1,202.29) -- (380.6,202.29) -- cycle ; \draw   (380.6,181.6) -- (380.6,202.29)(400.6,181.6) -- (400.6,202.29) ; \draw   (380.6,181.6) -- (401.1,181.6)(380.6,201.6) -- (401.1,201.6) ; \draw    ;
\draw  [draw opacity=0][fill={rgb, 255:red, 208; green, 2; blue, 27 }  ,fill opacity=1 ] (400.6,221.6) -- (421.1,221.6) -- (421.1,242.29) -- (400.6,242.29) -- cycle ; \draw   (400.6,221.6) -- (400.6,242.29)(420.6,221.6) -- (420.6,242.29) ; \draw   (400.6,221.6) -- (421.1,221.6)(400.6,241.6) -- (421.1,241.6) ; \draw    ;
\draw  [draw opacity=0][fill={rgb, 255:red, 208; green, 2; blue, 27 }  ,fill opacity=1 ] (440.6,201.6) -- (461.1,201.6) -- (461.1,222.29) -- (440.6,222.29) -- cycle ; \draw   (440.6,201.6) -- (440.6,222.29)(460.6,201.6) -- (460.6,222.29) ; \draw   (440.6,201.6) -- (461.1,201.6)(440.6,221.6) -- (461.1,221.6) ; \draw    ;
\draw  [draw opacity=0][fill={rgb, 255:red, 208; green, 2; blue, 27 }  ,fill opacity=1 ] (420.6,161.6) -- (441.1,161.6) -- (441.1,182.29) -- (420.6,182.29) -- cycle ; \draw   (420.6,161.6) -- (420.6,182.29)(440.6,161.6) -- (440.6,182.29) ; \draw   (420.6,161.6) -- (441.1,161.6)(420.6,181.6) -- (441.1,181.6) ; \draw    ;
\draw  [draw opacity=0][fill={rgb, 255:red, 208; green, 2; blue, 27 }  ,fill opacity=1 ] (300.6,121.6) -- (321.1,121.6) -- (321.1,142.29) -- (300.6,142.29) -- cycle ; \draw   (300.6,121.6) -- (300.6,142.29)(320.6,121.6) -- (320.6,142.29) ; \draw   (300.6,121.6) -- (321.1,121.6)(300.6,141.6) -- (321.1,141.6) ; \draw    ;
\draw  [draw opacity=0][fill={rgb, 255:red, 208; green, 2; blue, 27 }  ,fill opacity=1 ] (400.6,121.6) -- (421.1,121.6) -- (421.1,142.29) -- (400.6,142.29) -- cycle ; \draw   (400.6,121.6) -- (400.6,142.29)(420.6,121.6) -- (420.6,142.29) ; \draw   (400.6,121.6) -- (421.1,121.6)(400.6,141.6) -- (421.1,141.6) ; \draw    ;
\draw  [draw opacity=0][fill={rgb, 255:red, 208; green, 2; blue, 27 }  ,fill opacity=1 ] (440.6,101.6) -- (461.1,101.6) -- (461.1,122.29) -- (440.6,122.29) -- cycle ; \draw   (440.6,101.6) -- (440.6,122.29)(460.6,101.6) -- (460.6,122.29) ; \draw   (440.6,101.6) -- (461.1,101.6)(440.6,121.6) -- (461.1,121.6) ; \draw    ;
\draw  [draw opacity=0][fill={rgb, 255:red, 245; green, 166; blue, 35 }  ,fill opacity=1 ] (300.6,201.6) -- (321.1,201.6) -- (321.1,222.29) -- (300.6,222.29) -- cycle ; \draw   (300.6,201.6) -- (300.6,222.29)(320.6,201.6) -- (320.6,222.29) ; \draw   (300.6,201.6) -- (321.1,201.6)(300.6,221.6) -- (321.1,221.6) ; \draw    ;
\draw  [draw opacity=0][fill={rgb, 255:red, 245; green, 166; blue, 35 }  ,fill opacity=1 ] (320.6,141.6) -- (341.1,141.6) -- (341.1,162.29) -- (320.6,162.29) -- cycle ; \draw   (320.6,141.6) -- (320.6,162.29)(340.6,141.6) -- (340.6,162.29) ; \draw   (320.6,141.6) -- (341.1,141.6)(320.6,161.6) -- (341.1,161.6) ; \draw    ;
\draw  [draw opacity=0][fill={rgb, 255:red, 245; green, 166; blue, 35 }  ,fill opacity=1 ] (340.6,181.6) -- (361.1,181.6) -- (361.1,202.29) -- (340.6,202.29) -- cycle ; \draw   (340.6,181.6) -- (340.6,202.29)(360.6,181.6) -- (360.6,202.29) ; \draw   (340.6,181.6) -- (361.1,181.6)(340.6,201.6) -- (361.1,201.6) ; \draw    ;
\draw  [draw opacity=0][fill={rgb, 255:red, 245; green, 166; blue, 35 }  ,fill opacity=1 ] (360.6,121.6) -- (381.1,121.6) -- (381.1,142.29) -- (360.6,142.29) -- cycle ; \draw   (360.6,121.6) -- (360.6,142.29)(380.6,121.6) -- (380.6,142.29) ; \draw   (360.6,121.6) -- (381.1,121.6)(360.6,141.6) -- (381.1,141.6) ; \draw    ;
\draw  [draw opacity=0][fill={rgb, 255:red, 245; green, 166; blue, 35 }  ,fill opacity=1 ] (300.6,101.6) -- (321.1,101.6) -- (321.1,122.29) -- (300.6,122.29) -- cycle ; \draw   (300.6,101.6) -- (300.6,122.29)(320.6,101.6) -- (320.6,122.29) ; \draw   (300.6,101.6) -- (321.1,101.6)(300.6,121.6) -- (321.1,121.6) ; \draw    ;
\draw  [draw opacity=0][fill={rgb, 255:red, 245; green, 166; blue, 35 }  ,fill opacity=1 ] (380.6,161.6) -- (401.1,161.6) -- (401.1,182.29) -- (380.6,182.29) -- cycle ; \draw   (380.6,161.6) -- (380.6,182.29)(400.6,161.6) -- (400.6,182.29) ; \draw   (380.6,161.6) -- (401.1,161.6)(380.6,181.6) -- (401.1,181.6) ; \draw    ;
\draw  [draw opacity=0][fill={rgb, 255:red, 245; green, 166; blue, 35 }  ,fill opacity=1 ] (400.6,201.6) -- (421.1,201.6) -- (421.1,222.29) -- (400.6,222.29) -- cycle ; \draw   (400.6,201.6) -- (400.6,222.29)(420.6,201.6) -- (420.6,222.29) ; \draw   (400.6,201.6) -- (421.1,201.6)(400.6,221.6) -- (421.1,221.6) ; \draw    ;
\draw  [draw opacity=0][fill={rgb, 255:red, 245; green, 166; blue, 35 }  ,fill opacity=1 ] (420.6,141.6) -- (441.1,141.6) -- (441.1,162.29) -- (420.6,162.29) -- cycle ; \draw   (420.6,141.6) -- (420.6,162.29)(440.6,141.6) -- (440.6,162.29) ; \draw   (420.6,141.6) -- (441.1,141.6)(420.6,161.6) -- (441.1,161.6) ; \draw    ;
\draw  [draw opacity=0][fill={rgb, 255:red, 245; green, 166; blue, 35 }  ,fill opacity=1 ] (400.6,101.6) -- (421.1,101.6) -- (421.1,122.29) -- (400.6,122.29) -- cycle ; \draw   (400.6,101.6) -- (400.6,122.29)(420.6,101.6) -- (420.6,122.29) ; \draw   (400.6,101.6) -- (421.1,101.6)(400.6,121.6) -- (421.1,121.6) ; \draw    ;
\draw  [draw opacity=0][fill={rgb, 255:red, 245; green, 166; blue, 35 }  ,fill opacity=1 ] (440.6,181.6) -- (461.1,181.6) -- (461.1,202.29) -- (440.6,202.29) -- cycle ; \draw   (440.6,181.6) -- (440.6,202.29)(460.6,181.6) -- (460.6,202.29) ; \draw   (440.6,181.6) -- (461.1,181.6)(440.6,201.6) -- (461.1,201.6) ; \draw    ;
\draw  [draw opacity=0][fill={rgb, 255:red, 248; green, 231; blue, 28 }  ,fill opacity=1 ] (300.6,181.6) -- (321.1,181.6) -- (321.1,202.29) -- (300.6,202.29) -- cycle ; \draw   (300.6,181.6) -- (300.6,202.29)(320.6,181.6) -- (320.6,202.29) ; \draw   (300.6,181.6) -- (321.1,181.6)(300.6,201.6) -- (321.1,201.6) ; \draw    ;
\draw  [draw opacity=0][fill={rgb, 255:red, 248; green, 231; blue, 28 }  ,fill opacity=1 ] (320.6,121.6) -- (341.1,121.6) -- (341.1,142.29) -- (320.6,142.29) -- cycle ; \draw   (320.6,121.6) -- (320.6,142.29)(340.6,121.6) -- (340.6,142.29) ; \draw   (320.6,121.6) -- (341.1,121.6)(320.6,141.6) -- (341.1,141.6) ; \draw    ;
\draw  [draw opacity=0][fill={rgb, 255:red, 248; green, 231; blue, 28 }  ,fill opacity=1 ] (340.6,161.6) -- (361.1,161.6) -- (361.1,182.29) -- (340.6,182.29) -- cycle ; \draw   (340.6,161.6) -- (340.6,182.29)(360.6,161.6) -- (360.6,182.29) ; \draw   (340.6,161.6) -- (361.1,161.6)(340.6,181.6) -- (361.1,181.6) ; \draw    ;
\draw  [draw opacity=0][fill={rgb, 255:red, 248; green, 231; blue, 28 }  ,fill opacity=1 ] (360.6,101.6) -- (381.1,101.6) -- (381.1,122.29) -- (360.6,122.29) -- cycle ; \draw   (360.6,101.6) -- (360.6,122.29)(380.6,101.6) -- (380.6,122.29) ; \draw   (360.6,101.6) -- (381.1,101.6)(360.6,121.6) -- (381.1,121.6) ; \draw    ;
\draw  [draw opacity=0][fill={rgb, 255:red, 248; green, 231; blue, 28 }  ,fill opacity=1 ] (380.6,141.6) -- (401.1,141.6) -- (401.1,162.29) -- (380.6,162.29) -- cycle ; \draw   (380.6,141.6) -- (380.6,162.29)(400.6,141.6) -- (400.6,162.29) ; \draw   (380.6,141.6) -- (401.1,141.6)(380.6,161.6) -- (401.1,161.6) ; \draw    ;
\draw  [draw opacity=0][fill={rgb, 255:red, 248; green, 231; blue, 28 }  ,fill opacity=1 ] (400.6,181.6) -- (421.1,181.6) -- (421.1,202.29) -- (400.6,202.29) -- cycle ; \draw   (400.6,181.6) -- (400.6,202.29)(420.6,181.6) -- (420.6,202.29) ; \draw   (400.6,181.6) -- (421.1,181.6)(400.6,201.6) -- (421.1,201.6) ; \draw    ;
\draw  [draw opacity=0][fill={rgb, 255:red, 248; green, 231; blue, 28 }  ,fill opacity=1 ] (420.6,121.6) -- (441.1,121.6) -- (441.1,142.29) -- (420.6,142.29) -- cycle ; \draw   (420.6,121.6) -- (420.6,142.29)(440.6,121.6) -- (440.6,142.29) ; \draw   (420.6,121.6) -- (441.1,121.6)(420.6,141.6) -- (441.1,141.6) ; \draw    ;
\draw  [draw opacity=0][fill={rgb, 255:red, 248; green, 231; blue, 28 }  ,fill opacity=1 ] (440.6,161.6) -- (461.1,161.6) -- (461.1,182.29) -- (440.6,182.29) -- cycle ; \draw   (440.6,161.6) -- (440.6,182.29)(460.6,161.6) -- (460.6,182.29) ; \draw   (440.6,161.6) -- (461.1,161.6)(440.6,181.6) -- (461.1,181.6) ; \draw    ;
\draw  [draw opacity=0][fill={rgb, 255:red, 126; green, 211; blue, 33 }  ,fill opacity=1 ] (300.6,161.6) -- (321.1,161.6) -- (321.1,182.29) -- (300.6,182.29) -- cycle ; \draw   (300.6,161.6) -- (300.6,182.29)(320.6,161.6) -- (320.6,182.29) ; \draw   (300.6,161.6) -- (321.1,161.6)(300.6,181.6) -- (321.1,181.6) ; \draw    ;
\draw  [draw opacity=0][fill={rgb, 255:red, 126; green, 211; blue, 33 }  ,fill opacity=1 ] (320.6,101.6) -- (341.1,101.6) -- (341.1,122.29) -- (320.6,122.29) -- cycle ; \draw   (320.6,101.6) -- (320.6,122.29)(340.6,101.6) -- (340.6,122.29) ; \draw   (320.6,101.6) -- (341.1,101.6)(320.6,121.6) -- (341.1,121.6) ; \draw    ;
\draw  [draw opacity=0][fill={rgb, 255:red, 126; green, 211; blue, 33 }  ,fill opacity=1 ] (340.6,141.6) -- (361.1,141.6) -- (361.1,162.29) -- (340.6,162.29) -- cycle ; \draw   (340.6,141.6) -- (340.6,162.29)(360.6,141.6) -- (360.6,162.29) ; \draw   (340.6,141.6) -- (361.1,141.6)(340.6,161.6) -- (361.1,161.6) ; \draw    ;
\draw  [draw opacity=0][fill={rgb, 255:red, 126; green, 211; blue, 33 }  ,fill opacity=1 ] (380.6,121.6) -- (401.1,121.6) -- (401.1,142.29) -- (380.6,142.29) -- cycle ; \draw   (380.6,121.6) -- (380.6,142.29)(400.6,121.6) -- (400.6,142.29) ; \draw   (380.6,121.6) -- (401.1,121.6)(380.6,141.6) -- (401.1,141.6) ; \draw    ;
\draw  [draw opacity=0][fill={rgb, 255:red, 126; green, 211; blue, 33 }  ,fill opacity=1 ] (400.6,161.6) -- (421.1,161.6) -- (421.1,182.29) -- (400.6,182.29) -- cycle ; \draw   (400.6,161.6) -- (400.6,182.29)(420.6,161.6) -- (420.6,182.29) ; \draw   (400.6,161.6) -- (421.1,161.6)(400.6,181.6) -- (421.1,181.6) ; \draw    ;
\draw  [draw opacity=0][fill={rgb, 255:red, 126; green, 211; blue, 33 }  ,fill opacity=1 ] (440.6,141.6) -- (461.1,141.6) -- (461.1,162.29) -- (440.6,162.29) -- cycle ; \draw   (440.6,141.6) -- (440.6,162.29)(460.6,141.6) -- (460.6,162.29) ; \draw   (440.6,141.6) -- (461.1,141.6)(440.6,161.6) -- (461.1,161.6) ; \draw    ;
\draw  [draw opacity=0][fill={rgb, 255:red, 126; green, 211; blue, 33 }  ,fill opacity=1 ] (320.6,201.6) -- (341.1,201.6) -- (341.1,222.29) -- (320.6,222.29) -- cycle ; \draw   (320.6,201.6) -- (320.6,222.29)(340.6,201.6) -- (340.6,222.29) ; \draw   (320.6,201.6) -- (341.1,201.6)(320.6,221.6) -- (341.1,221.6) ; \draw    ;
\draw  [draw opacity=0][fill={rgb, 255:red, 126; green, 211; blue, 33 }  ,fill opacity=1 ] (360.6,181.6) -- (381.1,181.6) -- (381.1,202.29) -- (360.6,202.29) -- cycle ; \draw   (360.6,181.6) -- (360.6,202.29)(380.6,181.6) -- (380.6,202.29) ; \draw   (360.6,181.6) -- (381.1,181.6)(360.6,201.6) -- (381.1,201.6) ; \draw    ;
\draw  [draw opacity=0][fill={rgb, 255:red, 126; green, 211; blue, 33 }  ,fill opacity=1 ] (420.6,201.6) -- (441.1,201.6) -- (441.1,222.29) -- (420.6,222.29) -- cycle ; \draw   (420.6,201.6) -- (420.6,222.29)(440.6,201.6) -- (440.6,222.29) ; \draw   (420.6,201.6) -- (441.1,201.6)(420.6,221.6) -- (441.1,221.6) ; \draw    ;
\draw  [draw opacity=0][fill={rgb, 255:red, 126; green, 211; blue, 33 }  ,fill opacity=1 ] (420.6,101.6) -- (441.1,101.6) -- (441.1,122.29) -- (420.6,122.29) -- cycle ; \draw   (420.6,101.6) -- (420.6,122.29)(440.6,101.6) -- (440.6,122.29) ; \draw   (420.6,101.6) -- (441.1,101.6)(420.6,121.6) -- (441.1,121.6) ; \draw    ;
\draw  [draw opacity=0][fill={rgb, 255:red, 126; green, 211; blue, 33 }  ,fill opacity=1 ] (380.6,221.6) -- (401.1,221.6) -- (401.1,242.29) -- (380.6,242.29) -- cycle ; \draw   (380.6,221.6) -- (380.6,242.29)(400.6,221.6) -- (400.6,242.29) ; \draw   (380.6,221.6) -- (401.1,221.6)(380.6,241.6) -- (401.1,241.6) ; \draw    ;
\draw  [draw opacity=0][fill={rgb, 255:red, 74; green, 144; blue, 226 }  ,fill opacity=1 ] (300.6,141.6) -- (321.1,141.6) -- (321.1,162.29) -- (300.6,162.29) -- cycle ; \draw   (300.6,141.6) -- (300.6,162.29)(320.6,141.6) -- (320.6,162.29) ; \draw   (300.6,141.6) -- (321.1,141.6)(300.6,161.6) -- (321.1,161.6) ; \draw    ;
\draw  [draw opacity=0][fill={rgb, 255:red, 74; green, 144; blue, 226 }  ,fill opacity=1 ] (320.6,181.6) -- (341.1,181.6) -- (341.1,202.29) -- (320.6,202.29) -- cycle ; \draw   (320.6,181.6) -- (320.6,202.29)(340.6,181.6) -- (340.6,202.29) ; \draw   (320.6,181.6) -- (341.1,181.6)(320.6,201.6) -- (341.1,201.6) ; \draw    ;
\draw  [draw opacity=0][fill={rgb, 255:red, 74; green, 144; blue, 226 }  ,fill opacity=1 ] (340.6,221.6) -- (361.1,221.6) -- (361.1,242.29) -- (340.6,242.29) -- cycle ; \draw   (340.6,221.6) -- (340.6,242.29)(360.6,221.6) -- (360.6,242.29) ; \draw   (340.6,221.6) -- (361.1,221.6)(340.6,241.6) -- (361.1,241.6) ; \draw    ;
\draw  [draw opacity=0][fill={rgb, 255:red, 74; green, 144; blue, 226 }  ,fill opacity=1 ] (340.6,121.6) -- (361.1,121.6) -- (361.1,142.29) -- (340.6,142.29) -- cycle ; \draw   (340.6,121.6) -- (340.6,142.29)(360.6,121.6) -- (360.6,142.29) ; \draw   (340.6,121.6) -- (361.1,121.6)(340.6,141.6) -- (361.1,141.6) ; \draw    ;
\draw  [draw opacity=0][fill={rgb, 255:red, 74; green, 144; blue, 226 }  ,fill opacity=1 ] (360.6,161.6) -- (381.1,161.6) -- (381.1,182.29) -- (360.6,182.29) -- cycle ; \draw   (360.6,161.6) -- (360.6,182.29)(380.6,161.6) -- (380.6,182.29) ; \draw   (360.6,161.6) -- (381.1,161.6)(360.6,181.6) -- (381.1,181.6) ; \draw    ;
\draw  [draw opacity=0][fill={rgb, 255:red, 74; green, 144; blue, 226 }  ,fill opacity=1 ] (380.6,201.6) -- (401.1,201.6) -- (401.1,222.29) -- (380.6,222.29) -- cycle ; \draw   (380.6,201.6) -- (380.6,222.29)(400.6,201.6) -- (400.6,222.29) ; \draw   (380.6,201.6) -- (401.1,201.6)(380.6,221.6) -- (401.1,221.6) ; \draw    ;
\draw  [draw opacity=0][fill={rgb, 255:red, 74; green, 144; blue, 226 }  ,fill opacity=1 ] (420.6,181.6) -- (441.1,181.6) -- (441.1,202.29) -- (420.6,202.29) -- cycle ; \draw   (420.6,181.6) -- (420.6,202.29)(440.6,181.6) -- (440.6,202.29) ; \draw   (420.6,181.6) -- (441.1,181.6)(420.6,201.6) -- (441.1,201.6) ; \draw    ;
\draw  [draw opacity=0][fill={rgb, 255:red, 74; green, 144; blue, 226 }  ,fill opacity=1 ] (400.6,141.6) -- (421.1,141.6) -- (421.1,162.29) -- (400.6,162.29) -- cycle ; \draw   (400.6,141.6) -- (400.6,162.29)(420.6,141.6) -- (420.6,162.29) ; \draw   (400.6,141.6) -- (421.1,141.6)(400.6,161.6) -- (421.1,161.6) ; \draw    ;
\draw  [draw opacity=0][fill={rgb, 255:red, 74; green, 144; blue, 226 }  ,fill opacity=1 ] (440.6,121.6) -- (461.1,121.6) -- (461.1,142.29) -- (440.6,142.29) -- cycle ; \draw   (440.6,121.6) -- (440.6,142.29)(460.6,121.6) -- (460.6,142.29) ; \draw   (440.6,121.6) -- (461.1,121.6)(440.6,141.6) -- (461.1,141.6) ; \draw    ;
\draw  [draw opacity=0][fill={rgb, 255:red, 74; green, 144; blue, 226 }  ,fill opacity=1 ] (440.6,221.6) -- (461.1,221.6) -- (461.1,242.29) -- (440.6,242.29) -- cycle ; \draw   (440.6,221.6) -- (440.6,242.29)(460.6,221.6) -- (460.6,242.29) ; \draw   (440.6,221.6) -- (461.1,221.6)(440.6,241.6) -- (461.1,241.6) ; \draw    ;
\draw  [draw opacity=0][fill={rgb, 255:red, 248; green, 231; blue, 28 }  ,fill opacity=1 ] (420.6,221.6) -- (441.1,221.6) -- (441.1,242.29) -- (420.6,242.29) -- cycle ; \draw   (420.6,221.6) -- (420.6,242.29)(440.6,221.6) -- (440.6,242.29) ; \draw   (420.6,221.6) -- (441.1,221.6)(420.6,241.6) -- (441.1,241.6) ; \draw    ;
\draw  [draw opacity=0][fill={rgb, 255:red, 248; green, 231; blue, 28 }  ,fill opacity=1 ] (360.6,201.6) -- (381.1,201.6) -- (381.1,222.29) -- (360.6,222.29) -- cycle ; \draw   (360.6,201.6) -- (360.6,222.29)(380.6,201.6) -- (380.6,222.29) ; \draw   (360.6,201.6) -- (381.1,201.6)(360.6,221.6) -- (381.1,221.6) ; \draw    ;
\draw  [draw opacity=0][fill={rgb, 255:red, 248; green, 231; blue, 28 }  ,fill opacity=1 ] (320.6,221.6) -- (341.1,221.6) -- (341.1,242.29) -- (320.6,242.29) -- cycle ; \draw   (320.6,221.6) -- (320.6,242.29)(340.6,221.6) -- (340.6,242.29) ; \draw   (320.6,221.6) -- (341.1,221.6)(320.6,241.6) -- (341.1,241.6) ; \draw    ;
\draw  [draw opacity=0][fill={rgb, 255:red, 245; green, 166; blue, 35 }  ,fill opacity=1 ] (360.6,221.6) -- (381.1,221.6) -- (381.1,242.29) -- (360.6,242.29) -- cycle ; \draw   (360.6,221.6) -- (360.6,242.29)(380.6,221.6) -- (380.6,242.29) ; \draw   (360.6,221.6) -- (381.1,221.6)(360.6,241.6) -- (381.1,241.6) ; \draw    ;
\draw  [draw opacity=0][fill={rgb, 255:red, 74; green, 144; blue, 226 }  ,fill opacity=1 ] (380.6,101.6) -- (401.1,101.6) -- (401.1,122.29) -- (380.6,122.29) -- cycle ; \draw   (380.6,101.6) -- (380.6,122.29)(400.6,101.6) -- (400.6,122.29) ; \draw   (380.6,101.6) -- (401.1,101.6)(380.6,121.6) -- (401.1,121.6) ; \draw    ;
\draw    (482.5,162.5) -- (529.9,162.79) ;
\draw [shift={(531.9,162.8)}, rotate = 180.35] [color={rgb, 255:red, 0; green, 0; blue, 0 }  ][line width=0.75]    (10.93,-3.29) .. controls (6.95,-1.4) and (3.31,-0.3) .. (0,0) .. controls (3.31,0.3) and (6.95,1.4) .. (10.93,3.29)   ;
\draw  [draw opacity=0][fill={rgb, 255:red, 208; green, 2; blue, 27 }  ,fill opacity=1 ] (551.1,180.1) -- (571.6,180.1) -- (571.6,200.79) -- (551.1,200.79) -- cycle ; \draw   (551.1,180.1) -- (551.1,200.79)(571.1,180.1) -- (571.1,200.79) ; \draw   (551.1,180.1) -- (571.6,180.1)(551.1,200.1) -- (571.6,200.1) ; \draw    ;
\draw  [draw opacity=0][fill={rgb, 255:red, 245; green, 166; blue, 35 }  ,fill opacity=1 ] (551.1,160.1) -- (571.6,160.1) -- (571.6,180.79) -- (551.1,180.79) -- cycle ; \draw   (551.1,160.1) -- (551.1,180.79)(571.1,160.1) -- (571.1,180.79) ; \draw   (551.1,160.1) -- (571.6,160.1)(551.1,180.1) -- (571.6,180.1) ; \draw    ;
\draw  [draw opacity=0][fill={rgb, 255:red, 126; green, 211; blue, 33 }  ,fill opacity=1 ] (571.1,160.1) -- (591.6,160.1) -- (591.6,180.79) -- (571.1,180.79) -- cycle ; \draw   (571.1,160.1) -- (571.1,180.79)(591.1,160.1) -- (591.1,180.79) ; \draw   (571.1,160.1) -- (591.6,160.1)(571.1,180.1) -- (591.6,180.1) ; \draw    ;
\draw  [draw opacity=0][fill={rgb, 255:red, 248; green, 231; blue, 28 }  ,fill opacity=1 ] (571.1,180.1) -- (591.6,180.1) -- (591.6,200.79) -- (571.1,200.79) -- cycle ; \draw   (571.1,180.1) -- (571.1,200.79)(591.1,180.1) -- (591.1,200.79) ; \draw   (571.1,180.1) -- (591.6,180.1)(571.1,200.1) -- (591.6,200.1) ; \draw    ;
\draw  [draw opacity=0][fill={rgb, 255:red, 74; green, 144; blue, 226 }  ,fill opacity=1 ] (571.1,140.1) -- (591.6,140.1) -- (591.6,160.79) -- (571.1,160.79) -- cycle ; \draw   (571.1,140.1) -- (571.1,160.79)(591.1,140.1) -- (591.1,160.79) ; \draw   (571.1,140.1) -- (591.6,140.1)(571.1,160.1) -- (591.6,160.1) ; \draw    ;
\draw [color={rgb, 255:red, 208; green, 2; blue, 27 }  ,draw opacity=1 ]   (571.1,140.1) -- (591.1,140.1) ;
\draw [color={rgb, 255:red, 248; green, 231; blue, 28 }  ,draw opacity=1 ]   (551.1,160.1) -- (571.1,160.1) ;
\draw [color={rgb, 255:red, 74; green, 144; blue, 226 }  ,draw opacity=1 ]   (551.1,200.1) -- (571.1,200.1) ;
\draw [color={rgb, 255:red, 245; green, 166; blue, 35 }  ,draw opacity=1 ]   (571.1,200.1) -- (591.1,200.1) ;
\draw [color={rgb, 255:red, 245; green, 166; blue, 35 }  ,draw opacity=1 ]   (591.1,160.1) -- (591.1,140.1) ;
\draw [color={rgb, 255:red, 126; green, 211; blue, 33 }  ,draw opacity=1 ]   (551.1,200.1) -- (551.1,180.1) ;
\draw [color={rgb, 255:red, 74; green, 144; blue, 226 }  ,draw opacity=1 ]   (551.1,180.1) -- (551.1,160.1) ;
\draw [color={rgb, 255:red, 248; green, 231; blue, 28 }  ,draw opacity=1 ]   (571.1,160.1) -- (571.1,140.1) ;
\draw [color={rgb, 255:red, 74; green, 144; blue, 226 }  ,draw opacity=1 ]   (591.1,200.1) -- (591.1,180.1) ;
\draw [color={rgb, 255:red, 208; green, 2; blue, 27 }  ,draw opacity=1 ]   (591.1,180.1) -- (591.1,160.1) ;

\end{tikzpicture}

\caption[Quotienting a decomposition of $\R$ to get a model of a Tate curve]{Quotienting the decomposition in Figure \ref{fig:decomp} by the lattice identifies squares of the same color. The hypercube decomposition gives a formal analytic cover, and therefore a model of $T/M$ as in Proposition \ref{prop:formal analytic T from M}.}\label{fig:TM model}
\end{figure}

\begin{example}\label{eg:tate curve decomp}
The construction of the Tate curve $\G_m/q^\Z$ with $|q|=c$ in Example \ref{eg:Tate curve} corresponds to the decomposition of $\R$ into intervals of length $c/2$. See Figure \ref{fig:Tate decomp}.

\begin{figure}

\tikzset{every picture/.style={line width=0.75pt}} 

\begin{tikzpicture}[x=0.75pt,y=0.75pt,yscale=-.9,xscale=.9]

\draw  [color={rgb, 255:red, 144; green, 19; blue, 254 }  ,draw opacity=1 ][fill={rgb, 255:red, 255; green, 165; blue, 0 }  ,fill opacity=0.91 ][line width=3]  (204.2,635.17) .. controls (204.2,575.73) and (252.39,527.55) .. (311.83,527.55) .. controls (371.27,527.55) and (419.45,575.73) .. (419.45,635.17) .. controls (419.45,694.61) and (371.27,742.8) .. (311.83,742.8) .. controls (252.39,742.8) and (204.2,694.61) .. (204.2,635.17) -- cycle ;
\draw  [color={rgb, 255:red, 248; green, 231; blue, 28 }  ,draw opacity=1 ][fill={rgb, 255:red, 126; green, 211; blue, 33 }  ,fill opacity=0.91 ][line width=3]  (241.47,635.17) .. controls (241.47,596.31) and (272.97,564.81) .. (311.83,564.81) .. controls (350.69,564.81) and (382.19,596.31) .. (382.19,635.17) .. controls (382.19,674.03) and (350.69,705.53) .. (311.83,705.53) .. controls (272.97,705.53) and (241.47,674.03) .. (241.47,635.17) -- cycle ;
\draw    (437.5,628.69) -- (470.5,628.69) ;
\draw [shift={(473.5,628.69)}, rotate = 180] [fill={rgb, 255:red, 0; green, 0; blue, 0 }  ][line width=0.08]  [draw opacity=0] (12.5,-6.01) -- (0,0) -- (12.5,6.01) -- (8.3,0) -- cycle    ;
\draw  [color={rgb, 255:red, 144; green, 19; blue, 254 }  ,draw opacity=1 ][fill={rgb, 255:red, 245; green, 166; blue, 35 }  ,fill opacity=1 ][line width=3]  (266.15,635.17) .. controls (266.15,609.94) and (286.6,589.49) .. (311.83,589.49) .. controls (337.06,589.49) and (357.51,609.94) .. (357.51,635.17) .. controls (357.51,660.4) and (337.06,680.85) .. (311.83,680.85) .. controls (286.6,680.85) and (266.15,660.4) .. (266.15,635.17) -- cycle ;
\draw [color={rgb, 255:red, 248; green, 231; blue, 28 }  ,draw opacity=1 ][line width=3]    (527.67,619.5) -- (601.67,690.83) ;
\draw [color={rgb, 255:red, 144; green, 19; blue, 254 }  ,draw opacity=1 ][line width=3]    (601,658.17) -- (533.67,739.5) ;
\draw  [color={rgb, 255:red, 0; green, 0; blue, 0 }  ,draw opacity=1 ][fill={rgb, 255:red, 126; green, 211; blue, 33 }  ,fill opacity=1 ][line width=1.5]  (585.41,668.02) .. controls (588.72,667.97) and (591.44,670.62) .. (591.49,673.93) .. controls (591.54,677.25) and (588.89,679.97) .. (585.57,680.02) .. controls (582.26,680.06) and (579.54,677.41) .. (579.49,674.1) .. controls (579.44,670.79) and (582.09,668.06) .. (585.41,668.02) -- cycle ;
\draw [color={rgb, 255:red, 248; green, 231; blue, 28 }  ,draw opacity=1 ][line width=3]    (524.33,534.33) -- (598.33,605.67) ;
\draw [color={rgb, 255:red, 144; green, 19; blue, 254 }  ,draw opacity=1 ][line width=3]    (597.67,573) -- (530.33,654.33) ;
\draw  [color={rgb, 255:red, 0; green, 0; blue, 0 }  ,draw opacity=1 ][fill={rgb, 255:red, 126; green, 211; blue, 33 }  ,fill opacity=1 ][line width=1.5]  (582.07,582.85) .. controls (585.39,582.8) and (588.11,585.45) .. (588.16,588.77) .. controls (588.2,592.08) and (585.55,594.8) .. (582.24,594.85) .. controls (578.93,594.9) and (576.2,592.25) .. (576.16,588.93) .. controls (576.11,585.62) and (578.76,582.9) .. (582.07,582.85) -- cycle ;
\draw  [color={rgb, 255:red, 0; green, 0; blue, 0 }  ,draw opacity=1 ][fill={rgb, 255:red, 245; green, 166; blue, 35 }  ,fill opacity=1 ][line width=1.5]  (544.75,631.69) .. controls (548.06,631.64) and (550.79,634.29) .. (550.83,637.6) .. controls (550.88,640.92) and (548.23,643.64) .. (544.92,643.69) .. controls (541.6,643.73) and (538.88,641.08) .. (538.83,637.77) .. controls (538.79,634.46) and (541.44,631.73) .. (544.75,631.69) -- cycle ;
\draw [color={rgb, 255:red, 144; green, 19; blue, 254 }  ,draw opacity=1 ][line width=3]    (591,491) -- (523.67,572.33) ;
\draw  [color={rgb, 255:red, 0; green, 0; blue, 0 }  ,draw opacity=1 ][fill={rgb, 255:red, 245; green, 166; blue, 35 }  ,fill opacity=1 ][line width=1.5]  (542.08,546.52) .. controls (545.4,546.48) and (548.12,549.12) .. (548.17,552.44) .. controls (548.21,555.75) and (545.56,558.47) .. (542.25,558.52) .. controls (538.94,558.57) and (536.21,555.92) .. (536.17,552.6) .. controls (536.12,549.29) and (538.77,546.57) .. (542.08,546.52) -- cycle ;
\draw    (141.5,628.69) -- (174.5,628.69) ;
\draw [shift={(177.5,628.69)}, rotate = 180] [fill={rgb, 255:red, 0; green, 0; blue, 0 }  ][line width=0.08]  [draw opacity=0] (12.5,-6.01) -- (0,0) -- (12.5,6.01) -- (8.3,0) -- cycle    ;
\draw [color={rgb, 255:red, 245; green, 166; blue, 35 }  ,draw opacity=1 ][line width=3]    (59.76,500.36) -- (60.5,553.35) ;
\draw [color={rgb, 255:red, 126; green, 211; blue, 33 }  ,draw opacity=1 ][line width=3]    (60.5,553.35) -- (61.24,606.35) ;
\draw  [color={rgb, 255:red, 0; green, 0; blue, 0 }  ,draw opacity=1 ][fill={rgb, 255:red, 144; green, 19; blue, 254 }  ,fill opacity=1 ][line width=1.5]  (59.76,500.36) .. controls (63.07,500.31) and (65.8,502.96) .. (65.84,506.27) .. controls (65.89,509.59) and (63.24,512.31) .. (59.93,512.36) .. controls (56.61,512.4) and (53.89,509.76) .. (53.84,506.44) .. controls (53.8,503.13) and (56.45,500.41) .. (59.76,500.36) -- cycle ;
\draw  [color={rgb, 255:red, 0; green, 0; blue, 0 }  ,draw opacity=1 ][fill={rgb, 255:red, 248; green, 231; blue, 28 }  ,fill opacity=1 ][line width=1.5]  (60.42,547.35) .. controls (63.73,547.31) and (66.45,549.96) .. (66.5,553.27) .. controls (66.55,556.58) and (63.9,559.31) .. (60.58,559.35) .. controls (57.27,559.4) and (54.55,556.75) .. (54.5,553.44) .. controls (54.45,550.12) and (57.1,547.4) .. (60.42,547.35) -- cycle ;
\draw [color={rgb, 255:red, 245; green, 166; blue, 35 }  ,draw opacity=1 ][line width=3]    (60.76,602.69) -- (61.5,655.69) ;
\draw [color={rgb, 255:red, 126; green, 211; blue, 33 }  ,draw opacity=1 ][line width=3]    (61.5,655.69) -- (62.24,708.68) ;
\draw  [color={rgb, 255:red, 0; green, 0; blue, 0 }  ,draw opacity=1 ][fill={rgb, 255:red, 144; green, 19; blue, 254 }  ,fill opacity=1 ][line width=1.5]  (62.07,696.68) .. controls (65.39,696.64) and (68.11,699.29) .. (68.16,702.6) .. controls (68.2,705.91) and (65.55,708.64) .. (62.24,708.68) .. controls (58.93,708.73) and (56.2,706.08) .. (56.16,702.77) .. controls (56.11,699.45) and (58.76,696.73) .. (62.07,696.68) -- cycle ;
\draw  [color={rgb, 255:red, 0; green, 0; blue, 0 }  ,draw opacity=1 ][fill={rgb, 255:red, 248; green, 231; blue, 28 }  ,fill opacity=1 ][line width=1.5]  (61.42,649.69) .. controls (64.73,649.64) and (67.45,652.29) .. (67.5,655.6) .. controls (67.55,658.92) and (64.9,661.64) .. (61.58,661.69) .. controls (58.27,661.73) and (55.55,659.08) .. (55.5,655.77) .. controls (55.45,652.46) and (58.1,649.73) .. (61.42,649.69) -- cycle ;
\draw  [color={rgb, 255:red, 0; green, 0; blue, 0 }  ,draw opacity=1 ][fill={rgb, 255:red, 144; green, 19; blue, 254 }  ,fill opacity=1 ][line width=1.5]  (60.68,596.69) .. controls (63.99,596.65) and (66.71,599.3) .. (66.76,602.61) .. controls (66.81,605.92) and (64.16,608.65) .. (60.84,608.69) .. controls (57.53,608.74) and (54.81,606.09) .. (54.76,602.78) .. controls (54.71,599.46) and (57.36,596.74) .. (60.68,596.69) -- cycle ;
\draw  [color={rgb, 255:red, 248; green, 231; blue, 28 }  ,draw opacity=1 ][fill={rgb, 255:red, 126; green, 211; blue, 33 }  ,fill opacity=1 ][line width=3]  (277.2,635.17) .. controls (277.2,616.05) and (292.71,600.55) .. (311.83,600.55) .. controls (330.95,600.55) and (346.45,616.05) .. (346.45,635.17) .. controls (346.45,654.29) and (330.95,669.8) .. (311.83,669.8) .. controls (292.71,669.8) and (277.2,654.29) .. (277.2,635.17) -- cycle ;
\draw  [color={rgb, 255:red, 144; green, 19; blue, 254 }  ,draw opacity=1 ][fill={rgb, 255:red, 255; green, 255; blue, 255 }  ,fill opacity=1 ][line width=3]  (286.77,635.17) .. controls (286.77,621.33) and (297.99,610.11) .. (311.83,610.11) .. controls (325.67,610.11) and (336.89,621.33) .. (336.89,635.17) .. controls (336.89,649.01) and (325.67,660.23) .. (311.83,660.23) .. controls (297.99,660.23) and (286.77,649.01) .. (286.77,635.17) -- cycle ;

\draw (68.68,447.85) node [anchor=north west][inner sep=0.75pt]  [font=\LARGE,rotate=-89.41]  {$\cdots $};
\draw (70.68,718.51) node [anchor=north west][inner sep=0.75pt]  [font=\LARGE,rotate=-89.41]  {$\cdots $};
\draw (318.35,447.85) node [anchor=north west][inner sep=0.75pt]  [font=\LARGE,rotate=-89.41]  {$\cdots $};
\draw (318.68,612.85) node [anchor=north west][inner sep=0.75pt]  [font=\small,rotate=-89.41]  {$\cdots $};

\draw (582.68,448.51) node [anchor=north west][inner sep=0.75pt]  [font=\LARGE,rotate=-89.41]  {$\cdots $};
\draw (584.01,713.85) node [anchor=north west][inner sep=0.75pt]  [font=\LARGE,rotate=-89.41]  {$\cdots $};

\end{tikzpicture}

\caption[A model of $\G_m$ coming from a decomposition of $\R$]{Left: the decomposition of $\R$ corresponding to a Tate curve. Center: the corresponding formal affinoid cover of $\G_m$. Right: the special fiber of the corresponding formal model.}\label{fig:Tate decomp}
\end{figure}
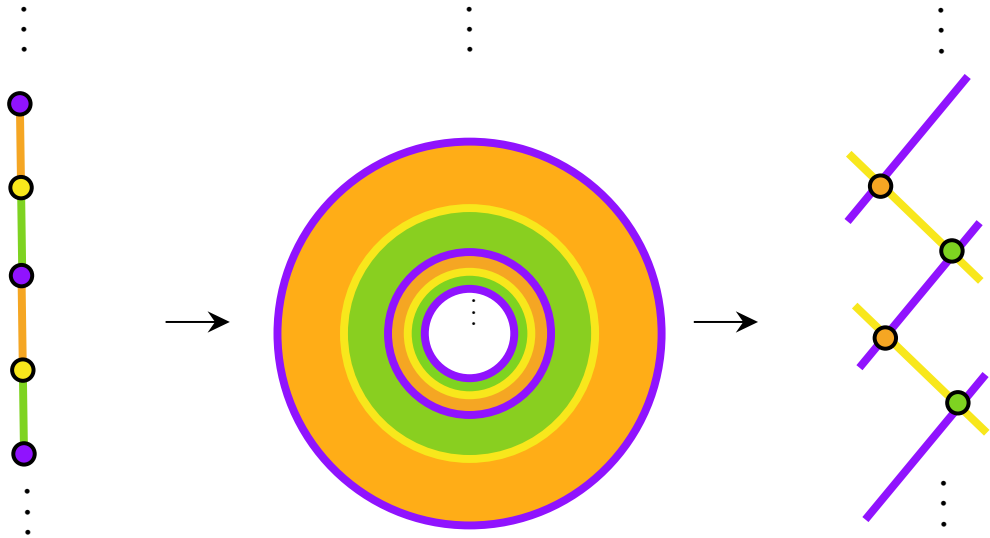
\end{example}

These models are compatible with the multiplication by $p$ map in the following sense:

\begin{proposition}\label{model for [p] on T}
The morphism $[p]:T\rightarrow T$ extends to a morphism $[\mathfrak p]_n:\mathfrak T_{\alpha/p}\rightarrow \mathfrak T_{\alpha}$ of formal models. This morphism reduces mod $\varpi$ to relative Frobenius: there is a canonical isomorphism $\widetilde{\mathfrak T}_{\alpha}\cong \widetilde{\mathfrak T}_{\alpha/p}^{(p)}$, and under this identification, the reduction of $[\mathfrak p]_n$ is relative Frobenius.
\end{proposition}

\begin{proof}
This follows from \cite[Proposition 6.4a]{Gub}, which shows that $[p]$ extends to $[\mathfrak p]_n$ and describes the morphism locally on hypercubes as the morphism $U_{\Delta_{\alpha/p}}\rightarrow U_{\Delta_\alpha}$ sending $x_i\mapsto x_i^p$.
\end{proof}

\begin{corollary}\label{cor:pF for T}
The data $\{\fT_{\alpha/p^n}\}$ is an $F$-tower for $T$, so it gives rise to a perfectoid space $T_\infty\sim\varprojlim_{[p]}T$.
\end{corollary}

We note that all of the models in this tower are already defined over $K^\circ$ (that is, with no field extension needed) as the perfectoidness of $K$ implies that $|K^\times|$ is $p$-divisible. The generic fiber of the formal torus $\overline T_\eta$ described in Notation \ref{not:tori} corresponds to the multi-annulus $U_{\{0\}}\subset T$. The translation action of $\overline T$ on $T$ therefore induces translation by 0 on $\R^r$, so it fixes any hypercube decomposition. This gives:

\begin{proposition}\label{action of formal torus}
For any $n\geq 1$, the action $\overline T_\eta\times T\rightarrow T$ extends to a morphism of formal models $\overline T\times \mathfrak T_{\alpha/p^n}\rightarrow \mathfrak T_{\alpha/p^n}$. This action is compatible with the models for $[p]$ from Proposition \ref{model for [p] on T} in that the following diagram commutes:

\begin{center}
		\begin{tikzcd}
			\overline{T}\times \mathfrak T_{\alpha/p^{n+1}} \arrow[d, "{[p]\times [\mathfrak p]_n}"'] \arrow[r] & \mathfrak T_{\alpha/p^{n+1}} \arrow[d, "{[\mathfrak p]_n}"] \\
			\overline{T}\times \mathfrak T_{\alpha/p^n} \arrow[r] & \mathfrak T_{\alpha/p^n}.
		\end{tikzcd}
		\end{center}
\end{proposition}

\begin{proof}
On the generic fiber, all of the maps are well defined and the diagram commutes by definition. On the special fiber, the action of $\overline{T}$ is trivial, so the diagram commutes by functoriality.
\end{proof}


\subsection{Formal models of abeloids}\label{sub:abeloid models}

Let $A$ be an abeloid variety over $K$. By Theorem \ref{thm:Raynaud abeloid}, after a suitable finite separable extension of $K$, there is a Raynaud extension 
	\begin{equation}\label{Raynaud extension2}
		1\rightarrow T\rightarrow E\xrightarrow{\pi} B\rightarrow 1
	\end{equation}
and a full rank lattice $M\subset E(K)$ such that $A\cong E/M$. By Proposition \ref{prop:Galois descent}, if we construct a perfectoid tilde-limit for $A$ after a finite Galois extension of $K$, then we can descend to get the desired perfectoid tilde-limit over $K$. We can therefore assume that our Raynaud extension and lattice are already defined over $K$.

 By Definition \ref{def:rigid lattice}, the image of $M$ under the log map is a lattice $\Lambda$ in $\R^r$. In this section, we use the formal models of $T$ constructed in Section \ref{sub:torus models} to construct formal models of $E$ that retain an action of $M$. Quotienting out by this action, we obtain formal models of $A$, which we use to construct an $F$-tower for $A$.

Recall from Section \ref{sub:Raynaud extensions} that $E$ was constructed as a pushout $T\times^{\overline{T}_\eta}\overline{E}_\eta$. We construct models for $E$ as pushouts of our models for $T$.

\begin{proposition}\label{prop:formal model for E}
Let $\alpha$ be rational number dividing $\Lambda$ as in Definition \ref{def:lattice division}. For any integer $n>1$, let $\mathfrak T_{\alpha/p^n}$ be the corresponding model for $T$ from Definition \ref{def:hypercube models}. 
	\begin{enumerate}
	\item\label{fm1} The formal scheme $\mathfrak E_{\alpha/p^n}:= \mathfrak T_{\alpha/p^n}\times^{\overline T} \overline E$ is a flat formal model of the rigid space $E$.
	\item\label{fm2} There is a morphism $\fEn\rightarrow \overline B$ which is a fiber bundle and a formal model of $E\rightarrow B$.
	\item\label{fm3} For any integer $n>1$, there is an affine morphism $[\fp]_n:\mathfrak E_{\alpha/p^{n+1}}\rightarrow \fEn$ which is a formal model of $[p]:E\rightarrow E$ and with reduction mod $\varpi$ factoring through relative Frobenius. 
	\end{enumerate}
	
We therefore obtain an $F$-tower $\{\mathfrak E_{\alpha/p^n}\}$ for $E$, giving us a perfectoid space $E_\infty\sim\varprojlim_{[p]} E.$
\end{proposition}

\begin{proof}
This follows by combining Propositions \ref{model for [p] on T} and \ref{action of formal torus} with standard results on fiber bundles, transferred over to to the category of adic spaces. We don't go into detail here as everything is about to be proven more explicitly in Proposition \ref{prop:formal covers for E}, but we want to point out that this result can be obtained by more abstract means.

\end{proof}

The models $\fEn$ come with an action of $M$, giving us models of $A$ that lead to an $F$-tower for $A$. To see this, we explicitly construct formal analytic covers for $\fEn$ and apply Proposition \ref{prop:fas to formal scheme} to show that these covers lead to the desired models. The precise requirements for our formal analytic covers are summarized in the following:

\begin{proposition}\label{prop:formal covers for E}
For every integer $n\geq 1$, these exists a formal analytic cover (in the sense of Definition \ref{def:formal analytic cover}) $\mathcal U_n=\{U_{i,n}\}_{i\in I_n}$ of $E$ such that:
	\begin{enumerate}
	\item\label{fc1} The formal model of $E$ corresponding to $\mathcal U_n$ is $\fEn$.
	\item\label{fc2} For all $m\in M$ and $U_{i,n}$ in $\mathcal U_n$, there is some $U_{j,n}$ in $\mathcal U_n$ such that the translation map $\tau_m:E\rightarrow E$ restricts to an isomorphism $U_{i,n}\rightarrow U_{j,n}$. If $m$ is not the identity, then $U_{i,n}\cap U_{j,n}=\emptyset$.
	\item\label{fc3} For all $n\geq 1$ and $U_{i,n+1}$ in $\mathcal U_{n+1}$, there is some $U_{j, n}$ in $\mathcal U_{n}$ such that the map $[p]:E\rightarrow E$ restricts to a morphism $U_{i,n+1}\rightarrow U_{j,n}$. Furthermore, the mod $\varpi$ reduction of this morphism factors through relative Frobenius.
	\end{enumerate}
\end{proposition}

\begin{proof}
We first construct the cover $\mathcal U_1$. The map $\pi:E\rightarrow B$ sends each $m\in M$ to a point $b_m\in B(K)$ which extends to a point $\overline{b}_m\in\overline{B}(K^\circ)$. As $\overline B$ is a formal abelian scheme, we get an isomorphism $\times\overline{b}_m:\overline B\rightarrow \overline B$. Choose a formal open cover $\{\overline V_{j,1}\}$ of $\overline B$ such that for any $m\in M$, the morphism $\times\overline{b}_m$ permutes the $\overline V_{j,1}$. Note that it is enough to choose generators $m_1,\dots,m_r$ of $M$ and ensure that the corresponding morphisms permute the $\overline V_{j,1}$. Taking the generic fiber, we get an affinoid cover of $B$ which is permuted by multiplication by $b_m$ for any $m\in M$. Refining the cover if needed, we may assume that $\pi$ splits over this cover, so $E$ is locally of the form $T\times V_{j,1}$. We can now define the cover $\mathcal U_1$ of $E$ to contain all opens of the form $U_{{\bf e}\Delta_\alpha}\times V_{j,1}$ for $U_{{\bf e}\Delta_\alpha}$ as in Definition \ref{def:hypercube models} and $V_{j,1}$ in our cover of $B$.

Now to construct $\mathcal U_n$, we first cover $B$ by the preimages $V_{j,n}$ of the affinoid opens $V_{j,1}$ under the morphism $[p^n]:B\rightarrow B$. This is again a cover of $B$ by affinoid opens permuted by morphisms $\tau_{b_m}$. By iterating the following lemma, we see that $\pi$ again splits over each $V_{j,n}$.

\begin{lemma}\label{lem:splitting pullback}
Let $\{ V_j\}$ be a formal analytic cover of the abelian variety $B$ over which the map $\pi: E\rightarrow B$ splits. Then $\pi$ also splits over the cover $\{[p]^{-1}(V_j)\}$. 
\end{lemma}

\begin{proof}
See the second half of the proof of \cite[Proposition 3.8]{AWS}.
\end{proof}

We now define the cover $\mathcal U_n$ of $E$ to contain all opens of the form $U_{{\bf e}\Delta_\alpha/p^n}\times V_{j,n}$. Note that the $\mathcal U_n$ are all formal analytic covers as each open is the product of opens in formal analytic covers for $T$ and $B$. Part (\ref{fc1}) follows as the cover $\mathcal U_n$ was constructed so that the corresponding affine formal schemes are subschemes of $\fEn$. 

For part (\ref{fc2}), we see that $\tau_m: E\rightarrow E$ restricts to an isomorphism \[\tau_m:U_{{\bf e}\Delta_\alpha/p^n}\times V_{j,n}\xrightarrow{\simeq} U_{\ell(m)+{\bf e}\Delta_\alpha/p^n}\tau_{b_m}V_{j,n},\] and the target is again an open in $\mathcal U_n$ as $p^n\ell(m)+\lambda\in \Lambda$ and $b_mV_{j,n}$ is again in our cover of $B$. If $m$ is not the identity, the hypercubes ${\bf e}\Delta_\alpha/p^n$ and $\ell(m)+{\bf e}\Delta_\alpha/p^n$ are disjoint, so the corresponding open subsets of $T$ and therefore $E$ are disjoint.

For part (\ref{fc3}), we see that $[p]:E\rightarrow E$ restricts to a morphism \[ [p]:U_{{\bf e}\Delta_\alpha/p^n}\times V_{j,n}\rightarrow U_{{\bf e}\Delta_\alpha/p^{n-1}}\times V_{j,n-1}\] as multiplication by $p$ acts by scaling the hypercube ${\bf e}\Delta_\alpha/p^n$ and we chose our covers of $B$ so that $[p](V_{j,n})=V_{j,n-1}$. The target is therefore an open in $\mathcal U_{n-1}$ as desired. To see that the reduction factors through relative Frobenius, we use Proposition \ref{model for [p] on T} for the torus part and Example \ref{eg:good reduction} for the abelian part.
\end{proof}

The following theorem now follows directly.

\begin{theorem}\label{thm:formal model for A}
For every integer $n\geq 1$, let $\fEn$ be the model of $E$ constructed in Proposition \ref{prop:formal model for E}. Then
	\begin{enumerate}
	\item\label{fmA1} For every $m\in M$ and $n\geq 1$, the morphism $\tau_m:E\rightarrow E$ extends to a morphism $\tau_{\mathfrak {m}}:\fEn\rightarrow \fEn$.
	\item\label{fmA2} Taking the quotient of $\fEn$ by the action of $M$ defined in (\ref{fmA1}), we obtain models $\fAn$ of $A$ such that the analytic quotient map $E\rightarrow A$ extends to a map of models $\fEn\rightarrow\fAn$.
	\item\label{fmA3} The models for $\tau_m$ in (\ref{fmA1}) commute with the models for $[p]$ in Proposition \ref{prop:formal model for E} in that the commutative diagram 	
	
	\begin{center}
		\begin{tikzcd}
			E \arrow[d, "{[p]}"'] \arrow[r, "\tau_m"] & E \arrow[d, "{[p]}"] \\
			E \arrow[r, "\tau_m^p"] & E
		\end{tikzcd}
	\end{center}
	
	extends to a commutative diagram
	
		\begin{center}
		\begin{tikzcd}
			\mathfrak E_{\alpha/p^{n+1}} \arrow[d, "{[\mathfrak p]_n}"'] \arrow[r, "\tau_{\mathfrak m}"] & \mathfrak E_{\alpha/p^{n+1}} \arrow[d, "{[\mathfrak p]_n}"] \\
			\mathfrak E_{\alpha/p^n} \arrow[r, "\tau_{\mathfrak {m^p}}"] & \mathfrak E_{\alpha/p^n}.
		\end{tikzcd}
		\end{center}
	
	\item\label{fmA4} The models for $[p]:E\rightarrow E$ induce models of $[p]:A\rightarrow A$ in that the commutative diagram
	
	\begin{center}
		\begin{tikzcd}
			E \arrow [r] \arrow[d, "{[p]}"] & A \arrow[d, "{[p]}"] \\
			E \arrow [r] & A
		\end{tikzcd}
	\end{center}

	extends to a commutative diagram of formal models
	\begin{center}
		\begin{tikzcd}
			\mathfrak E_{\alpha/p^{n+1}} \arrow [r] \arrow[d, "{[\mathfrak p]_n}"] & \mathfrak A_{\alpha/p^{n+1}} \arrow[d, "{[\mathfrak p]_n}"] \\
			\fEn \arrow [r] & \fAn.
		\end{tikzcd}
	\end{center}
	
	\item\label{fmA5} The mod $\varpi$ reduction of $[\mathfrak p]_n:\mathfrak A_{\alpha/p^{n+1}}\rightarrow\fAn$ factors through relative Frobenius.
	
	\end{enumerate}
\end{theorem}

\begin{proof}
Parts (\ref{fmA1}), (\ref{fmA3}), and (\ref{fmA4}) are all formal consequences of Proposition \ref{prop:formal covers for E} and Proposition \ref{prop:fas to formal scheme}, obtained by checking that all the maps in question preserve the data of the formal analytic covers. The model $\fAn$ in Part (\ref{fmA2}) is constructed from the formal analytic cover of $A$ obtained by using the translation maps $\tau_m:E\rightarrow E$ to identify the opens $U_{i,n}$ in $\mathcal U_n$ that are sent to the same place in $A$. For Part (\ref{fmA5}), we note that $[\fp]_n: \mathfrak A_{\alpha/p^{n+1}}\rightarrow \fAn$ is locally on the source the same as $[\fp]_n: \mathfrak E_{\alpha/p^{n+1}}\rightarrow \fEn$, and the latter map factors through relative Frobenius.
\end{proof}

Restating the above Theorem, we recover the main theorem of \cite{AWS}.

\begin{theorem}\label{thm:cover of A}
We have an $F$-tower $\{\mathfrak A_{\alpha/p^n}\}$ for $A$, giving us a perfectoid space $A_\infty\sim\varprojlim_{[p]} A.$ 
\end{theorem}

\begin{figure}

\tikzset{every picture/.style={line width=0.75pt}} 

\begin{tikzpicture}[x=0.75pt,y=0.75pt,yscale=-1,xscale=1]

\draw [color={rgb, 255:red, 245; green, 166; blue, 35 }  ,draw opacity=1 ][line width=2.25]    (86.5,482.5) -- (146,482.25) ;
\draw  [fill={rgb, 255:red, 144; green, 19; blue, 254 }  ,fill opacity=1 ] (84.25,482.5) .. controls (84.25,481.26) and (85.26,480.25) .. (86.5,480.25) .. controls (87.74,480.25) and (88.75,481.26) .. (88.75,482.5) .. controls (88.75,483.74) and (87.74,484.75) .. (86.5,484.75) .. controls (85.26,484.75) and (84.25,483.74) .. (84.25,482.5) -- cycle ;
\draw [color={rgb, 255:red, 126; green, 211; blue, 33 }  ,draw opacity=1 ][line width=2.25]    (146,482.25) -- (205.5,482) ;
\draw  [fill={rgb, 255:red, 248; green, 231; blue, 28 }  ,fill opacity=1 ] (143.75,482.25) .. controls (143.75,481.01) and (144.76,480) .. (146,480) .. controls (147.24,480) and (148.25,481.01) .. (148.25,482.25) .. controls (148.25,483.49) and (147.24,484.5) .. (146,484.5) .. controls (144.76,484.5) and (143.75,483.49) .. (143.75,482.25) -- cycle ;
\draw  [fill={rgb, 255:red, 144; green, 19; blue, 254 }  ,fill opacity=1 ] (203.25,482) .. controls (203.25,480.76) and (204.26,479.75) .. (205.5,479.75) .. controls (206.74,479.75) and (207.75,480.76) .. (207.75,482) .. controls (207.75,483.24) and (206.74,484.25) .. (205.5,484.25) .. controls (204.26,484.25) and (203.25,483.24) .. (203.25,482) -- cycle ;
\draw [color={rgb, 255:red, 245; green, 166; blue, 35 }  ,draw opacity=1 ][line width=2.25]    (86.8,252.25) -- (108,252.25) ;
\draw  [fill={rgb, 255:red, 144; green, 19; blue, 254 }  ,fill opacity=1 ] (84.55,252.25) .. controls (84.55,251.01) and (85.56,250) .. (86.8,250) .. controls (88.04,250) and (89.05,251.01) .. (89.05,252.25) .. controls (89.05,253.49) and (88.04,254.5) .. (86.8,254.5) .. controls (85.56,254.5) and (84.55,253.49) .. (84.55,252.25) -- cycle ;
\draw [color={rgb, 255:red, 126; green, 211; blue, 33 }  ,draw opacity=1 ][line width=2.25]    (108,252.25) -- (128.4,252.25) ;
\draw  [fill={rgb, 255:red, 248; green, 231; blue, 28 }  ,fill opacity=1 ] (105.75,252.25) .. controls (105.75,251.01) and (106.76,250) .. (108,250) .. controls (109.24,250) and (110.25,251.01) .. (110.25,252.25) .. controls (110.25,253.49) and (109.24,254.5) .. (108,254.5) .. controls (106.76,254.5) and (105.75,253.49) .. (105.75,252.25) -- cycle ;
\draw [color={rgb, 255:red, 245; green, 166; blue, 35 }  ,draw opacity=1 ][line width=2.25]    (128.4,252.25) -- (149.6,252.25) ;
\draw [color={rgb, 255:red, 126; green, 211; blue, 33 }  ,draw opacity=1 ][line width=2.25]    (149.6,252.25) -- (170,252.25) ;
\draw  [fill={rgb, 255:red, 248; green, 231; blue, 28 }  ,fill opacity=1 ] (145.1,252.25) .. controls (145.1,251.01) and (146.11,250) .. (147.35,250) .. controls (148.59,250) and (149.6,251.01) .. (149.6,252.25) .. controls (149.6,253.49) and (148.59,254.5) .. (147.35,254.5) .. controls (146.11,254.5) and (145.1,253.49) .. (145.1,252.25) -- cycle ;
\draw  [fill={rgb, 255:red, 144; green, 19; blue, 254 }  ,fill opacity=1 ] (126.15,252.25) .. controls (126.15,251.01) and (127.16,250) .. (128.4,250) .. controls (129.64,250) and (130.65,251.01) .. (130.65,252.25) .. controls (130.65,253.49) and (129.64,254.5) .. (128.4,254.5) .. controls (127.16,254.5) and (126.15,253.49) .. (126.15,252.25) -- cycle ;
\draw [color={rgb, 255:red, 245; green, 166; blue, 35 }  ,draw opacity=1 ][line width=2.25]    (170,252.25) -- (191.2,252.25) ;
\draw [color={rgb, 255:red, 126; green, 211; blue, 33 }  ,draw opacity=1 ][line width=2.25]    (191.2,252.25) -- (211.6,252.25) ;
\draw  [fill={rgb, 255:red, 248; green, 231; blue, 28 }  ,fill opacity=1 ] (186.7,252.25) .. controls (186.7,251.01) and (187.71,250) .. (188.95,250) .. controls (190.19,250) and (191.2,251.01) .. (191.2,252.25) .. controls (191.2,253.49) and (190.19,254.5) .. (188.95,254.5) .. controls (187.71,254.5) and (186.7,253.49) .. (186.7,252.25) -- cycle ;
\draw  [fill={rgb, 255:red, 144; green, 19; blue, 254 }  ,fill opacity=1 ] (207.1,252.25) .. controls (207.1,251.01) and (208.11,250) .. (209.35,250) .. controls (210.59,250) and (211.6,251.01) .. (211.6,252.25) .. controls (211.6,253.49) and (210.59,254.5) .. (209.35,254.5) .. controls (208.11,254.5) and (207.1,253.49) .. (207.1,252.25) -- cycle ;
\draw  [fill={rgb, 255:red, 144; green, 19; blue, 254 }  ,fill opacity=1 ] (165.5,252.25) .. controls (165.5,251.01) and (166.51,250) .. (167.75,250) .. controls (168.99,250) and (170,251.01) .. (170,252.25) .. controls (170,253.49) and (168.99,254.5) .. (167.75,254.5) .. controls (166.51,254.5) and (165.5,253.49) .. (165.5,252.25) -- cycle ;
\draw  [color={rgb, 255:red, 144; green, 19; blue, 254 }  ,draw opacity=1 ][fill={rgb, 255:red, 245; green, 166; blue, 35 }  ,fill opacity=1 ][line width=1.5]  (300,482.7) .. controls (300,448.07) and (328.07,420) .. (362.7,420) .. controls (397.33,420) and (425.4,448.07) .. (425.4,482.7) .. controls (425.4,517.33) and (397.33,545.4) .. (362.7,545.4) .. controls (328.07,545.4) and (300,517.33) .. (300,482.7) -- cycle ;
\draw  [color={rgb, 255:red, 248; green, 231; blue, 28 }  ,draw opacity=1 ][fill={rgb, 255:red, 126; green, 211; blue, 33 }  ,fill opacity=1 ][line width=1.5]  (325.72,482.7) .. controls (325.72,462.28) and (342.28,445.73) .. (362.7,445.73) .. controls (383.12,445.73) and (399.68,462.28) .. (399.68,482.7) .. controls (399.68,503.12) and (383.12,519.68) .. (362.7,519.68) .. controls (342.28,519.68) and (325.72,503.12) .. (325.72,482.7) -- cycle ;
\draw  [color={rgb, 255:red, 144; green, 19; blue, 254 }  ,draw opacity=1 ][fill={rgb, 255:red, 255; green, 255; blue, 255 }  ,fill opacity=1 ][line width=1.5]  (343.32,482.7) .. controls (343.32,471.99) and (351.99,463.32) .. (362.7,463.32) .. controls (373.41,463.32) and (382.08,471.99) .. (382.08,482.7) .. controls (382.08,493.41) and (373.41,502.08) .. (362.7,502.08) .. controls (351.99,502.08) and (343.32,493.41) .. (343.32,482.7) -- cycle ;
\draw  [color={rgb, 255:red, 144; green, 19; blue, 254 }  ,draw opacity=1 ][fill={rgb, 255:red, 245; green, 166; blue, 35 }  ,fill opacity=1 ][line width=1.5]  (300,254.7) .. controls (300,220.07) and (328.07,192) .. (362.7,192) .. controls (397.33,192) and (425.4,220.07) .. (425.4,254.7) .. controls (425.4,289.33) and (397.33,317.4) .. (362.7,317.4) .. controls (328.07,317.4) and (300,289.33) .. (300,254.7) -- cycle ;
\draw  [color={rgb, 255:red, 248; green, 231; blue, 28 }  ,draw opacity=1 ][fill={rgb, 255:red, 126; green, 211; blue, 33 }  ,fill opacity=1 ][line width=1.5]  (308.92,254.7) .. controls (308.92,225) and (333,200.92) .. (362.7,200.92) .. controls (392.4,200.92) and (416.48,225) .. (416.48,254.7) .. controls (416.48,284.4) and (392.4,308.48) .. (362.7,308.48) .. controls (333,308.48) and (308.92,284.4) .. (308.92,254.7) -- cycle ;
\draw  [color={rgb, 255:red, 144; green, 19; blue, 254 }  ,draw opacity=1 ][fill={rgb, 255:red, 245; green, 166; blue, 35 }  ,fill opacity=1 ][line width=1.5]  (317.48,254.7) .. controls (317.48,229.73) and (337.73,209.48) .. (362.7,209.48) .. controls (387.67,209.48) and (407.92,229.73) .. (407.92,254.7) .. controls (407.92,279.67) and (387.67,299.92) .. (362.7,299.92) .. controls (337.73,299.92) and (317.48,279.67) .. (317.48,254.7) -- cycle ;
\draw  [color={rgb, 255:red, 248; green, 231; blue, 28 }  ,draw opacity=1 ][fill={rgb, 255:red, 126; green, 211; blue, 33 }  ,fill opacity=1 ][line width=1.5]  (325.72,254.7) .. controls (325.72,234.28) and (342.28,217.72) .. (362.7,217.72) .. controls (383.12,217.72) and (399.67,234.28) .. (399.67,254.7) .. controls (399.67,275.12) and (383.12,291.68) .. (362.7,291.68) .. controls (342.28,291.68) and (325.72,275.12) .. (325.72,254.7) -- cycle ;
\draw  [color={rgb, 255:red, 144; green, 19; blue, 254 }  ,draw opacity=1 ][fill={rgb, 255:red, 245; green, 166; blue, 35 }  ,fill opacity=1 ][line width=1.5]  (332.51,254.7) .. controls (332.51,238.03) and (346.03,224.51) .. (362.7,224.51) .. controls (379.37,224.51) and (392.89,238.03) .. (392.89,254.7) .. controls (392.89,271.37) and (379.37,284.89) .. (362.7,284.89) .. controls (346.03,284.89) and (332.51,271.37) .. (332.51,254.7) -- cycle ;
\draw  [color={rgb, 255:red, 248; green, 231; blue, 28 }  ,draw opacity=1 ][fill={rgb, 255:red, 126; green, 211; blue, 33 }  ,fill opacity=1 ][line width=1.5]  (339.25,254.7) .. controls (339.25,241.75) and (349.75,231.25) .. (362.7,231.25) .. controls (375.65,231.25) and (386.15,241.75) .. (386.15,254.7) .. controls (386.15,267.65) and (375.65,278.15) .. (362.7,278.15) .. controls (349.75,278.15) and (339.25,267.65) .. (339.25,254.7) -- cycle ;
\draw  [color={rgb, 255:red, 144; green, 19; blue, 254 }  ,draw opacity=1 ][fill={rgb, 255:red, 255; green, 255; blue, 255 }  ,fill opacity=1 ][line width=1.5]  (343.32,254.7) .. controls (343.32,243.99) and (351.99,235.32) .. (362.7,235.32) .. controls (373.41,235.32) and (382.08,243.99) .. (382.08,254.7) .. controls (382.08,265.41) and (373.41,274.08) .. (362.7,274.08) .. controls (351.99,274.08) and (343.32,265.41) .. (343.32,254.7) -- cycle ;
\draw [color={rgb, 255:red, 144; green, 19; blue, 254 }  ,draw opacity=1 ][line width=2.25]    (541.33,489.17) .. controls (570.33,519.33) and (611.67,519.33) .. (641.33,489.17) ;
\draw [color={rgb, 255:red, 248; green, 231; blue, 28 }  ,draw opacity=1 ][line width=2.25]    (542,510.5) .. controls (571,470) and (611,471.33) .. (642,510.5) ;
\draw  [fill={rgb, 255:red, 245; green, 166; blue, 35 }  ,fill opacity=1 ] (550.58,499.17) .. controls (550.58,497.92) and (551.59,496.92) .. (552.83,496.92) .. controls (554.08,496.92) and (555.08,497.92) .. (555.08,499.17) .. controls (555.08,500.41) and (554.08,501.42) .. (552.83,501.42) .. controls (551.59,501.42) and (550.58,500.41) .. (550.58,499.17) -- cycle ;
\draw  [fill={rgb, 255:red, 126; green, 211; blue, 33 }  ,fill opacity=1 ] (628.25,499.17) .. controls (628.25,497.92) and (629.26,496.92) .. (630.5,496.92) .. controls (631.74,496.92) and (632.75,497.92) .. (632.75,499.17) .. controls (632.75,500.41) and (631.74,501.42) .. (630.5,501.42) .. controls (629.26,501.42) and (628.25,500.41) .. (628.25,499.17) -- cycle ;
\draw [color={rgb, 255:red, 144; green, 19; blue, 254 }  ,draw opacity=1 ][line width=2.25]    (559.8,282.4) -- (640.6,282.4) ;
\draw [color={rgb, 255:red, 248; green, 231; blue, 28 }  ,draw opacity=1 ][line width=2.25]    (559.93,221.85) -- (641,221.6) ;
\draw [color={rgb, 255:red, 144; green, 19; blue, 254 }  ,draw opacity=1 ][line width=2.25]    (610.2,202.4) -- (645.83,272.4) ;
\draw [color={rgb, 255:red, 144; green, 19; blue, 254 }  ,draw opacity=1 ][line width=2.25]    (555.13,271.85) -- (589.8,202.4) ;
\draw [color={rgb, 255:red, 248; green, 231; blue, 28 }  ,draw opacity=1 ][line width=2.25]    (610.73,301.85) -- (645,231.6) ;
\draw [color={rgb, 255:red, 248; green, 231; blue, 28 }  ,draw opacity=1 ][line width=2.25]    (554.2,232) -- (590.63,302) ;
\draw  [fill={rgb, 255:red, 126; green, 211; blue, 33 }  ,fill opacity=1 ] (617.78,281.97) .. controls (617.78,280.72) and (618.79,279.72) .. (620.03,279.72) .. controls (621.28,279.72) and (622.28,280.72) .. (622.28,281.97) .. controls (622.28,283.21) and (621.28,284.22) .. (620.03,284.22) .. controls (618.79,284.22) and (617.78,283.21) .. (617.78,281.97) -- cycle ;
\draw  [fill={rgb, 255:red, 245; green, 166; blue, 35 }  ,fill opacity=1 ] (578.58,281.97) .. controls (578.58,280.72) and (579.59,279.72) .. (580.83,279.72) .. controls (582.08,279.72) and (583.08,280.72) .. (583.08,281.97) .. controls (583.08,283.21) and (582.08,284.22) .. (580.83,284.22) .. controls (579.59,284.22) and (578.58,283.21) .. (578.58,281.97) -- cycle ;
\draw  [fill={rgb, 255:red, 245; green, 166; blue, 35 }  ,fill opacity=1 ] (632.58,252.37) .. controls (632.58,251.12) and (633.59,250.12) .. (634.83,250.12) .. controls (636.08,250.12) and (637.08,251.12) .. (637.08,252.37) .. controls (637.08,253.61) and (636.08,254.62) .. (634.83,254.62) .. controls (633.59,254.62) and (632.58,253.61) .. (632.58,252.37) -- cycle ;
\draw  [fill={rgb, 255:red, 126; green, 211; blue, 33 }  ,fill opacity=1 ] (562.98,251.97) .. controls (562.98,250.72) and (563.99,249.72) .. (565.23,249.72) .. controls (566.48,249.72) and (567.48,250.72) .. (567.48,251.97) .. controls (567.48,253.21) and (566.48,254.22) .. (565.23,254.22) .. controls (563.99,254.22) and (562.98,253.21) .. (562.98,251.97) -- cycle ;
\draw  [fill={rgb, 255:red, 126; green, 211; blue, 33 }  ,fill opacity=1 ] (618.18,221.57) .. controls (618.18,220.32) and (619.19,219.32) .. (620.43,219.32) .. controls (621.68,219.32) and (622.68,220.32) .. (622.68,221.57) .. controls (622.68,222.81) and (621.68,223.82) .. (620.43,223.82) .. controls (619.19,223.82) and (618.18,222.81) .. (618.18,221.57) -- cycle ;
\draw  [fill={rgb, 255:red, 245; green, 166; blue, 35 }  ,fill opacity=1 ] (578.18,221.97) .. controls (578.18,220.72) and (579.19,219.72) .. (580.43,219.72) .. controls (581.68,219.72) and (582.68,220.72) .. (582.68,221.97) .. controls (582.68,223.21) and (581.68,224.22) .. (580.43,224.22) .. controls (579.19,224.22) and (578.18,223.21) .. (578.18,221.97) -- cycle ;
\draw [line width=1.5]    (149,342.5) -- (149,385.5) ;
\draw [shift={(149,388.5)}, rotate = 270] [color={rgb, 255:red, 0; green, 0; blue, 0 }  ][line width=1.5]    (14.21,-4.28) .. controls (9.04,-1.82) and (4.3,-0.39) .. (0,0) .. controls (4.3,0.39) and (9.04,1.82) .. (14.21,4.28)   ;
\draw [line width=1.5]    (601,342.5) -- (601,385.5) ;
\draw [shift={(601,388.5)}, rotate = 270] [color={rgb, 255:red, 0; green, 0; blue, 0 }  ][line width=1.5]    (14.21,-4.28) .. controls (9.04,-1.82) and (4.3,-0.39) .. (0,0) .. controls (4.3,0.39) and (9.04,1.82) .. (14.21,4.28)   ;
\draw [line width=1.5]    (361,342.5) -- (361,385.5) ;
\draw [shift={(361,388.5)}, rotate = 270] [color={rgb, 255:red, 0; green, 0; blue, 0 }  ][line width=1.5]    (14.21,-4.28) .. controls (9.04,-1.82) and (4.3,-0.39) .. (0,0) .. controls (4.3,0.39) and (9.04,1.82) .. (14.21,4.28)   ;
\draw [line width=1.5]    (149,100) -- (149,143) ;
\draw [shift={(149,146)}, rotate = 270] [color={rgb, 255:red, 0; green, 0; blue, 0 }  ][line width=1.5]    (14.21,-4.28) .. controls (9.04,-1.82) and (4.3,-0.39) .. (0,0) .. controls (4.3,0.39) and (9.04,1.82) .. (14.21,4.28)   ;
\draw [line width=1.5]    (601,100) -- (601,143) ;
\draw [shift={(601,146)}, rotate = 270] [color={rgb, 255:red, 0; green, 0; blue, 0 }  ][line width=1.5]    (14.21,-4.28) .. controls (9.04,-1.82) and (4.3,-0.39) .. (0,0) .. controls (4.3,0.39) and (9.04,1.82) .. (14.21,4.28)   ;
\draw [line width=1.5]    (361,100) -- (361,143) ;
\draw [shift={(361,146)}, rotate = 270] [color={rgb, 255:red, 0; green, 0; blue, 0 }  ][line width=1.5]    (14.21,-4.28) .. controls (9.04,-1.82) and (4.3,-0.39) .. (0,0) .. controls (4.3,0.39) and (9.04,1.82) .. (14.21,4.28)   ;

\draw (155.6,16.5) node [anchor=north west][inner sep=0.75pt]  [font=\LARGE,rotate=-90]  {$\cdots $};
\draw (608.6,16.5) node [anchor=north west][inner sep=0.75pt]  [font=\LARGE,rotate=-90]  {$\cdots $};
\draw (371.6,16.5) node [anchor=north west][inner sep=0.75pt]  [font=\LARGE,rotate=-90]  {$\cdots $};

\end{tikzpicture}

\caption[The first two levels of the $F$-tower for a Tate curve]{The first two levels of the $F$-tower for a Tate curve, extending Figure \ref{fig:Tate model}. Left: a tower of hypercube decompositions of $[1,c]$, coming from the quotient of the decompositions in Corollary \ref{cor:pF for T} by $c^\Z$. Center: The corresponding tower of formal affinoid covers of $\G_m/q^\Z$. Right: The corresponding tower of special fibers. The maps respect the coloring, so the preimage of the orange annulus is the three disjoint orange annuli, and the preimage of the purple $\P^1$ on the special fiber is three disjoint purple $\P^1$s above.}\label{fig:Tate tower}
\end{figure}
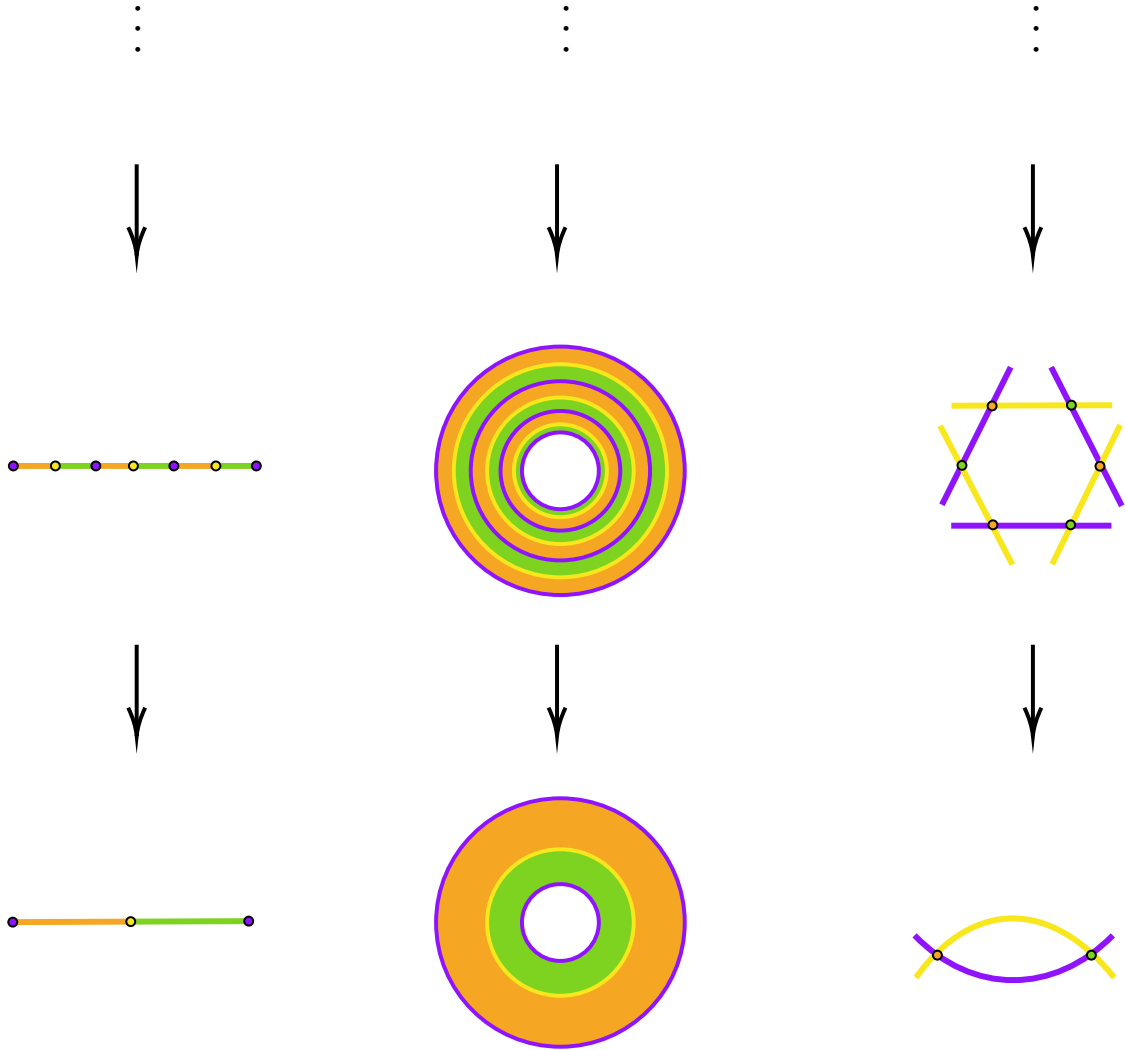


\newpage

\section{Tilting perfectoid covers of abeloids}\label{sec:tilting abeloids}

We continue to assume that $K$ is a perfectoid field with value group $\Gamma\subset\Q$. Given an abeloid $A$ over $K$, in Theorem \ref{thm:cover of A} we constructed an $F$-tower for $A$, giving us a perfectoid cover $A_\infty$ of $A$. In this section, after taking a pro-$p$ extension of $K$, we construct an abeloid $A'$ over $K^\flat$ such that the corresponding perfectoid cover $(A')_\infty$ is the tilt $A_\infty^\flat$ of $A_\infty$. By Lemma \ref{lemma:pF tilt}, if a perfectoid space is built using an $F$-tower over $K^\circ$, its tilt can be built using an $F$-tower over $K^{\flat\circ}$ such that the isomorphism $K^\circ/\varpi\cong K^{\flat\circ}/\varpi^\flat$ induces an isomorphism of the special fibers of the two $F$-towers. 

To construct $A'$, we first use the deformation theory of abelian varieties to construct a Raynaud extension $E'$ over $K^\flat$ such that the special fibers of $\overline E$ and $\overline E'$ agree. This implies that $E_\infty^\flat=E'_\infty$. Choosing a compatible system of $p$th roots of $M$, we get a subgroup $M_\infty\subset E_\infty(K)$. This is where we need the pro-$p$ extension of $K$. This subgroup tilts to a subgroup $(M_\infty)^\flat\subset E_\infty^\flat(K^\flat)=E'_\infty(K^\flat)$, which we show projects down to a lattice $M'\subset E'(K^\flat)$. Taking the quotient, we obtain the desired (non-unique) abeloid $A':=E'/M'$. 

In \cite{HW}, we give an alternate construction of $A'$ in the spirit of \cite{AWS} and explain the relationship with the perfectoid Shimura varieties of \cite{Torsion}.

\subsection{Deformation theory of Raynaud extensions}\label{sub:deformations}

Let $E$ be a Raynaud extension over $K$, recall from Section \ref{sub:Raynaud extensions} that this is equivalent to a strict exact sequence of formal groups over $K^\circ$ 
\begin{equation}\label{eq:formal Raynaud}
1\rightarrow \overline T\rightarrow \overline{E}\xrightarrow{\overline \pi}\overline B\rightarrow 1.
\end{equation}
In this section, we combine deformation theoretic results of Illusie with a theorem of Grothendieck to construct a non-unique Raynaud extension $E'$ over $K^\flat$ such that the mod $\varpi$ special fiber of the sequence (\ref{eq:formal Raynaud}) is isomorphic to the mod $\varpi^\flat$ special fiber of the corresponding sequence for $E'$. 

\begin{proposition}\label{prop:semi-abelian schemes deform}
Let $\widetilde E$ be a semi-abelian scheme over $\Spec(K^\circ/\varpi)$. Then there is a semi-abelian formal scheme $\overline E$ over $\Spf(K^\circ)$ with mod $\varpi$ special fiber $\widetilde E$.
\end{proposition}

\begin{proof}
Given an exact sequence of group schemes over $\Spec(K^\circ/\varpi)$
$$1\rightarrow \widetilde T\rightarrow \widetilde E\rightarrow \widetilde B\rightarrow 1,$$ we deform $\widetilde E$ by first deforming $\widetilde T$ and $\widetilde B$. The split torus $\widetilde T$ lifts to the formal torus $\overline T\cong\mathbb \overline{\G}_{m,K^\circ}^r$. To lift $\widetilde B$ to a formal abelian scheme $\overline B$, we apply the argument of \cite[Theorem 2.2.1]{Oort}, but replace his use of the cotangent bundle $\Theta_{X_0}$ with the cotangent complex as defined in \cite{Ill} or \cite[Tag 08P5]{Stacks}. 

It remains to lift $\widetilde E$ to an element $\overline E$ of $\Ext(\overline B,\overline T)$. By \cite[Theorem A.2.8]{Lut}, $\Ext(\overline B, \overline T)$ is canonically isomorphic to the set of $K^\circ$-valued points of $(\overline{B}^\vee)^r$, where $\overline{B}^\vee$ is the dual abelian scheme of $\overline{B}$. Similarly, $\Ext(\widetilde B, \widetilde T)$ is canonically isomorphic to the set of $K^\circ/\varpi$-valued points of $(\widetilde{B}^\vee)^r$. Under these isomorphisms, the reduction map $\Ext(\overline B,\overline T)\rightarrow\Ext(\widetilde B, \widetilde T)$ is the reduction map $$\overline{B}^\vee(K^\circ)^r\rightarrow \overline{B}^\vee(K^\circ/\varpi)^r\cong \widetilde{B}^\vee(K^\circ/\varpi)^r.$$ As this map is surjective, every semi-abelian $K^\circ/\varpi$-variety $\widetilde E$ lifts to a formal semi-abelian $K^\circ$-variety $\overline E$.
\end{proof}

In particular, if we start with a Raynaud extension $E$ over $K$, we can consider the mod $\varpi$ reduction $\widetilde E$ of $\overline E$ as a semi-abelian scheme over $K^{\flat \circ}/\varpi^\flat$. Applying Proposition \ref{prop:semi-abelian schemes deform}, we deform $\widetilde E$ to a formal scheme $\overline E'$ corresponding to a Raynaud extension $E'$ over $K^\flat$.

\begin{theorem}\label{thm:tilting E}
With $E$ and $E'$ defined as in the previous paragraph, we have $E_\infty^\flat\cong (E')_\infty$.
\end{theorem}

\begin{proof}
As $\overline{E}$ and $\overline{E'}$ both have the same special fiber, the torus parts $T$ and $T'$ of $E$ and $E'$ must have the same rank $r$. The log maps $\ell$ and $\ell'$ for $E/K$ and $E'/K^\flat$ therefore both map to $\R^r$. The images of $\ell$ and $\ell'$ in $\R^r$ are the same as the tilting equivalence identifies the value groups $|K^\times|$ and $|(K^\flat)^\times|$. We can therefore define an $F$-tower for $E'$ using models $\fE'_{\alpha/p^n}$ constructed from the same lattice $\Lambda$ and $\alpha$ dividing $\Lambda$ from Definition \ref{def:hypercube models} as we used for $E$. 

This $F$-tower defines the perfectoid cover $(E')_\infty$ of $E'$. By Lemma \ref{lemma:pF tilt}, it suffices to show that the mod $\varpi^\flat$ special fiber of this $F$-tower is isomorphic to the mod $\varpi$ special fiber of the $F$-tower constructed in Proposition \ref{prop:formal model for E}. Recall that the models $\fEn$ are defined as pushouts $\fTn\times^{\overline{T}}\overline{E},$ and that the morphisms $[\fp]_n:\fE_{\alpha/p^{n+1}}\rightarrow\fEn$ are determined by the morphisms $[\fp]_n:\fT_{\alpha/p^{n+1}}\rightarrow\fTn$, $[p]:\overline T\rightarrow \overline T$, and $[p]:\overline E\rightarrow\overline E$. Moreover, this construction commutes with reducing to the mod $\varpi$ special fiber, so it is enough to check that the special fibers of these morphisms agree with the corresponding morphisms on the $K^\flat$ side. 

The mod $\varpi$ special fiber of $\overline{T}$ is the split torus of rank $r$ over $K^\circ/\varpi$, and the mod $\varpi$ special fiber of $\overline{T'}$ is the split torus of rank $r$ over $K^{\flat\circ}/\varpi^\flat$. As $K^\circ/\varpi\cong K^{\flat\circ}/\varpi^\flat$, these tori are isomorphic, and the multiplication by $p$ maps on $\overline{T}$ and $\overline{T'}$ both reduce to the multiplication by $p$ map on the special fibers, so they also agree. By assumption, the same is true for the special fibers of $\overline{E}$ and $\overline{E'}$. 

We next check that mod $\varpi$ special fiber $\widetilde{\fT}_{\alpha/p^n}$ of a model of the rank $r$ split torus $T$ over $K$ agrees with the mod $\varpi^\flat$ special fiber $\widetilde{\fT'}_{\alpha/p^n}$ of the corresponding model of the rank $r$ split torus $T'$ over $K^\flat$. By the construction of Section \ref{sub:torus models}, this reduces to the simple check that the formal annuli $\An(c^{n+1},c^n)$ over $K^\circ$ and $K^{\flat,\circ}$ have isomorphic special fibers. 

By Proposition \ref{model for [p] on T}, the reduction of the morphisms $[\fp]_n$ on each model is relative Frobenius. Now that we know the reductions of the models over $K$ and $K^\flat$ agree, we conclude that the reductions of the morphisms also agree.
\end{proof}


\subsection{Tilting the $F$-tower for $A$}\label{sub:tilting A}

Let us summarize our current progress. Given an abeloid $A$ over $K$, uniformized as a quotient $E/M$, we constructed an $F$-tower for $A$ by constructing an $F$-tower for $E$ such that the action of the lattice $M$ on $E$ extends to an action on the formal tower. We then constructed a Raynaud extension $E'$ over $K^\flat$ and an $F$-tower for $E'$ such that the mod $\varpi$ and mod $\varpi^\flat$ special fibers of the towers for $E$ and $E'$ agree, implying that $E_\infty^\flat\cong(E')_\infty$. In this section, we construct a lattice $M'\subset E'$ such that the action of $M'$ on $E'$ extends to an action on the formal tower for $E'$, and such that the mod $\varpi^\flat$ special fiber of this action agrees with the mod $\varpi$ special fiber of the action of $M$ on $E$. Letting $A':=E'/M'$ be the corresponding abeloid over $K^\flat$, we conclude that $A_\infty^\flat\cong (A')_\infty$.

To construct $M'$, we fix lattices $M^{1/p^n}\subset E$ such that for any $n\geq 1$, the map $[p]:E\rightarrow E$ restricts to an isomorphism $M^{1/p^n}\rightarrow M^{1/p^{n-1}}$. This is equivalent to fixing a subgroup $M_\infty\subset E_\infty$ such that the projection map $E_\infty\rightarrow E$ to the bottom of the inverse system restricts to an isomorphism $M_\infty\rightarrow M$. This requires us to make a pro-$p$ extension of $K$ to ensure that these lattices exist. We noted in Proposition \ref{prop:G_inf a group} that $E_\infty$ is a perfectoid group, so for point $m_\infty\in M_\infty$, translation by $m_\infty$ induces an isomorphism $\tau_{m_\infty}:E_\infty\rightarrow E_\infty.$ This map fits into the commutative diagram

	\begin{center}
		\begin{equation}\label{roots of m commute}
		\begin{tikzcd}
			E_\infty \arrow [r, "{\tau_{m_\infty}}"] \arrow[dd, bend right = 25, "q_2"] \arrow[ddd, bend right = 35, "q_1", swap] & E_\infty \arrow[dd, bend left = 25, "q_2", swap] \arrow[ddd, bend left = 35, "q_1"] \\
			\vdots \arrow[d, "{[p]}"] & \vdots \arrow[d, "{[p]}", swap] \\
			E \arrow [r, "{\tau_{m^{1/p}}}"] \arrow[d, "{[p]}"] & E \arrow[d, "{[p]}", swap] \\
			E \arrow[r, "{\tau_m}"] & E.
		\end{tikzcd}
		\end{equation}
	\end{center}
	
Here all the rectangles are fiber diagrams, including the large rectangles with $E_\infty$ in the top row. We can therefore think of the choice of $p$th power roots of $m$ as being equivalent to fixing a choice of isomorphism $E_\infty\simeq E\times_{\tau_m,E} E_\infty$. 

By the tilting equivalence, all of the isomorphisms $E_\infty\xrightarrow{\tau_{m_\infty}} E_\infty$ correspond to isomorphisms $E_\infty^\flat\xrightarrow{\tau_{m_\infty}^\flat} E_\infty^\flat$. Here $m_\infty^\flat$ is the image of $m_\infty$ under the homeomorphism $E_\infty(K)\cong E_\infty^\flat(K^\flat)$. In Theorem \ref{thm:tilting E}, we constructed a Raynaud extension $E'$ over $K^\flat$ with an isomorphism $$\iota:E_\infty^\flat\xrightarrow{\cong}(E')_\infty\sim \varprojlim_{[p]} E'.$$ Let $q':E'_\infty\rightarrow E'$ be the induced map to the bottom of the inverse system. We can now define $M'\subset E'$ by tracing $M$ through these maps. Let $s_1:M\rightarrow E_\infty$ denote the section of $q_1|_{M}$ which sends $M$ to $M_\infty$.

\begin{lemma}\label{lemma:M' a lattice}
The lattice $M \subset E$ is mapped bijectively to a lattice $M':=q'_1\circ\iota\circ\flat\circ s_1(M)$ in $E'$ with $\ell'(M')=\ell(M)\subset\R^r$.
\end{lemma}

\begin{proof}
We claim that the following diagram commutes: 
\begin{center}
	\begin{equation}\label{log and tilt diagram}
	\begin{tikzcd}
		E_\infty(K) \arrow[rr, "{\iota\circ\flat}"] \arrow[d, "{q_1}"] & & E'_\infty(K^\flat) \arrow[d, "{q'_1}"] \\
		E(K) \arrow[rd, "{\ell}"] & &  E'(K^\flat) \arrow[ld, "{\ell'}",swap] \\
		& \R^r. &
	\end{tikzcd}
	\end{equation}
\end{center}

Assuming this, the lemma is clear as $\ell\circ q_1$ maps $M_\infty$ bijectively to $\Lambda\subset \R^r$, so the same is true when we traverse the diagram in the other direction. To see commutativity of the diagram, it is enough to fix a point $\bf c$ in $\R^r$ and show that the preimage of $\bf c$ in $E_\infty(K)$ is mapped to the preimage of $\bf c$ in $E'_\infty(K^\flat)$. By our assumption that the value group of $K$ is contained in $\Q$, we can assume that $\bf c$ is in $\Q^r$. 

We can refine the hypercube decompositions of $\R^r$ used to construct $E_\infty$ without changing the Diagram \ref{log and tilt diagram}: this corresponds to choosing finer formal analytic coverings of $E$ and $E'$, which correspond to blowups on the special fiber that aren't noticed on the generic fiber. So we can go back to Definition \ref{def:hypercube models} and replace our choice of rational number $\alpha$ dividing $\Lambda$ with some $\alpha'$ with larger denominator such that $c$ is a vertex of the corresponding decomposition of $\R^r$. The claim now follows as in Theorem \ref{thm:tilting E}: the tilt of the perfectoid affinoid lying over any $U_\Delta$ in the formal cover of $E$ is the perfectoid affinoid lying over the corresponding $U'_\Delta$ in the formal cover of $E'$, and we can now take $\Delta=\{\bf c\}$. 
\end{proof}

Note that in the above Lemma, it was essential that we fixed the maps $q_1$ and $q'_1$ going all the way to the bottom of the inverse systems, and that we defined $E_\infty$ and $E'_\infty$ using the same decomposition of $\R$. If we hadn't done this, our lattices $M$ and $M'$ would be off by a factor of $p^n$. 

We are now ready to prove the main theorem of the section. 

\begin{theorem}[Heuer, W. \cite{HW}]\label{thm:tilting A}
Let  $A=E/M$ be an abeloid over a perfectoid field $K$ with value group contained in $\Q$. Let $A'$ be the abeloid $E'/M'$ over $K^\flat$ for $E'$ as in Theorem \ref{thm:tilting E} and $M'$ as in Lemma \ref{lemma:M' a lattice}, defined after a suitable pro-$p$ extension of $K$. Defining $A_\infty$ and $(A')_\infty$ as in Theorem \ref{thm:cover of A}, we have $A_\infty^\flat\cong(A')_\infty$.
\end{theorem}

\begin{proof}
Let $m$ be an element of $M$, let $m'$ be the corresponding element of $M'$ from Lemma \ref{lemma:M' a lattice}. By part (\ref{fmA3}) of Theorem \ref{thm:formal model for A}, Diagram \ref{roots of m commute} and the analogous diagram for $m'$ extend to diagrams of formal models

\begin{center}
		\begin{equation}\label{formal models of roots of m commute}
		\begin{tikzcd}
			\mathfrak E_\infty \arrow [r, "{\tau_{\mathfrak m_\infty}}"] \arrow[dd, bend right = 25, "q_2"] \arrow[ddd, bend right = 35, "q_1", swap] & \mathfrak E_\infty \arrow[dd, bend left = 25, "q_2", swap] \arrow[ddd, bend left = 35, "q_1"] & & & \mathfrak E'_\infty \arrow [r, "{\tau_{\mathfrak m_\infty^\flat}}"] \arrow[dd, bend right = 25, "q'_2"] \arrow[ddd, bend right = 35, "q'_1 ", swap] & \mathfrak E'_\infty \arrow[dd, bend left = 25, "q'_2", swap] \arrow[ddd, bend left = 35, "q'_1 "] \\
			\vdots \arrow[d, "{[\mathfrak p]_2}"] & \vdots \arrow[d, "{[\mathfrak p]_2}", swap] & & & \vdots \arrow[d, "{[\mathfrak p']_2}"] & \vdots \arrow[d, "{[\mathfrak p']_2}", swap] \\
			\mathfrak E_{\frac{1}{p^2}D} \arrow [r, "{\tau_{\mathfrak m^{1/p}}}"] \arrow[d, "{[\mathfrak p]_1}"] & \mathfrak E_{\frac{1}{p^2}D} \arrow[d, "{[\mathfrak p]_1}", swap] & & & \mathfrak E'_{\frac{1}{p^2}D} \arrow [r, "{\tau_{(\mathfrak m')^{1/p}}}"] \arrow[d, "{[\mathfrak p']_1}"] & \mathfrak E'_{\frac{1}{p^2}D} \arrow[d, "{[\mathfrak p']_1}", swap] \\
			\mathfrak E_{\frac{1}{p}D} \arrow[r, "{\tau_{\mathfrak m}}"] & \mathfrak E_{\frac{1}{p}D} & & & \mathfrak E'_{\frac{1}{p}D} \arrow[r, "{\tau_{\mathfrak m'}}"] & \mathfrak E'_{\frac{1}{p}D}.
		\end{tikzcd}
		\end{equation}
	\end{center}  
	
We claim that the mod $\varpi$ special fiber of the first diagram is isomorphic to the mod $\varpi^\flat$ special fiber of the second. By the tilting equivalence, this is true for the maps $\times\mathfrak m_\infty$ and $\times\mathfrak m^\flat_\infty$. By construction, this is true for all the vertical maps. As all squares are fiber diagrams, we conclude that this must be true for all the other horizontal maps. In particular, as $M'\subset M'^{1/p^n}$ for all $n$, we have actions of $M'$ on every $\mathfrak E'_{\alpha/p^n}$ agreeing with the $M$ action on $\fEn$ on the special fiber. Applying parts (\ref{fmA2}) and (\ref{fmA4}) of Theorem \ref{thm:formal model for A}, we obtain an $F$-tower for $A'$ satisfying the condition in Lemma \ref{lemma:pF tilt}: for all $n\geq 1$, the commutative diagrams

	\begin{center}
		\begin{tikzcd}
			\mathfrak E_{\alpha/p^{n+1}} \arrow [r] \arrow[d, "{[\mathfrak p]_n}"] & \mathfrak A_{\alpha/p^{n+1}} \arrow[d, "{[\mathfrak p]_n}"] & \mathfrak E'_{\alpha/p^{n+1}} \arrow [r] \arrow[d, "{[\mathfrak p']_n}"] & \mathfrak A'_{\alpha/p^{n+1}} \arrow[d, "{[\mathfrak p']_n}"] \\
			\fEn \arrow [r] & \fAn & \mathfrak E'_{\alpha/p^n} \arrow[r] & \mathfrak A'_{\alpha/p^n}
		\end{tikzcd}
	\end{center}

have isomorphic special fibers. We conclude that $A_\infty^\flat\cong A'_\infty$.

\end{proof}

\begin{remark}
In \cite{Heuer_uniform}, Heuer determines all morphisms between perfectoid covers of abeloids. This lets him determine precisely when two abeloids have isomorphic perfectoid covers in \cite[Theorem 1.4]{Heuer_uniform}. Roughly speaking, he shows that up to $p$-isogeny, the special fibers of the Raynaud extensions must agree, and the lattices must be ``$p$-adically close". 
\end{remark}


\section{Line bundles}\label{sec:line bundles}


In this section, we show how to transfer a line bundle $L$ on an abeloid $A/K$ to a line bundle $L'$ on a suitable abeloid $A'/K^\flat$ as constructed in Theorem \ref{thm:tilting A}. This process has two applications. If $A$ is an abelian variety, it has an ample line bundle $L$. We will see that the corresponding $L'$ on $A'$ must also be ample, giving us a refinement of Theorem \ref{thm:tilting A}: if $A$ is an abelian variety, we can choose $A'$ to be an abelian variety as well. In Section \ref{sec:weight-monodromy}, we'll see that moving line bundles from $A$ to $A'$ is the key input in our analogue of Scholze's approximation lemma, which is itself a key step in the proof of weight-monodromy. 

We first give an overview of the classification of line bundles on Raynaud extensions in Section \ref{sub:Raynaud bundles}. We use this in Section \ref{sub:formal bundles} to show that there is some model $\fA_\alpha$ in the $F$-tower for $A$ constructed in Theorem \ref{thm:cover of A} such that $L$ extends to $\fA_\alpha$. This is based off of an argument of Gubler in the totally degenerate case. Having this model gives us a well-defined metric on $L$, which will let us approximate global sections of $L$.

In Section \ref{sub:perfectoid bundles}, we transfer line bundles to the tilt by developing the theory of $\G_{m,\infty}$-torsors on perfectoid spaces. These give rise to to systems of $p$th power roots of line bundles. Any line bundle over a perfectoid space of characteristic $p$ extends to a $\G_{m,\infty}$-torsor, which we can then untilt to get a $\G_{m,\infty}$-torsor in characteristic 0. We show that every line bundle $L$ on $A$ pulls back to a line bundle on $A_\infty$ which arises in this way. 

An important subtlety arises when we try to transfer the data of a line bundle $L$ from $A$ to $A'$. When $L$ is ample, it defines a polarization: an isogeny $\varphi_L:A\rightarrow A^\vee$. Tilting $L$ will therefore involve constructing an isogeny $\varphi_{L'}:A'\rightarrow (A')^\vee$. There are many choices involved in the construction of $A'$ in Theorem \ref{thm:tilting A}. If these choices are made at random for $A$ and then again for $A^\vee$, there is no reason to expect the resulting abelian varieties $A'$ and $(A^\vee)'$ over $K^\flat$ to be duals. To get dual varieties over $K^\flat$, we need to make compatible choices. We use $\G_{m,\infty}$-torsors to do this, constructing and tilting a perfectoid cover of the Poincar\'e bundle over $B\times B^\vee$, and showing that the choices made in this process induce compatible choices for $A$ and $A^\vee$.


\subsection{Line bundles on Raynaud extensions}\label{sub:Raynaud bundles}


We first summarize a bit more background on Raynaud extensions. Let $A$ be an abelian variety over a non-archimedean field $K$. As described in Section \ref{sub:Raynaud extensions}, after taking a finite separable extension of $K$, there is a short exact sequence of rigid groups
\begin{equation}\label{eq:Raynaud sequence}
1\rightarrow T\rightarrow E\xrightarrow{\pi} B\rightarrow 1
\end{equation}
for a rigid torus $T$ of rank $r$ and an abelian variety $B$ with good reduction, and a lattice $h:M\hookrightarrow E$ such that $A\cong E/M$. Let $M^\vee:=\Hom(T,\G_m)$ denote the character group of the torus, so $M^\vee$ is a free abelian group of rank $r$. Let $B^\vee$ denote the dual abelian variety of $B$, let $P_{B\times B^\vee}$ denote the Poincar\`e bundle on $B\times B^\vee$. We write $P_{B\times B^\vee}^\times$ for the associated $\G_m$-torsor.

\begin{proposition}\label{prop:Raynaud morphism}
A Raynaud extension $1\rightarrow T\rightarrow E\xrightarrow{\pi} B\rightarrow 1$ is equivalent to a group homomorphism $\phi^\vee:M^\vee\rightarrow B^\vee$. 
\end{proposition}

\begin{proof}
This follows from \cite[Theorem A.2.8]{Lut} and the discussion before \cite[Definition 6.1.2]{Lut}. We sketch the construction. For any character $m^\vee\in M^\vee$, taking the pushout of Equation (\ref{eq:Raynaud sequence}) by $m^\vee$ gives a commutative diagram

\begin{equation}\label{eq:Raynaud pushout}
	\begin{tikzcd}
	T \arrow[r] \arrow[d, "{m^\vee}"] & E \arrow[r, "{\pi}"] \arrow[d, "{\langle \cdot, m^\vee\rangle}"] & B \arrow[d, equal] \\
	\G_{m} \arrow[r] & P^\times_{B\times\phi^\vee(m^\vee)} \arrow[r]  & B.
	\end{tikzcd}
\end{equation}

This diagram \emph{defines} the image $\phi^\vee(m^\vee)\in B^\vee$ and the map $\langle\cdot, m^\vee\rangle$, as there is an isomorphism $\Ext(B,\G_m)\cong B^\vee$. The line bundle $\phi^\vee(m^\vee)$ must be translation-invariant because of the group law on $E$: any group homomorphism $M^\vee\rightarrow \Pic_B$ will define an extension of $B$ by $T$, but the extension will only admit a group law if the image of lands in $\Pic^0(B)=B^\vee$. See \cite[VII.3, Theorem 6]{AGCF} and the surrounding material for more details, or \cite[Appendix A.2]{Lut}. 

To go in the other direction, we start with a group homomorphism $\phi^\vee:M^\vee\rightarrow B^\vee$. Choosing a basis $m_1^\vee,\dots,m_r^\vee$ of $M^\vee$, we can write 
\begin{equation}\label{eq:E coords}
E\cong P^\times_{B\times\phi^\vee(m_1^\vee)}\times_B\cdots\times_B P^\times_{B\times\phi^\vee(m_r^\vee)}.
\end{equation}
\end{proof}

Once a decomposition for $E$ has been fixed as in (\ref{eq:E coords}), we can describe the map $\langle\cdot, m^\vee\rangle$ more explicitly as follows. Expand $m^\vee=e_1m^\vee_1+\cdots+e_r m^\vee_r$ in terms of our basis. Given $z\in E$, choose an open of $B$ containing $\pi(z)$ over which the $T$-torsor $E$ trivializes, so that we can represent $z$ as an element $(t_1,\dots,t_r)$ in the copy of $T$ over $\pi(z)$. Then \[\langle z,m^\vee\rangle=t_1^{e_1}\cdots t_r^{e_r}\] in the copy of $\G_m$ over $\pi(z)$ in the trivialization of $P^\times_{B\times \phi^\vee(m^\vee)}$. As in the discussion before Definition \ref{def:rigid lattice}, changing the trivialization will not change the absolute value of the $t_i$. We therefore have a well-defined valuation map 
\[\ell: E(K)\rightarrow \R^r; (t_1,\dots,t_r)\mapsto (-\log |t_1|,\dots,-\log |t_r|).\]
The above description shows that $\langle \cdot,m^\vee\rangle$ ``factors over $\R$", giving an affine function $F_{m^\vee}$ so that the following diagram commutes:

\begin{equation}\label{eq:affine fiber}
\begin{tikzcd}
E(K) \arrow[r, "\ell"] \arrow[d, "{\langle \cdot, m^\vee\rangle}"] & \R^r \arrow[d, "{F_{m^\vee}}"] \\
P^\times_{B\times \phi^\vee(m^\vee)} \arrow[r, "\ell"] & \R.
\end{tikzcd}
\end{equation} 

We call $F_{m^\vee}$ the \emph{tropicalization} of $\langle \cdot, m^\vee\rangle$.

We now sketch the construction of the Raynaud extension uniformizing the dual abelian variety $A^\vee$ from the uniformization $A\cong E/M$. Composing the inclusion $h:M\rightarrow E$ with $\pi:E\rightarrow B$ gives a group homomorphism $\phi:M\rightarrow B$. As $(B^\vee)^\vee\cong B$, Proposition \ref{prop:Raynaud morphism} gives a Raynaud extension
\begin{equation}\label{eq:dual Raynaud}
1\rightarrow T^\vee \rightarrow E^\vee\xrightarrow{\pi^\vee} B^\vee\rightarrow 1
\end{equation}
where the dual torus $T^\vee$ is again a rigid torus of rank $r$. By \cite[Proposition 6.1.8]{Lut}, the map $h$ is equivalent to a non-degenerate bilinear form $\langle \cdot,\cdot\rangle:M\times M^\vee\rightarrow P_{B\times B^\vee}^\times$ living over $\phi\times \phi^\vee$. That is, a commutative diagram 
\begin{equation}\label{eq:Poincare form}
\begin{tikzcd}
& P_{B\times B^\vee}^\times \arrow[d] \\
M\times M^\vee \arrow[ur, "{\langle \cdot,\cdot\rangle}"] \arrow[r, "{\phi\times \phi^\vee}"]  & B\times B^\vee
\end{tikzcd}
\end{equation}
such that for any fixed $m^\vee\in M^\vee$, the restriction $\langle \cdot,m^\vee\rangle:M\rightarrow  P^\times_{B\times\phi^\vee(m^\vee)}$ is the restriction of the middle vertical arrow in (\ref{eq:Raynaud pushout}), and the symmetric condition is true for $m\in M$. By non-degenerate bilinear form, we mean that the tropicalization $F:M\times M^\vee\rightarrow P^\times_{B\times B^\vee}\rightarrow \R^r$ which restricts to $F_{m^\vee}$ for any fixed $m^\vee\in M^\vee$ is a non-degenerate bilinear form. By symmetry, we can extract an inclusion $h^\vee:M^\vee\hookrightarrow E^\vee$ from the bilinear form, giving the dual lattice.

By \cite[Theorem 6.3.3]{Lut} the dual lattice uniformizes the dual abelian variety $A^\vee=E^\vee/M^\vee$. We remark that the notation $E^\vee$ may be a little misleading: this is not some canonical dual associated to the semi-abelian variety $E$, it also depends on the lattice $M\subset E$. 

We are now ready to state the classification of line bundles on $A$, collecting many results from \cite[Chapter 6]{Lut}. Every line bundle on $E$ arises as  the pullback along $\pi$ of some line bundle $N$ on $B$. Any line bundle $L$ on $A$ is given by the $M$-linearization of a line bundle $\pi^*(N)$ on $E$: a compatible family of isomorphisms 
\[\{c_m:\pi^*N\rightarrow \tau^*_m\pi^* N: m\in M\}.\]
Here $\tau_m$ is the multiplication by $m$ map, and the compatibility condition is 
\begin{equation}\label{eq:compatibility}
\tau_{m_1}^*(c_{m_2})\cdot c_{m_2}=c_{m_1+m_2}
\end{equation}
for all $m_1,m_2\in M$.  For many choices of line bundle $N$ and lattice $M$, there will not be any such isomorphisms - we need to add two conditions to see when families of $c_m$ exist.

Let $\varphi_N:B\rightarrow B^\vee$ be the homomorphism sending $b\mapsto \tau_b^*N\otimes N^{-1}.$ The first compatibility condition on $M$ and $N$ is that there must be a group homomorphism $\sigma:M\rightarrow M^\vee$ making the following diagram commute:
\begin{equation}\label{eq:bundle diagram}
	\begin{tikzcd}
	M \arrow[r, "\phi"] \arrow[d, "\sigma"] & B \arrow[d, "\varphi_N"] \\
	M^\vee \arrow[r, "\phi^\vee"] & B^\vee.
	\end{tikzcd}
\end{equation}

The second condition is a trivialization $\chi:M\rightarrow \phi^* N$ of the $\G_m$-torsor $\phi^*N$ on $M$ satisfying
\begin{equation}\label{eq:r compatibility}
\chi(m_1+m_2)\otimes \chi(m_1)^{-1}\otimes \chi(m_2)^{-1}=\langle m_1,\sigma(m_2)\rangle
\end{equation}
for all $m_1,m_2\in M$. The factor $\langle m_1,\sigma(m_2)\rangle$ is needed to make translations match up for non-translation-invariant bundles. It defines a symmetric, non-degenerate bilinear form on $M$. Taking valuations, we get a bilinear form 
\begin{equation}\label{eq:line bundle form}
\langle\cdot,\cdot\rangle_\sigma:M\times M\rightarrow \R
\end{equation}
which extends to a bilinear form $F_\sigma:\R^r\times\R^r\rightarrow \R$.

We make $c_m$ more explicit for future use. Choose a basis for $M^\vee$, giving coordinates for $E$ as in (\ref{eq:E coords}), and expand $\sigma(m)=e_1m_1^\vee+\cdots e_r m_r^\vee$. Over a trivialization of $\pi^*(N)$, $c_m$ takes the fiber over a point $z\in E$ corresponding to $(t_1,\dots,t_r)$ in the copy of $T$ over $\pi(z)$, translates it to the fiber over $\tau_m(z)$, and scales the copy of $\G_m$ by the unit $r(m)t_1^{e_1}\cdots t_r^{e_r}$ in $K^\times$. Again taking valuations, we get an affine function $z_m: \R^r\rightarrow \R$ satisfying
\begin{equation}\label{eq:zm}
z_m(\ell(z)):=\ell(c_m(z))=\ell(r(m))+F(\ell(z),\ell(\sigma(m))).
\end{equation}
This is the tropicalization of $c_m$.

This is enough data to classify all line bundles on $A$:

\begin{proposition}[Theorems 6.3.2, 6.4.4 \cite{Lut}]\label{prop:Raynaud line bundles}
A line bundle $L$ on an abelian variety $A$ is determined by the data of a line bundle $N$ on $B$, a group homomorphism $\sigma:M\rightarrow M^\vee$, and a trivialization $\chi:M\rightarrow \phi^*N$ such that Diagram \ref{eq:bundle diagram} commutes and Equation (\ref{eq:r compatibility}) holds. Furthermore, $L$ is ample precisely when $N$ is ample and the bilinear form $\langle\cdot,\cdot\rangle_\sigma$ of Equation (\ref{eq:line bundle form}) is positive definite.
\end{proposition}

\begin{corollary}\label{cor:Raynaud bundle multiplication}
Given a line bundle $L$ on $A$ corresponding to the data $(N,\sigma,r)$ as in Proposition \ref{prop:Raynaud line bundles}, the line bundle $L^{\otimes p}$ corresponds to the data $(N^{\otimes p},p\sigma,\chi^p)$. 
\end{corollary}

\begin{example}\label{eg:ti bundles}
Translation-invariant line bundles on $A$ correspond to the case where $N$ is a translation-invariant line bundle on $B$ and $\sigma$ is the trivial morphism. In this case, $\phi_N$ is trivial, so Diagram \ref{eq:bundle diagram} commutes. The right hand side of Equation \ref{eq:r compatibility} is now 0, so $\chi$ is a group homomorphism $M\rightarrow \phi^*N$. Choosing a basis $m_1,\dots,m_r$ of $M$, we get coordinates of $E^\vee\cong P_{\phi(m_1)\times B^\vee}\times_{B^\vee}\cdots \times_{B^\vee} P_{\phi(m_r)\times B^\vee}$ as in Equation (\ref{eq:E coords}), so that $\chi$ determines a point $(\chi(m_1),\dots,\chi(m_r))\in E^\vee$. So as expected, points of $E^\vee$ give rise to translation-invariant bundles on $A$ via the surjection $E^\vee\rightarrow A^\vee$. 
\end{example}

\begin{example}\label{eg:Tate bundles}
We specialize Proposition \ref{prop:Raynaud line bundles} to the Tate curve $\G_m/q^\Z$. The abelian part of the Raynaud extension $\G_m$ is trivial, so there is no line bundle $N$ to worry about (tori have trivial Picard group). The lattice is one-dimensional, so the group homomorphism $\sigma$ and trivialization $\chi$ are completely determined by the image of the generator $q$. In other words, once we choose an automorphism $c_{q}$ of the trivial line bundle on $\G_m$, everything else will fall into place. 

As the lattice is one-dimensional, $M^\vee=\Hom(\G_m,\G_m)\cong \Z$, where the identity map $m^\vee$ is sent to 1. The homomorphism $\sigma$ sends $q$ to $dm^\vee$ for some integer $d$: this is the degree of the corresponding line bundle $L$. The map $\chi$ can send $q$ to any element $\chi(q)$ in $K^\times$. When $d=0$ and we are dealing with $\Pic^0$, the line bundle is therefore determined by an element of $K^\times$. We get the trivial line bundle precisely when $\chi(q)\in q^\Z$, giving the expected self-duality $\Pic^0(\G_m/q^\Z)\cong \G_m/q^\Z$. 
\end{example}

The takeaway from these examples is that $\chi$ controls $\Pic^0$, $\sigma$ controls the part of the N\'eron-Severi group coming from the lattice $M$, and $N$ controls the part of N\'eron-Severi coming from $B$. 


\subsection{Formal models of line bundles}\label{sub:formal bundles}


As a first step towards moving line bundles from $A$ to $A'$, we need to extend a line bundle $L$ on $A$ to a formal line bundle $\fL_n$ on $\fAn$ for some $n$.  Any line bundle on $B$ (and therefore $E$) extends to a formal bundle on the corresponding formal (semi-)abelian scheme by \cite[Lemma 6.2.2]{Lut}. We still need to extend the $M$-linearization to the formal model. Gubler does this very explicitly in the totally degenerate case in \cite[Proposition 6.6]{Gub}. In this section, we summarize his construction and show that it extends to the general case. We first make precise what we mean by formal model in this context.

\begin{definition}\label{def:formal_line_bundle}
Let $\fL$ be a line bundle on an admissible formal scheme $\fX$ as in Proposition \ref{prop:fas to formal scheme}. Then $\fL$ is trivial over some cover of $\fX$ by formal affine opens $\{\Spf(R_i^\circ)\}$ with transition maps given by compatible units $f_{ij}$ in $(R_{ij}^\circ)^\times$. The \emph{generic fiber} of $\fL$ is the line bundle $L$ on $X=(\fX)_\eta$ which is trivial on the formal analytic cover $\Sp(R_i)$, with transition maps given by the images of the $f_{ij}$ under the maps $R_{ij}^\circ\rightarrow R_{ij}$. In this situation, we say that $\fL$ is a formal model of $L$.
\end{definition}

\begin{remark}\label{remark:formal_line_bundle}
Note that the geometric formal line bundle $\fL$ isn't a formal model of the geometric line bundle $L$, because the generic fiber of a formal line isn't a rigid line.
\end{remark}

Fix a line bundle $L$ on an abelian variety $A\cong E/M$, with associated data $(N,\sigma,r)$ as in Proposition \ref{prop:Raynaud line bundles}. For each $m\in M$, we have the associated affine function $z_m$ from Equation \ref{eq:zm}: the tropicalization of the isomorphism $c_m:\pi^*N\rightarrow \tau_m^*\pi^*N$. Fix a formal analytic cover $\mathcal U_n$ of $E$ as in Proposition \ref{prop:formal covers for E} such that the line bundle $\pi^*N$ is trivial over this cover. Explicitly, we choose a $\pi(M)$-invariant cover $\{\overline V_{j,n}\}$ of $\overline B$ over which $N$ is trivial and a rational number $\alpha$ dividing the lattice. We can then write the typical open of $\mathcal U_n$ as a product $U_{{\bf e}\Delta_{\alpha/p^n}}\times V_{j,n}$ for ${\bf e}\in\Z^r$.  In Theorem \ref{thm:formal model for A}, we saw that the action of $M$ extends to the corresponding formal model $\fEn$ of $E$, so quotienting gives a formal model $\fAn$ of $A$. 

\begin{proposition}\label{prop:formal bundles}
With notation as in the preceding paragraph, formal models $\fLn$ of $L$ on $\fAn$ can be constructed from pairs $(\mathfrak N,f)$ where $\mathfrak N$ is a formal model of $N$ on the formal abelian scheme $\overline B$, and $f:\R^r\rightarrow \R$ is a continuous function such that:
\begin{enumerate}
\item\label{fb1} When restricted to a hypercube $\Delta={\bf e}\Delta_{\alpha/p^n}$ in our cover of $\R^r$, we have $f(u)=m_\Delta\cdot u+c_\Delta$ for some $m_\Delta\in\Z^r$, $c_\Delta\in \Gamma$, where $\cdot$ denotes dot product;
\item\label{fb2} For $\lambda\in \Lambda=\ell(M)$, we have $f(u+\lambda)=f(u)+z_{\lambda}(u)$.
\end{enumerate}
\end{proposition}

\begin{proof}
When $A$ is totally degenerate, this is a weaker version of \cite[Proposition 6.6]{Gub}. Gubler shows in this case that formal models of $L$ are in bijection with these functions $f$ (the line bundle $N$ being trivial in this case). We expect a statement along these lines to hold in general, but have not checked this as we will not need it in this paper. Note that for a given model $\fAn$ and line bundle $L$, there might not be a pair $(\mathfrak N, f)$ satisfying the necessary requirements. We will see that for $n$ large enough, a pair $(\mathfrak N, f)$ must exist. 

To extend $L$ to $\fLn$, we use the data of $(\mathfrak N,f)$ to construct a suitable frame of $L$: units $g_{i,n}$ on each $U_{i,n}$ in the cover such that the transition functions $\frac{g_{i_1,n}}{g_{i_2,n}}$ on the intersection $U_{i_1,n}\cap U_{i_2,n}$ have supremum norm 1, and therefore give transition functions of a formal model. 

We first choose a frame $(h_{j,n})$ of $\mathfrak N$ on the cover $\{\overline V_{j,n}\}$ of $\overline B$. As $\mathfrak N$ is a formal line bundle, the transition functions are formal units and therefore have supremum norm 1. For each open $U_{\Delta}\times V_{j,n}$ in our cover, we choose any $a_\Delta\in K^\times$ with $\ell(a_\Delta)=c_\Delta$. Writing $T=\Spa(K\langle x_1^{\pm 1},\dots,x_r^{\pm 1}\rangle,K^\circ\langle x_1^{\pm 1},\dots,x_r^{\pm 1}\rangle)$ and shortening $x_1^{m_1}\cdots x_r^{m_r}$ to ${\mathbf x}^{\mathbf{m}}$, the restriction of $a_\Delta{\bf x}^{m_\Delta}$ from $T$ to $\Delta$ is a unit.

We claim that setting $g_{i,n}=a_\Delta{\bf x}^{\bf e}\otimes h_{j,n}$ gives an $M$-invariant frame of $\mathfrak N$. As $h_{j,n}$ does not affect the supremum norm, Gubler's argument carries over to this case. By construction, for any point $z\in U_{\Delta_i}\times V_{j_i,n},$ we have $v(g_{i,n}(z))=f(\ell(u))$. As $f$ is continuous, the absolute values of the frame units agree on intersections, so the transition functions have absolute value 1 and give a formal model as desired. By condition (\ref{fb2}), pulling back along $\tau_m$ sends the unit $g_{i,n}$ on $U_{\Delta_i}\times V_{j_i,n}$ to a frame on $\tau_m^*(U_{\Delta_i}\times V_{j_i,n})$ with the correct absolute value, so our frame descends to the quotient.

\end{proof}

\begin{example}\label{eg:formal Tate bundles}
For this example, let $K=\C_p$. We give some examples of formal models of line bundles on the Tate curve $C=\G_m/\langle p^2\rangle$ constructed in Example \ref{eg:Tate curve}. We hope this helps illuminate the role of the function $f$.  As we saw in Example \ref{eg:Tate bundles}, a line bundle $L$ on $\G_m/\langle p^2 \rangle$ is determined by its degree $d$ and an element $\chi(p^2)$ of $K^\times$. 
 
\begin{enumerate}
\item Let $L$ be the trivial bundle, corresponding to $d=0,\chi(p^2)=1$. A formal analytic cover of $C$ over which $L$ is trivial is given by the hypercube cover \[U_{[0,1]}=\Sp(K\langle T,\frac{p}{T}\rangle),U_{[1,2]}=\Sp(K\langle \frac{S}{p},\frac{p^2}{S}\rangle),\] glued along the inner circles $U_{\{1\}}$ by the map sending $T\mapsto S$, and along the outer circles by the map 
\[\tau_{p^2}:U_{\{0\}}\rightarrow U_{\{2\}}\]
\[K\langle \frac{S}{p^2},\frac{p^2}{S}\rangle\rightarrow K\langle T,\frac{1}{T}\rangle; S\mapsto p^2 T.\]

The tropicalization of the (identity) isomorphism $c_{p^2}$ is the zero map $z_{p^2}:\R\rightarrow 0$, so the functions $f:\R\rightarrow \R$ that could lead to formal models of $L$ are periodic and determined by their behavior on the closed interval $[0,2]$. They must be affine on $[0,1]$ and $[1,2]$ with integer slopes, and must satisfy $f(0)=f(2)=c_0$ for some $c_0\in \Q\cong \Gamma$. 

Letting $f(u)=u$ for $u\in[0,1]$ and $f(u)=2-u$ for $u\in [1,2]$, we get a function that will lead to a non-trivial model of the trivial line bundle. A frame corresponding to this $f$ is given by $g_{01}=T$ on $U_{[0,1]}$ and $g_{12}=\frac{p^2}{S}$ on $U_{[1,2]}$. As $L$ is trivial, the restriction maps on the intersections are the same as the gluing maps. Writing everything in terms of $T$, our chosen frame gives the transition map $\frac{g_{01}}{g_{12}}=\frac{T^2}{p^2}$ on $U_{\{1\}}$, which always has absolute value 1 when $|T|=|p|$ - the defining property of $U_{\{1\}}$, so we get a formal unit. Writing the other transition map on $U_{\{0\}}$, we first use the gluing map to send $g_{12}=\frac{p^2}{S}$ to $\frac{1}{T}$, then take the quotient $\frac{g_{01}}{g_{12}}$ to get the transition map $T^2$, which again has absolute value 1 on $U_{\{0\}}$ and is therefore a formal unit. As our transition maps are now made from formal units, we get a formal model of our line bundle. However, this is not the trivial line bundle - the unit $g_{01}=T$ on $U_{[0,1]}$ is not a formal unit as it has absolute value less than 1 on $U_{(0,1]}$

\item Let $L$ be the line bundle corresponding to $d=0,r(p^2)=1+p$. We can use the same formal analytic cover of $C$, function $f$, and frame as in the previous example. Restriction maps of this $L$ on the cover are the normal gluing map on the inner circles $U_{\{1\}}$, but are given by sending $S\mapsto p^2(1+p)T$ on the outer circles. The transition functions corresponding to the previous frame are now $\frac{T^2}{p^2}$ on $U_{\{1\}}$ and $\frac{T^2}{1+p}$ on $U_{\{0\}}$. 

\item\label{tb3} Let $L$ be the line bundle corresponding to $d=0, r(p^2)=p^{1/p}$. The formal analytic cover from the previous examples still trivializes $L$, but now there is no function $f$ satisfying the conditions of Proposition \ref{prop:formal bundles} and therefore no formal line bundle extending $L$ to the corresponding model. In this case, the tropicalization $z_{p^2}$ of $c_{p^2}$ is the constant function sending $u$ to $1/p$, so we must have $f(2)=f(0)+1/p$ by condition (\ref{fb2}) of Proposition \ref{prop:formal bundles}. But by condition (\ref{fb1}), $f$ must be continuous and linear on $[0,1]$ and $[1,2]$ with integer slopes, so this is impossible. 

To fix this, we refine the hypercube decomposition used to create the formal analytic cover, instead using the affinoids corresponding to the intervals $[\frac{\lambda}{p},\frac{\lambda+1}{p}]$ for $\lambda\in \{0,\dots,2p-1\}$. We can now choose $f$ to have slope 1 on the interval $[0,1/p]$ and slope 0 on $[1/p,2].$ 
\end{enumerate}
\end{example}

As the tropicalization maps $z_m$ all take values in $v(K)=\Gamma\subset \R$, if we assume that $\Gamma\subset \Q$, Example \ref{eg:Tate bundles}(\ref{tb3}) is the only reason the function $f$ might not exist. Refining the formal analytic cover will always give a model over which $f$ exists. As any line bundle on the abelian variety $B$ extends to a formal line bundle on the formal abelian scheme $\overline B$, we can extend line bundles to the formal models in our towers, completing our main goal of the section.

\begin{proposition}\label{prop:formal alpha}
A line bundle $L$ on $A$ extends to the formal model $\fA_\alpha$ if for every $m\in M$, $v(\chi(m))$ is an integer multiple of $\alpha$.  
\end{proposition}

Combining this with Corollary \ref{cor:Raynaud bundle multiplication}, we get:

\begin{corollary}\label{cor:formal roots}
If the line bundle $L^{\otimes p}$ on $A$ extends to the formal model $\fA_\alpha$, then $L$ extends to the model $\fA_{\alpha/p}$. 
\end{corollary}


\subsection{Perfectoid covers of line bundles}\label{sub:perfectoid bundles}


In the remainder of this section, we explain how to take a line bundle $L$ on $A$ and construct a line bundle $L'$ on a suitable $A'$ with the same special fiber. One might expect this to follow directly from applying Proposition \ref{prop:formal bundles} to $A'$, but transferring the homomorphism $\sigma:M\rightarrow M^\vee$ to the tilt side such that Diagram \ref{eq:bundle diagram} still commutes seems quite subtle. When we tilted lattices in Section \ref{sub:tilting A}, we made an arbitrary choice of $p$th roots $M^{1/p^n}$ of $M$. If the choices we make when tilting the dual lattice $M^\vee$ aren't compatible with the choices we make when tilting $M$, we won't get nice duality properties on the tilt side. To solve this issue, we need to tilt both $M$ and $M^\vee$ simultaneously. We saw in Diagram \ref{eq:dual Raynaud} that the lattice $M\times M^\vee$ injects into the Poincare torsor $P^\times_{B\times B^\vee}$. We will construct and tilt a perfectoid cover of this torsor, with $M\times M^\vee$ inside it, and construct both $A'$ and $A'^\vee$ from this.

This will require us to work with perfectoid covers of line bundles, extending $\G_m$-torsors on $A$ to $\G_{m,\infty}$-torsors on $A_\infty$. This type of construction is very closely related to the arguments in \cite{DH}. Perfectoid covers of line bundles are already used implicitly in Scholze's approximation lemma \cite[Proposition 8.7]{PS}, so we will also need them when we generalize it from toric varieties to abelian varieties in Proposition \ref{prop:approx}.
 

In this subsection, we give some generalities on $\G_{m,\infty}$-torsors. 

\begin{definition}\label{def:Gm infty torsor}
A $\G_m$-torsor on an adic space $X$ is a morphism $\pi:L^\times\rightarrow X$ of adic spaces with an action $\G_m\times L^\times\rightarrow L^\times$ over $X$, so that over some open cover $\{U_i\}$ of $X$, we have $\G_m$-equivariant isomorphisms $\pi^{-1}(U_i)\cong \G_m\times U_i$. A $\G_{m,\infty}$-torsor is defined by replacing $\G_m$ in the above definition by the perfectoid cover $\G_{m,\infty}\sim\varprojlim_{[p]}\G_m$ constructed in Corollary \ref{cor:pF for T}.
\end{definition}

There are correspondences between $\G_m$-torsors $L^\times$, line bundles $L$, and invertible sheaves $\mathcal L$. Given an invertible sheaf $\mathcal L$ on a rigid space $X$, the corresponding $\G_m$-torsor $L^\times$ is the analytification of the scheme \[\underline{\Spec}(\bigoplus_{i\in\Z} \mathcal L^{\otimes i}).\]

Taking the pushout along the map $[p]:\G_m\rightarrow\G_m$ gives a map of $\G_m$-torsors
\begin{center}
\begin{tikzcd}
L^\times \arrow[r] \arrow[d] & X \arrow[d, equal] \\
(L^{\otimes p})^{\times} \arrow[r] & X,
\end{tikzcd}
\end{center} 
where $(L^{\otimes p})^\times$ is the torsor associated to the $p$th power of the line bundle $L$. Locally on $X$, this diagram looks like 

\begin{center}
\begin{tikzcd}
\G_m\times U_i \arrow[r] \arrow[d, "{[p]\times id}"] & U_i \arrow[d, equal] \\
\G_m\times U_i \arrow[r] & U_i.
\end{tikzcd}
\end{center}

The induced map on symmetric algebras is the analytification of the inclusion 
\begin{equation}\label{eq:p power sections}
\bigoplus_{i\in\Z}\mathcal L^{\otimes pi}\hookrightarrow\bigoplus_{i\in\Z} \mathcal L^{\otimes i}.
\end{equation}

Now assume $X_\infty$ is a perfectoid space, and let $(L_n)_{n\in \N}$ be a family of line bundles on $X_\infty$ with $L_{n+1}^{\otimes p}=L_n$. We would like to use this data to build a $\G_{m,\infty}$-torsor $L_\infty^\times$ at the top of the diagram

\begin{center}
\begin{equation}\label{eq:Gminfty diagram}
\begin{tikzcd}
 L^\times_\infty \arrow[r]  & X_\infty \\
\vdots \arrow[d]  & \vdots \arrow[d, equal] \\
L^\times_2 \arrow[r] \arrow[d] & X_\infty \arrow[d, equal] \\
L^\times_1 \arrow[r] & X_\infty.
\end{tikzcd}
\end{equation}
\end{center}

This can be done if we find an open cover $\{U_{i,\infty}\}$ of $X_\infty$ over which all of the $L^\times_i$ are compatibly locally split. In this case, $L^\times_\infty$ is locally a product $\G_{m,\infty}\times U_\infty$. This will show that $L^\times_\infty$ is perfectoid, and by Lemma \ref{lemma:product tilde} it will be a tilde-limit for the tower of $L^\times_i$. More explicitly, we can extend a $\G_m$-torsor $L^\times$ on a perfectoid space $X$ to a $\G_{m,\infty}$-torsor $L_\infty^\times$ if we can write a cocycle of transition functions $(g_{ij})$ of $L$ for some open cover $\{U_{i,\infty}\}$ for which the elements $g_{ij}\in\Gamma(U_{i,\infty}\cap U_{j,\infty})$ all have a compatible sequence of $p$th power roots $g_{ij},g_{ij}^{1/p},g_{ij}^{1/p^2},\dots$, as we can then define $L^{1/p^n}$ using the cocycle $(g_{ij}^{1/p^n})$. We can always do this in characteristic $p$ as perfectoid $K^\flat$ algebras are perfect, giving the following:

\begin{lemma}\label{lemma:Gminfty exists char p}
If $X'_\infty$ is a perfectoid space over a perfectoid field $K^\flat$ of characteristic $p$, then every $\G_{m,K^\flat}$-torsor $L'^\times$ extends to a unique $\G_{m,\infty,K^\flat}$-torsor $L'^\times_\infty$.
\end{lemma}

Given a perfectoid space $X_\infty$ with tilt $X_\infty^\flat$, and a suggestively named $\G_{m,\infty}$-torsor $L^{\flat,\times}_\infty$, we can untilt the morphism 
\[L^{\flat,\times}_\infty \rightarrow X_\infty^\flat\]
to get a sequence
\[ L^\times_\infty \rightarrow X_\infty.\]
The bottom morphism defines a $\G_{m,\infty}$-torsor $L^\times_\infty$ on $X_\infty$ living over a $\G_m$-torsor $L$ on $X_\infty$. This can be seen either by untilting the local splitting of $L^{\flat,\times}_\infty$, or by noting that the sharp map sends the cocycle $(g_{ij})$ defining $L'$ to a cocycle $(g_{ij}^\sharp)$ defining a line bundle $L$ on $X_\infty$. This recovers the map $\theta_0: \Pic(X^\flat_\infty)\rightarrow\Pic(X_\infty)$ constructed in \cite[Section 5.1]{DH}. With this interpretation, we see that Dorfsman-Hopkins's ``cohomological untilting" map on line bundles corresponds to literal untilting of the associated $\G_{m,\infty}$-torsors (see also \cite[Lemma 3.24]{Survey}). 

\begin{remark}
Dorfsman-Hopkins also studies the map $\theta:\Pic X^\flat\rightarrow \varprojlim_{\times p} \Pic X$. He shows in \cite[Theorem 1.2]{DH_untilting} that this map is an isomorphism when $\Gamma(X,\cO_X)$ is a perfectoid ring. Heuer gives an example \cite[Section 6.3]{BenGood} of a perfectoid space where this map fails to be bijective. 
\end{remark}

We will show that in our situation, the image of $\theta_0$ contains all of the line bundles on $A_\infty$ which are pulled back from finite level. This lets us transfer bundles from $A$ to $A'$. We start with the good reduction case.

\begin{proposition}\label{prop:good bundles}
Let $B$ be an abelian variety with good reduction, let $N$ be a line bundle on $B$. Then the $\G_m$-torsor $q_1^*N^\times$ on $B_\infty$ extends to a $\G_{m,\infty}$-torsor.
\end{proposition}

\begin{proof}
This follows directly from Heuer's more general result \cite[Theorem 4.1]{BenGood} which completely classifies $\Pic(B_\infty)$. As we do not need the full power of his result, we can give a shorter proof.

We need to construct $p$th power roots of $q_1^* N$ along with an open cover $\{U_{i,\infty}\}$ of $B_\infty$ over which all of these line bundles are trivial. We first construct the $p$th power roots. By the Theorem of the cube, for any integer $n$, we have
 \[[n]^* N\cong N^{\frac{n^2+n}{2}}\otimes[-1]^* N^{\frac{n^2-n}{2}}\] (see e.g. {\cite[Corollary 6.6]{Mil}}). So when $p\neq 2$, $[p]^* N$ is the $p$th power of the line bundle $N^{\frac{p+1}{2}}\otimes [-1]^*N^{\frac{p-1}{2}}$. In this case, the pullback $q_2^*(N^{\frac{p+1}{2}}\otimes [-1]^*N^{\frac{p-1}{2}})$ is the desired $p$th root of $q_1^*(N)$. Iterating this process, the $p^n$th root of $q_1^*(N)$ comes from level $n+1$ of the inverse system. If $p=2$, $[4]^*N$ is the square of the line bundle $N^5\otimes [-1]^*N^3$. The argument will therefore follow similarly in this case, but the $2^n$th root now comes from level $2n+1$.
 
To find an open cover of $B_\infty$ over which all of these line bundles are trivial, we pass to the special fiber. The argument is similar to that of Lemma \ref{lem:splitting pullback}, which handles the case of translation-invariant line bundles. By \cite[Lemma 6.2.2]{Lut}, every line bundle $N$ on $B$ extends to a formal bundle $\overline N$ on the formal abelian scheme $\overline B/K^\circ$ and then reduces mod $\varpi$ to a line bundle $\widetilde N$ on the abelian scheme $\widetilde B$ over $(K^\circ/\varpi).$ 

On the special fiber, our tower becomes an inverse system of qcqs schemes over $(K^\circ/\varpi)$ with affine transition maps $\widetilde B_{\infty}=\varprojlim_{[p]} \widetilde B$. By \cite[Tag 0B8W]{Stacks}, we have 
\[\Pic(\widetilde B_{\infty})=\varinjlim_{[p]^*} \Pic(\widetilde B).\] In particular, any trivialization of the special fiber of some line bundle $(q_1^*N)^{1/p^n}$ is pulled back from a trivialization on some finite level. By smoothness, this trivialization lifts to a trivialization on the formal scheme, and therefore the corresponding rigid space as in \cite[Proposition A.2.5]{Lut}. It is therefore enough to trivialize all our $p$th roots on a cover of $\widetilde B_\infty$. 

To do this, we work our way over to the tilt. If the degree of $N$ is prime to $p$, then a theorem of Grothendieck (see \cite[Theorem 2.4.1]{Oort} with the added input of the cotangent complex as in Section \ref{sub:deformations}) shows that there is some lift of $\widetilde B$ to a formal abelian scheme $\overline{B'}$ over $K^{\flat\circ}$ such that $\widetilde N$ lifts to a line bundle $\overline{N'}$ on $\overline{B'}$. By Lemma \ref{lemma:Gminfty exists char p}, the generic fiber $N'$ pulls back to a line bundle $q'^*_0(N')$ on $B'_\infty$ which extends to a $\G_{m,\infty,K^\flat}$-torsor. This gives a cover trivializing all the $p$th roots of $q'^*_0(N')$, which reduces to a cover trivializing the $p$th roots of the special fiber as desired. 

If the degree is not prime to $p$, we can (perhaps after a finite extension of $K$) find some isogeny $B\rightarrow B_1$ of $p$-power degree such that $N$ descends to a line bundle $N_1$ with degree prime to $p$. By Lemma \ref{lemma:isogeny covers} below, $B_\infty$ is isomorphic to $(B_1)_\infty$, so we can run the above argument for $N_1$ to get the desired result. 
\end{proof}

\begin{lemma}\label{lemma:isogeny covers} 
Let $f:A\rightarrow A_1$ be an isogeny of abelian varieties over $K$ with $p$-power degree. Then the perfectoid covers $A_\infty$ and $(A_1)_\infty$ constructed in Theorem \ref{thm:cover of A} are isomorphic.
\end{lemma}

\begin{proof}
Let $f^\vee$ be the dual isogeny, so we have $f^\vee\cdot f=[p^k]$ for some positive integer $k$. Then we have a commutative diagram:
	\begin{center}
	\begin{tikzcd}
	&  \dots \arrow[r] & A_1 \arrow[rd, "f^\vee"] \arrow[rr, "{[p^k]}"] & & A_1 \arrow[rd, "f^\vee"] & \\
	\dots \arrow[r] & A \arrow[rr, "{[p^k]}"] \arrow[ru, "f"] & & A \arrow[rr, "{[p^k]}"] \arrow[ru, "f"] & & A.
	\end{tikzcd}
	\end{center}
	
This diagram induces morphisms $f_\infty: A_\infty\rightarrow (A_1)_\infty$ and $f^\vee_\infty: (A_1)_\infty\rightarrow A_\infty$ which compose to $[p^k]: A_\infty\rightarrow A_\infty$ and $[p^k]: (A_1)_\infty\rightarrow (A_1)_\infty$. As the latter maps are isomorphisms, we are done.
\end{proof}

\begin{example}
When $N$ is translation-invariant, the corresponding $\G_m$-torsor is a semi-abelian variety. The corresponding $\G_{m,\infty}$-torsor in this case is exactly the one constructed in Proposition \ref{prop:formal model for E}.
\end{example}

\subsection{Tilting perfectoid covers of line bundles}\label{sub:tilting bundles}

We now apply the theory of the previous subsection to move from line bundles over an abeloid over $K$ to line bundles on a suitable abeloid $A'$ over $K^\flat$ as constructed in Theorem \ref{thm:tilting A}. Applying Proposition \ref{prop:good bundles} to the Poincar\'e torsor $P_{B\times B^\vee}^\times$ on the abelian variety $B\times B^\vee$, we get our perfectoid cover.

\begin{corollary}\label{cor:P_inf-construction}
There is a $\mathbb G_{m,\infty}$-torsor $ P^\times_\infty\rightarrow (B\times B^\vee)_\infty\cong B_\infty\times B^\vee_\infty$ fitting into the commutative diagram of adic spaces

\begin{equation}\label{eqn:P_inf-diagram}
\begin{tikzcd}
 P^\times_\infty\arrow[r] \arrow[d] & B_\infty\times B^\vee_\infty \arrow[d, equal] \\
q_1^*(P_{B\times B^\vee}^\times) \arrow[r] \arrow[d] & B_\infty\times B^\vee_\infty \arrow[d, "{q_1}"] \\
P_{B\times B^\vee}^\times \arrow[r] & B\times B^\vee.
\end{tikzcd}
\end{equation}
\end{corollary}

\begin{corollary}\label{cor:P_inf-fibers}
Fix a point $x^\vee_\infty\in B_\infty^\vee$, let $x^\vee:=q_1(x^\vee_\infty)\in B^\vee$, so \[ P^\times_{x^\vee_\infty}:= P^\times_\infty\times_{B_\infty\times B_\infty^\vee} (B_\infty\times x^\vee_\infty)\] is a $(\G_{m,K})_\infty$-torsor over $B_\infty$. Then $ P^\times_{x^\vee_\infty}\sim \varprojlim_{[p]}  P^\times_{B\times x^\vee}$. The symmetric statement is true for $x_\infty\in B_\infty$, and these give compatible partial group laws which project down to the usual partial group laws on $P^\times_{B\times B^\vee}$.
\end{corollary}

\begin{proof}
Add the map $B\times x^\vee\rightarrow B\times B^\vee$ to (\ref{eqn:P_inf-diagram}), take fiber products everywhere, and notice that we recover precisely the diagram for $ P_{x_\infty^\vee}$.
\end{proof}

This construction tilts nicely. As $B$ has good reduction, it extends to an abelian scheme $\overline B/K^\circ$ with special fiber $\widetilde B/(K^\circ/\varpi)$. Using the isomorphism $K^\circ/\varpi\cong K^{\flat\circ}/\varpi^\flat$, we lift $\widetilde B$ to a formal abelian scheme $\overline B'/K^{\flat\circ}$ with adic generic fiber $B'/K^\flat$. Then we have $B_\infty^\flat\cong (B')_\infty$ as the special fibers agree, and $(B^\vee)_\infty^\flat\cong (B'^\vee)_\infty$ as taking duals commutes with taking the special fiber. 

By \cite[Theorem 6.1.1]{Lut}, the Poincar\'e bundle $ P_{B\times B^\vee}$ extends to a formal line bundle on $\overline{B\times B^\vee}$ with the same special fiber as the Poincar\'e bundle $ P_{B'\times B'^\vee}$ on $B'\times B'^\vee$. Let $ P'^\times_\infty$ be the $(\G_{m,K^\flat})_\infty$-torsor over the Poincar\'e bundle on $B'\times B'^\vee$. Now returning to Diagram \ref{eqn:P_inf-diagram}, we have formal models for the spaces making up the outer columns and bottom row. As the bottom right square is a pullback and the top left square is a pushout, this is enough to give formal models for the middle column as well. As all of this commutes with taking the special fiber, we conclude:

\begin{proposition}\label{prop:Poincare-tilt}
There is an isomorphism $ (P^\times_\infty)^\flat\cong  P'^\times_\infty$.
\end{proposition}

Restricting this isomorphism to the situation of Corollary \ref{cor:P_inf-fibers} and noting that the fiber diagram of perfectoid spaces

\begin{tikzcd}
 P^\times_{x_\infty^\vee}\arrow[r] \arrow[d] &  P^\times_\infty \arrow[d] \\
 B_\infty\times x_\infty^\vee \arrow[r] &  B_\infty\times B_\infty^\vee
\end{tikzcd} tilts to the diagram \begin{tikzcd}
 (P^\times_{x_\infty^\vee})^\flat\arrow[r] \arrow[d] &  P'^\times_\infty \arrow[d] \\
 B'_\infty\times (x_\infty^\vee)^\flat \arrow[r] &  B'_\infty\times B_\infty'^\vee,
\end{tikzcd}
we obtain

\begin{corollary}\label{cor:Poincare-fiber-tilt}
For any point $x_\infty^\vee\in B_\infty^\vee$, there is an isomorphism $( P^\times_{x_\infty^\vee})^\flat\cong  P^\times_{(x_\infty^\vee)^\flat}$.
\end{corollary}

We have now constructed a perfectoid cover of the vertical map in Diagram \ref{eq:Poincare form} and showed that it tilts to an analogous map over $K^\flat$ .To complete the perfectoid version of this diagram we must lift and tilt the bilinear form $\langle \cdot,\cdot\rangle$. This requires us to make choices, so we then check that $A_\infty$ and $(A^\vee)_\infty$ are independent of these choices. 

Fix a basis $m_1,\dots,m_r$ for $M$ and $m_1^\vee,\dots,m_r^\vee$ for $M^\vee$. We first choose group homomorphisms $\phi_\infty: M\rightarrow B_\infty$ and $\phi_\infty^\vee: M^\vee\rightarrow B_\infty^\vee$ lifting $\phi$ and $\phi^\vee$ by picking any lift on our basis elements and then extending from there. Concretely, the element $\phi_\infty(m)$ corresponds to a sequence $(\phi(m), \frac{\phi(m)}{p}, \frac{\phi(m)}{p^2},\dots)$ of elements of $B$ where $[p](\frac{\phi(m)}{p^i})=\frac{\phi(m)}{p^{i-1}}$ for all $i>0$. We then choose a bilinear form $\langle \cdot,\cdot\rangle_\infty\rightarrow P^\times_\infty$ lifting $\langle \cdot,\cdot\rangle$ by picking any lift on the elements $m_i\times m^\vee_j$ which lies over $\phi_\infty(m_i)\times \phi_\infty^\vee(m^\vee_j)$, then extending bilinearly to all of $M\times M^\vee$ (using the compatibility of the partial group laws). This gives us the diagram

\begin{equation}\label{eq:Poincare form-inf}
\begin{tikzcd}
& &  P^\times_\infty \arrow[ld] \arrow[dd] \\
M\times M^\vee \arrow[rru, "{\langle \cdot,\cdot\rangle_\infty}"] \arrow[r, "{\phi_\infty\times\phi^\vee_\infty}", swap] \arrow[dd, equal] & B_\infty\times B^\vee_\infty \arrow[dd] & \\
& & P_{B\times B^\vee}^\times \arrow[ld] \\
M\times M^\vee \arrow[rru, "{\langle \cdot, \cdot\rangle}"] \arrow[r, "{\phi\times \phi^\vee}", swap] & B\times B^\vee &
\end{tikzcd}
\end{equation}

Here the top triangle is a commutative diagram of perfectoid spaces, and does depend on our choice of lifts. However, we claim that no matter the choice, we can still extract the spaces $E_\infty$, $E^\vee_\infty$, $A_\infty$, and $A^\vee_\infty$ just from the top triangle alone. 

\begin{proposition}\label{prop:E_inf-coordinates}
We have \[E_\infty\cong  P^\times_{\phi^\vee_\infty(m_1^\vee)}\times_{B_\infty}\cdots\times_{B_\infty}  P^\times_{\phi^\vee_\infty(m_r^\vee)},\] with notation as in Corollary \ref{cor:P_inf-fibers}.
\end{proposition}
 
\begin{proof}
This is the analogue of Equation (\ref{eq:E coords}), and the proof builds off of this decomposition. The map $[p]:E\rightarrow E$ is induced by the maps $[p]: P^\times_{B\times\phi^\vee(m_i^\vee)}\rightarrow  P^\times_{B\times\phi^\vee(m_i^\vee)}$, which live over $[p]:B\rightarrow B$. By Corollary \ref{cor:P_inf-fibers}, we have $ P^\times_{\phi^\vee_\infty(m_i^\vee)}\sim \varprojlim_{[p]}  P^\times_{B\times\phi^\vee(m_i^\vee)}.$  Combining these, the proposition follows from uniqueness of perfectoid tilde-limits.
 
\end{proof}

The bilinear form $\langle\cdot,\cdot\rangle_\infty$ defines a group homomorphism $M \rightarrow E_\infty$ by sending $m\mapsto (\langle m,m_1^\vee\rangle_\infty,\dots,\langle m,m_r^\vee\rangle_\infty)$. By the commutativity of Diagram (\ref{eq:Poincare form-inf}), this lifts the homomorphism $M\rightarrow E$. In the language of \cite[Section 4]{AWS}, we have recovered a choice $M_\infty$ of $p$-power roots in $E$ of the lattice $M$, and our homomorphism is precisely the map $M_\infty\rightarrow E_\infty$ from \cite[Definition 4.2]{AWS}. Continuing with the notation from \cite{AWS}, we define $D_\infty:= M_\infty\times_\Z \Z_p$. Then by \cite[Theorem 4.6.5]{AWS}, we have $A_\infty=D_\infty\times^{M_\infty} E_\infty$. By symmetry, the same bilinear form also defines $M^\vee\rightarrow E^\vee_\infty$, from which we can recover $A^\vee_\infty$. In summary, we have:

\begin{proposition}\label{prop:A_inf-from-Poincare} 
The spaces $E_\infty$, $E^\vee_\infty$, $A_\infty$, and $A^\vee_\infty$ are determined by the top triangle of Diagram (\ref{eq:Poincare form-inf}).
\end{proposition}

Tilting the top triangle of (\ref{eq:Poincare form-inf}) and applying Proposition \ref{prop:Poincare-tilt} to fill in the bottom triangle, we get a diagram full of leading notation which will take some unpacking

\begin{equation}\label{eqn:Poincare-triangle-inf-tilt}
\begin{tikzcd}
& &  P'^\times_\infty \arrow[ld] \arrow[dd] \\
M'\times M'^\vee \arrow[rru, "{\langle \cdot,\cdot\rangle'_\infty}"] \arrow[r, "{\phi'_\infty\times\phi'^\vee_\infty}", swap] \arrow[dd, equal] & B'_\infty\times B'^\vee_\infty \arrow[dd] & \\
& &  P^\times_{B'\times B'^\vee} \arrow[ld] \\
M'\times M'^\vee \arrow[rru, "{\langle \cdot, \cdot\rangle'}"] \arrow[r, "{\phi'\times \phi'^\vee}", swap] & B'\times B'^\vee. &
\end{tikzcd}
\end{equation}

Here we have applied the isomorphism $B_\infty^\flat\cong B'_\infty$ which comes because the special fibers of $B$ and $B'$ agree, the analogous isomorphism for $B'^\vee$, and the isomorphism $ P^\times_\infty\cong  P'^\times_\infty$ from Proposition \ref{prop:Poincare-tilt}, which also gives the commutativity of the right vertical square. We are \emph{defining} $M'\times M'^\vee$ to be the tilt of $M\times M^\vee$, and the morphisms filling out the top triangle are defined by the tilting correspondence. The morphisms in the bottom triangle are defined by composition of the morphisms in the top triangle with projections to the bottom. The map $\phi'\times\phi'^\vee$ is the composition of homomorphisms of adic groups, so it is itself a group homomorphism.

\begin{lemma}\label{lemma:bilinear-tilt}
The map $\langle\cdot,\cdot\rangle':M'\times M'^\vee\rightarrow  P^\times_{B'\times B'^\vee}$ is a non-degenerate bilinear form.
\end{lemma}

\begin{proof}
Bilinearity follows because tilting preserves the two partial group laws on $ P_\infty$. Non-degeneracy follows as in Lemma \ref{lemma:M' a lattice}. 
\end{proof}

Applying \cite[Proposition 6.1.8]{Lut} to this bilinear form, we get a pair of abeloids $A':=E'/M'$ and $A'^\vee:=E'^\vee/M'^\vee$. Our goal is to show that $A'_\infty\cong A_\infty^\flat$, which will imply by symmetry that $A'^\vee_\infty\cong (A'^\vee_\infty)^\flat$. By Theorem \ref{thm:tilting A}, it is enough to show that $E'_\infty\cong E_\infty^\flat$ and that the map $M'\rightarrow E'_\infty$ induced by $\langle\cdot,\cdot\rangle'_\infty$ agrees with the tilt of the map $M\rightarrow E_\infty$ induced by $\langle\cdot,\cdot\rangle'_\infty$.

By the decomposition in Proposition \ref{prop:E_inf-coordinates}, to see that $E'_\infty\cong E_\infty^\flat$ it is enough to see that 
\[( P^\times_{\phi^\vee_\infty(m^\vee)})^\flat\cong  P'^\times_{\phi'^\vee_\infty (m^\flat)},\] but this is precisely the statement of Corollary \ref{cor:Poincare-fiber-tilt}. The statement about the bilinear forms follows because $h':M'\rightarrow E'$ is determined by the $r$ maps 
\[\langle\cdot,\phi'^\vee(m'^\vee_i)\rangle'_\infty:M'\rightarrow  P'^\times_{B'_\infty\times \phi'^\vee_\infty(m_i^\vee)},\]
which are by definition the tilts of the $r$ analogous maps defining $h$ in terms of $\langle\cdot,\cdot\rangle_\infty.$ We therefore get the desired result:

\begin{proposition}\label{prop:dual tilts}
Let $A$ be an abeloid over a perfectoid field $K$ with dual $A^\vee$, let $A_\infty$ and $(A^\vee)_\infty$ be the corresponding perfectoid covers. Then there is an abeloid $A'$ over $K^\flat$ with $A'_\infty\cong (A_\infty)^\flat$ and $(A'^\vee)_\infty\cong (A^\vee)_\infty^\flat$.
\end{proposition}

With this in hand, we tilt our line bundles.

\begin{theorem}\label{thm:tilting bundles}
Given a line bundle $L$ on an abeloid $A/K$, there is some model $\fA_\alpha$ of $A$ such that $L$ extends to a formal line bundle $\fL_\alpha$ on $\fA_\alpha$. If the degree of $L$ is prime to $p$, then after replacing $K$ by a pro-$p$ extension, there is some $A'/K^\flat$ as in Theorem \ref{thm:tilting A} and line bundle $L'$ extending to a formal line bundle $\fL'_{\alpha}$ such that the special fibers of the formal line bundles agree. The line bundle $L$ is ample if and only if $L'$ is.
\end{theorem}

\begin{proof}
By Proposition \ref{prop:Raynaud line bundles}, $L$ is equivalent to a line bundle $N$ on $B$, group homomorphism $\sigma:M\rightarrow M^\vee$, and trivialization $\chi:M\rightarrow \phi^*N$ such that the diagram 
\begin{equation}\label{eq:bundle diagram2}
	\begin{tikzcd}
	M \arrow[r, "\phi"] \arrow[d, "\sigma"] & B \arrow[d, "\varphi_N"] \\
	M^\vee \arrow[r, "\phi^\vee"] & B^\vee.
	\end{tikzcd}
\end{equation}
commutes, and the equation
\begin{equation}\label{eq:r compatibility2}
\chi(m_1+m_2)\otimes \chi(m_1)^{-1}\otimes \chi(m_2)^{-1}=\langle m_1,\sigma(m_2)\rangle
\end{equation}
holds for all $m_1,m_2\in M$. 

By Proposition \ref{prop:formal bundles} and Proposition \ref{prop:formal alpha}, $L$ extends to some formal model $\fA_\alpha$, and is equivalent to a choice of formal model $\mathfrak N'$ of $N$ and a continuous function $f$ built from ``tropical data" associated to $\sigma$ and $\chi$. To construct the desired $L'$, we use the data for $L$ to construct analogous data $N', \sigma',$ and $\chi'$ for $A'$. 

As we saw while proving Proposition  \ref{prop:good bundles}, there is some lift of $\widetilde{B'}$ to a formal abelian scheme $\overline{B'}$ over $K^{\flat\circ}$ such that the line bundle $\overline {\mathfrak N}$ lifts to a line bundle $\mathfrak N'$ with generic fiber $N'$. The line bundle $N'$ is ample if and only if its special fiber is. Using this $B'$ when we construct $A'$, we get the desired $N'$. We note that this step requires the condition on the degree of $L$ (and is the only place where it is used).

Constructing $A'$ and $A'^\vee$ simultaneously as in Proposition \ref{prop:dual tilts}, we get isomorphisms $\flat:M\cong M'$ and $\flat:M'\cong M'^\vee$ which agree on the special fiber (that is, their actions on the special fiber agree). We use these isomorphisms to transfer $\sigma$ to a group homomorphism $\sigma':M'\rightarrow M'^\vee$ which agrees with $\sigma$ on the special fiber. These isomorphisms therefore respect the resulting bilinear forms (\ref{eq:line bundle form}), so one is positive definite precisely when the other is. By the ampleness criteria in Proposition \ref{prop:Raynaud line bundles}, we see that $L$ is ample if and only if $L'$ is. 

To tilt the trivialization $\chi:M\rightarrow \phi^*N\cong M\times K$, we let $m'_1,\dots,m'_r$ be the image in $M'$ of the basis $m_1,\dots, m_r$ of $M$ under $\flat$. We define $r'(m'_i)$ to be $(m_i', a_i')$ for any $a_i'$ in $\sharp^{-1}(\chi(m_i))$ - perhaps making a pro-$p$ extension of $K$ to ensure that such an element exists - then extend to all of $M'$ using the rule (\ref{eq:r compatibility2}). 
\end{proof}

The condition on the degree of $L$ is somewhat annoying, but we believe it is unavoidable. For our perfectoid applications, Lemma \ref{lemma:isogeny covers} will let us get around the issue.

As a first application, we can upgrade Theorem \ref{thm:tilting A} to work for abelian varieties, not just abeloids.

\begin{theorem}\label{thm:tilting polarizations}
Let $A$ be an abelian variety over a perfectoid field $K$ with value group contained in $\Q$. Let $A_\infty$ be the perfectoid cover constructed in Theorem \ref{thm:cover of A}. After replacing $K$ by a pro-$p$ extension, there is an abelian variety $A'$ over $K^\flat$ such that $(A_\infty)^\flat\cong (A')_\infty$.
\end{theorem}

\begin{proof}
Proposition \ref{thm:tilting bundles} proves this for abelian varieties with a polarization of degree prime to $p$. Fixing any polarization $\varphi_A$ of $A$, we can quotient $A$ by a $p$-Sylow subgroup of $\ker(\varphi_A)$, giving us an isogeny $A\rightarrow A/\ker(\varphi_A)=:A_1$ of degree $p^k$ for some positive integer $k$. Then $\varphi_A$ induces a polarization on $A_1$ of degree prime to $p$. By Lemma \ref{lemma:isogeny covers} we can prove the theorem after replacing $A$ by $A_1$.
\end{proof}

\begin{corollary}\label{cor:Ginf bundles}
Given a line bundle $L$ on an abelian variety $A$ over $K$, after replacing $K$ by a pro-$p$ extension, there is a $(\G_{m,K})_\infty$ torsor $L^\times_\infty$ living over $q_1^*(L^\times)$.
\end{corollary}

\begin{proof}
Now that we can transfer line bundles from $A$ to $A'$, the argument is essentially the same as in Proposition \ref{prop:good bundles}. We can again use the Theorem of the cube to construct $p$th power roots at finite levels of the tower, and we can use Theorem \ref{thm:tilting bundles} to find a system of line bundles over $A'$ with isomorphic special fibers. On the tilt side, Lemma \ref{lemma:Gminfty exists char p} gives a $\G_{m,\infty}^\flat$-torsor. Untilting, we get a $\G_{m,\infty}$ torsor with the correct special fiber, so it must be the one we are looking for.
\end{proof}

\begin{remark}\label{rem:Ben}
Ben Heuer tells us that he can prove the stronger result that $\Pic(A_\infty)\cong \Pic(A_\infty^\flat)$, using that $H^i(A_\infty,\Z_p)=0$ for $i>0$. In particular, this gives an affirmative answer to \cite[Open Problem 5.14]{DH} in this case. See \cite[Theorem 1.8]{BenGood} for a description of $\Pic^0(A_\infty)$. 
\end{remark}


\section{Weight-monodromy for complete intersections in abelian varieties}\label{sec:weight-monodromy}


We are now ready to prove our main result: the weight-monodromy conjecture for complete intersections in abelian varieties. In Section \ref{sub:WMC background}, we give a more detailed overview of the setup of the weight-monodromy conjecture. We then prove the theorem in Section \ref{sub:WMC proof}.

\subsection{Background on the weight-monodromy conjecture}\label{sub:WMC background}

Let $X$ be a proper smooth $d$-dimensional variety over a local field $k$ of residue characteristic $p$, let $q$ be the cardinality of the residue field. Let $\ell$ be a prime not equal to $p$, let $i$ be an integer between 0 and $2d$. Then the $i$th \'etale cohomology group $V:=H^i_{\et}(X_{\overline k},\overline \Q_\ell)$ is a finite-dimensional $\overline \Q_\ell$-vector space with an action of the absolute Galois group $G_k:=\Gal(\overline k/k)$ of $k$. This action defines a Galois representation $\rho:G_k\rightarrow \GL(V)$.

In this section, we explain how to define two filtrations on $V$. The monodromy filtration comes from the inertia subgroup $I_k$ of $G_k$, and the weight filtration comes from the action of a lift of Frobenius from the residue field. The weight-monodromy conjecture says that these filtrations are the same.

The monodromy filtration comes out of the following result of Grothendieck. 

\begin{proposition}[{\cite[Proposition on page 515]{ST}}]\label{prop:unipotent}
Let $\rho:G_k\rightarrow \GL(V)$ be a finite-dimensional $\overline \Q_\ell$-representation. Then there is an open subgroup $I_1$ of the inertia subgroup $I_k$ such that for all $g$ in $I_1$, the matrix $\rho(g)$ is unipotent.
\end{proposition}

We can now define the monodromy filtration using the nilpotent operator $N:=\log(\rho(g))$ on $V$ constructed in the course of the previous proof (after a Tate twist).

\begin{DefProp}\label{def:monodromy filtration}
Let $N:V\rightarrow V$ be a nilpotent operator on a finite-dimensional vector space. Then there is a unique finite increasing filtration $\Fil_j^N$ of $V$ such that $N\Fil_j^N$ is contained in $\Fil_{j-2}^N$ and such that $N^k$ induces an isomorphism $\Gr_{i+j}^N V\cong \Gr_{i-j}^N V.$ Explicitly, we have
\[\Fil_j^N V = \displaystyle\sum_{j_1-j_2=j} \ker N^{j_1+1}\cap \tIm N^{j_2}.\]
\end{DefProp}

\begin{proof} 
\cite[Proposition 1.6.1]{Weil2}
\end{proof} 

The monodromy filtration therefore depends only on the pro-$\ell$ part of the inertia in the Galois group. Taking pro-$p$ extensions of our base field will therefore preserve the monodromy filtration.

To define the weight filtration, we fix an element $\Phi\in \Gal_k$ which maps to geometric Frobenius on the residue field. We use the matrix $\rho(\Phi)\in \GL(V)$ to decompose $V$ into eigenspaces. Work of Rapoport-Zink, along with de Jong's alterations theorem, gives us information about the eigenvalues.

\begin{proposition}\label{prop:weight decomp}
With notation as above, we have a decomposition \[V=\bigoplus_{j=0}^{2i} W'_j,\] where the eigenvalues of $\rho(\Phi)$ acting on $W'_j$ are Weil numbers of weight $j$. That is, they are algebraic numbers, and for any embedding of $k$ into $\C$, they have absolute value $q^{j/2}$.
\end{proposition}

\begin{proof}[Sketch]
See \cite{RZ} for details, or the nice expositions in \cite[Section 3]{Ill_monodromy} and \cite[Section 2]{Scholl}. We use de Jong's alterations theorem to reduce to the case where $X$ has a semistable model. The special fiber of this model consists of finitely many smooth varieties, intersecting transversely. By the Riemann hypothesis part of the Weil conjectures, the Frobenius eigenvalues on each irreducible component of the special fiber are all Weil numbers of weight $i$. If $j$ components have non-empty intersection, it is a smooth variety with codimension $j$, so the Frobenius eigenvalues are Weil numbers of weight $i-j$. The Rapoport-Zink spectral sequence builds the actual cohomology of $X$ out of these pieces, giving the desired decomposition.
\end{proof}

If we chose a different lift of Frobenius, we'd get a different decomposition of $V$. However, the weights and dimensions remain the same, as does the weight filtration:

\begin{definition}\label{def:weight filtration}
With notation as in Proposition \ref{prop:weight decomp}, the \emph{weight filtration} on $V$ is defined by letting the $m$th piece be \[W_m=\bigoplus_{j=0}^m W'_j.\]
\end{definition}

For any lift $\Phi$ of geometric Frobenius, we have $N\Phi = q\Phi N$. This implies that the monodromy operator interacts nicely with the weight filtration: we have $NW_m\subset W_{m-2}$. The weight-monodromy conjecture says more:

\begin{conjecture}[Deligne, \cite{Hodge}]\label{conj:WMC7}
For a proper smooth variety $X$ over a local field $k$, the weight and monodromy filtrations on the $i$th $\ell$-adic \'etale cohomology group of $X$ are the same. Equivalently, the eigenvalues of any lift $\Phi$ of geometric Frobenius on any graded piece $\operatorname{gr}_j^N V$ of the monodromy filtration are Weil numbers of weight $i+j$.
\end{conjecture}

The equivalence of these statements is in \cite[Section 1.6]{Weil2}. Deligne gave another interpretation of this conjecture in terms of the poles of $L$-functions associated to sheaves. He used this interpretation to prove the conjecture when $k$ has characteristic $p$ and $X$ is defined over a curve. See \cite[Section 1.8]{Weil2} for the proof, or \cite[Section 9]{PS} for a quick overview of the idea.  


\subsection{Weight-monodromy for complete intersections in abelian varieties}\label{sub:WMC proof}

Scholze's approach to weight-monodromy \cite[Section 9]{PS} is summarized in the following:

\begin{proposition}\label{prop:Sch-framework}
Let $Y$ be a $d$-dimensional, geometrically connected, proper smooth variety over a local field $k$ of residue characteristic $p$, fix a prime $\ell\neq p$. Let $K$ be a perfectoid pro-$p$ extension of $k$, let $G$ be the absolute Galois group of $K$ and $K^\flat$. To prove the weight-monodromy conjecture for $Y$, it is enough to construct a $d$-dimensional, geometrically connected, proper smooth variety $Z$ over $K^\flat$ which can be defined over some local field contained in $K^\flat$, along with maps \[f_i^*: H^i(Y_{\C_p,\et}, \overline\Q_\ell)\rightarrow H^i(Z_{\C_p^\flat,\et},\overline\Q_\ell)\] for $i\in\{0,\dots,2d\}$ which are $G$-equivariant, compatible with cup product, and such that $f^*_{2d}$ is an isomorphism.
\end{proposition}

\begin{proof}
We first check that the maps $f_i^*$ are all injective. To see this, we work with the diagram

\begin{center}
\begin{tikzcd}
H^i(Y_{\C_p,\et}, \overline\Q_\ell)\times H^{2d-i}(Y_{\C_p,\et}, \overline\Q_\ell) \arrow[d, "{f_i^*\times f_{2d-i}^*}"] \arrow[r, "\cup"] & H^{2d}(Y_{\C_p,\et}, \overline\Q_\ell) \arrow[d, "{f_{2d}^*}"] \arrow[r, "\cong"] & \overline\Q_\ell \arrow[d, "\cong"] \\
H^i(Z_{\C^\flat_p,\et}, \overline\Q_\ell)\times H^{2d-i}(Z_{\C^\flat_p,\et}, \overline\Q_\ell) \arrow[r, "\cup"] & H^{2d}(Z_{\C^\flat_p,\et}, \overline\Q_\ell)  \arrow[r, "\cong"] & \overline\Q_\ell. 

\end{tikzcd}
\end{center}

By Poincar\'e duality, the rows give perfect pairings between $H^i$ and $H^{2d-i}$ for every $i$. The left square commutes because the $f_i^*$ are compatible with cup product. The right square is all isomorphisms because $f_{2d}^*$ is an isomorphism and $Y$ and $Z$ are geometrically connected. 

Let $c\in H^i(Y_{\C_p,\et}, \overline\Q_\ell)$ be a class with $f_i^*(c)=0\in H^i(Z_{\C^\flat_p,\et}, \overline\Q_\ell).$ As every class in $H^{2d-i}(Z_{\C^\flat_p,\et}, \overline\Q_\ell)$ pairs to 0 with 0, commutativity of the diagram tells us that every class in $H^{2d-i}(Y_{\C_p,\et}, \overline\Q_\ell)$ must pair to 0 with $c$. As the top row is a perfect pairing, $c$ must therefore be 0, so $f_i^*$ is injective as desired.

By $G$-equivariance, the map $f_i^*$ respects the monodromy filtrations on $H^i(Y_{\C_p,\et}, \overline\Q_\ell)$ and $H^i(Z_{\C^\flat_p,\et}, \overline\Q_\ell)$. By injectivity, the $j$th graded piece of $H^i(Y_{\C_p,\et}, \overline\Q_\ell)$ is a sub-vector space of the $j$th graded piece of $H^i(Z_{\C^\flat_p,\et}, \overline\Q_\ell).$ Every Frobenius eigenvalue for $Y$ is therefore a Frobenius eigenvalue for $Z$. As $Z$ is proper, smooth, and can be defined over a characteristic $p$ local field, Deligne proved weight-monodromy for $Z$. That is, every Frobenius eigenvalue of $Z$ is a Weil number of degree $i+j$. The same is therefore true for $Y$ and weight-monodromy holds. 
\end{proof}

Specializing to our case of interest, let $Y$ be a complete intersection in an abelian variety $A$ which satisfies the properties of Proposition \ref{prop:Sch-framework}. We must construct a variety $Z\subset A'$ and maps $f^*_i$ which also satisfy the properties of Proposition \ref{prop:Sch-framework}. To construct $Z$, we use the following approximation result: our analogue of Scholze's approximation result \cite[Proposition 8.7]{PS}. 

\begin{proposition}\label{prop:approx}
Let $Y\subset A$ be a hypersurface, let $\tilde{Y}$ be a small open neighborhood of $Y$, where we are considering $Y$ and $A$ as adic spaces. Let $q_1:|(A_\infty)^\flat|\rightarrow |A|$ be the map of topological spaces induced by tilting to $A_\infty$, then projecting to the bottom of the inverse system. Let $q_1':|(A_\infty)^\flat|\rightarrow |A'|$ be the map induced by projection. Then there exists a hypersurface $Z'\subset A'$ such that $q_1'^{-1}(Z')\subset q_1^{-1}(\tilde Y)$. Equivalently, there exists some open $\tilde{Y'}\subset A'$ with $q_1'^{-1}(\tilde{Y'})=q_1^{-1}(\tilde Y)$ and a hypersurface $Z'\subset \tilde{Y'}$.
\end{proposition}

\begin{proof}

There is some invertible sheaf $\cL_1$ on $A$ and an element $f\in H^0(A,\cL_1)$ with zero locus $Y$. We will inductively define an inverse system of line bundles $L_1\rightarrow L_2\rightarrow\cdots$ so that the induced direct system of rings of global functions $R_1\leftarrow R_2\leftarrow\cdots$ gives a perfectoid ring $R_\infty=(\varinjlim R_i)^\wedge$. The tilt $R_\infty^\flat$ will be given by an analogous limit over $A'$. This will allow us to approximate $q_1^*f$ by a function pulled back from $A'$.

Given an invertible sheaf $\cL_n$ on $A$, the associated line bundle (in the category of schemes, not adic or rigid spaces) is \[L_n:=\underline{\Spec}\big( \bigoplus_{m\in \N} \cL_n^{\otimes m}\big).\] We write \[R_n \cong  \bigoplus_{m\in \N} H^0(A,\cL_n^{\otimes m})\] for the ring of global functions of this line bundle. Note that if we worked in the category of rigid or adic spaces and analytified this line bundle (as we did in the previous chapter), we would get many more global functions: the rigid affine line has many more global functions than just $K[T]$. However, if we start with a formal model $\fL_n$ in the sense of \ref{def:formal_line_bundle}, we have $H^0(\fA_n, \fL_n)[\frac{1}{p}]\cong H^0(A^{an},\cL^{an}_n)\cong H^0(A,\cL_n)$, where the last isomorphism is rigid GAGA \cite[\S 5]{Kopf}, \cite[Example 3.2.6]{Conrad}. These isomorphisms compatibly extend to the tensor powers $\cL_n^{\otimes m}$, so we can use them define a $K^\circ$-submodule of $R_n$. We topologize $R_n$ by letting this submodule be a ring of definition. This allows us to make sense of the completion in the definition of $R_\infty$.

When $p$ is odd, we define \[L_{n+1}=L_n^{\otimes \frac{p+1}{2}}\otimes [-1]^*L_n^{\frac{p-1}{2}}.\] By the Theorem of the cube, we have $L_{n+1}^{\otimes p}\cong [p]^*L_n$. If $p=2$, we instead let $L_{n+1}=L_n^{\otimes 5}\otimes[-1]^*L_n^{\otimes 3}$, which satisfies $L_{n+1}^{\otimes 2}\cong [4]^*L_n$, then argue similarly. As in the discussion after Definition \ref{def:Gm infty torsor}, we have a morphism $L_{n+1}\rightarrow L_{n+1}^{\otimes p}$ induced by the inclusion $ \bigoplus_{m\in \N} \cL_{n+1}^{\otimes pm}\hookrightarrow  \bigoplus_{m\in \N} \cL_{n+1}^{\otimes m}$. Now we can define the morphism $L_{n+1}\rightarrow L_n$ as the composition \[L_{n+1}\rightarrow L_{n+1}^{\otimes p}\cong [p]^*{L_n} \rightarrow L_n.\] This fits into the diagram 
\begin{equation}
\begin{tikzcd}
L_{n+1} \arrow[d, "f_n"] \arrow[r] & A \arrow[d, "{[p]}"] \\
L_n \arrow[r] & A.
\end{tikzcd}
\end{equation}

By Proposition \ref{prop:formal bundles}, there is some formal model $\fA_\alpha$ of $A$ and formal line bundle $\fL_1$ on $\fA_\alpha$ which is a model of $L_1$. Inductively applying Corollary \ref{cor:formal roots}, there is then a formal line bundle $\fL_n$ on $\fA_{\alpha/p^n}$ which is a model of $L_n$. We therefore get a formal model of the above diagram
\begin{equation}\label{eq:formal line bundle square}
\begin{tikzcd}
\fL_{n+1} \arrow[d] \arrow[r] & \fA_{\alpha/p^{n+1}} \arrow[d, "{[p]}"] \\
\fL_n \arrow[r] & \fA_{\alpha/p^n}.
\end{tikzcd}
\end{equation}

\begin{lemma}
The mod $\varpi$ special fiber of the map $\fL_{n+1}\rightarrow\fL_{n}$ in diagram \ref{eq:formal line bundle square} factors through relative Frobenius.
\end{lemma}

\begin{proof}
In the proof of Theorem \ref{thm:formal model for A}, we constructed formal analytic covers $\cU_{n+1}=\{U_{i,n+1}\}$ and $\cU_n=\{U_{j,n}\}$ of $A$ leading to the models $\fA_{\alpha/p^{n+1}}$ and $\fA_{\alpha/p^n}$. These covers were constructed so that the map $[p]:A\rightarrow A$ sends each $U_{i,n+1}$ to some $U_{j,n}$, and the special fibers of the corresponding map of formal affine opens factors through relative Frobenius. We have chosen $\alpha$ so that the line bundles $L_{n+1}$ and $L_n$ are trivial over the covers $\cU_{n+1}$ and $\cU_n$ respectively. So locally on the source, $f_n$ restricts to the map $\A^1_K\times U_{i,n+1}\rightarrow \A^1_K\times U_{j,n}$ which is the product of the map $\A^1_K\rightarrow \A^1_K$ sending the coordinate $T$ to $T^p$ with the restriction of $[p]$ to $U_{i,n+1}$. Both of these factor through relative Frobenius on the special fiber, so their product does as well.
\end{proof}

\begin{corollary}
The ring $R_\infty$ is perfectoid.
\end{corollary}

\begin{proof}
As the morphism $\overline{\fL_{n+1}}\rightarrow\overline{\fL_{n}}$ factors through relative Frobenius, the induced map $R^\circ_n/\varpi\rightarrow R^\circ_{n+1}/\varpi$ also factors through relative Frobenius. Indeed, by the Kunneth formula (see \cite[Theorem 14]{Kempf} for a version with the appropriate generality), we have
 \[\Gamma(\overline{\fL_{n+1}}^{(p)})= \Gamma(\overline{\fL_{n+1}}\times_{\Spec(K^\circ/\varpi), \Frob} \Spec(K^\circ/\varpi))\cong  \Gamma(\overline{\fL_{n+1}})\otimes_{K^\circ/\varpi,\Frob} K^\circ/\varpi\cong\]
 \[R_{n+1}^\circ/\varpi\otimes_{K^\circ/\varpi,\Frob} K^\circ/\varpi\cong (R_{n+1}^\circ/\varpi)^{(p)},\] so the correct universal property is satisfied on global functions. Now the local argument of Proposition \ref{prop:pF tower} applies.
\end{proof}

Now by Theorem \ref{thm:tilting bundles}, we have a line bundle $L'_1$ over an abelian variety $A'/K^\flat$ with a formal model $\fL'_1$ defined over $\fA'_{\alpha}$ such that the mod $\varpi^\flat$ special fiber agrees with the mod $\varpi$ special fiber of $\fL_1\rightarrow \fA_{\alpha}$. Tracing through the above construction with $\fL'_1$ in place of $\fL_1$ gives $R_\infty^\flat=(\varprojlim R'_i)^\wedge$.

We can let $\tilde Y=\{x\in A: |f(x)|\leq \epsilon\}$ for some $\epsilon$, where the absolute value comes from the integral model for $L$ defined in Proposition \ref{prop:formal bundles}. Pulling back $\tilde Y$ and $f$ along $q_1$, we get an open set in $A_\infty$ defined as an $\epsilon$-neighborhood of $q_1^*(f)$, which we can view as an element of $R_\infty$. Using \cite[Proposition 9.2.7]{Bhatt_notes}, which builds off of \cite[Lemma 6.5]{PS}, we can approximate the vanishing locus of $q_1^*(f)$ by that of an element $g^\sharp$ in the image of $\sharp:R_\infty^\flat\rightarrow R_\infty.$ We can tweak $g$ so that it is the pullback of some $g\in R'_i$ from some finite level of the tower, and so that it is defined over $k$. Then $q_1^{-1}(\tilde Y)$ is an $\epsilon$-neighborhood of the vanishing locus of $g$. Taking a big enough power of $g$, we can assume that it comes from the bottom copy of $A'$ on the tower. The vanishing locus $Z'$ of $g$ is the desired hypersurface. 
\end{proof}

As in \cite[Corollary 8.8]{PS}, we can intersect hypersurfaces to get the analogous result when $Y$ is a complete intersection in $A$. We get an open $\tilde Y'\subset A'$ such that $q_1'^{-1}(\tilde Y')=q_1^{-1}(\tilde Y)$ as topological spaces in $A'_\infty$, and a closed variety $Z'\subset A'$ of dimension $d$. We can choose $Z'$ to be defined over a local field in $K^\flat$, and take a component to ensure that it is geometrically connected. We emphasize that even though we used the theory of adic spaces to construct $Z'$, it is a variety over a local field.

As in \cite[Section 9]{PS}, we now make some reductions. By \cite[Theorem 3.6a]{Huber_finite}, the relevant cohomology groups can be computed on the associated adic spaces. We therefore analytify everything from now on, working strictly with adic spaces. By \cite[Theorem 3.8.1]{Huber_etale}, there is an open neighborhood $\tilde Y$ of $Y$ in $A$ such that $Y$ and $Y'$ have the same cohomology. By de Jong's alterations theorem \cite{deJ}, we can choose a projective smooth alteration $Z\rightarrow Z'$. The induced maps on cohomology are injective, giving $H^i(Z'_{\C_p^\flat, \et},\overline\Q_\ell)$ as a direct summand of $H^i(Z_{\C_p^\flat, \et},\overline\Q_\ell)$. It is therefore enough to construct maps from the cohomology of $Y$ to the cohomology of $Z$.

The final difficulty, compared with Scholze's proof for toric varieties, is that the map $[p]:A'\rightarrow A'$ does not induce an isomorphism of \'etale topoi. We therefore have no analogue of Scholze's projection map $\pi:\P^n_{K^\flat}\rightarrow \P^n_K$ (defined on \'etale topoi and topological spaces, and more generally defined for toric varieties). To get around this, our cohomology computations must run through the perfectoid covers more explicitly.

The tilting equivalence along with \cite[Theorem 7.17]{PS} gives isomorphisms of \'etale topoi

\[\varprojlim_{[p]} (A'_{K^\flat})^\sim_{\text{\'et}} \cong (A_\infty^\flat)^\sim_{\text{\'et}}  \cong \varprojlim_{[p]} (A_K)^\sim_{\text{\'et}}.\] 
By \cite[Theorem 7.16]{PS}, these inverse limits commute with taking fiber products along \'etale maps to $A_K$ or $A'_{K^\flat}$. We therefore get the following diagram of \'etale topoi

\begin{center}
\begin{tikzcd}
 & (A_\infty^\flat)^\sim_{\text{\'et}} \arrow[rd, "q_1"]  \arrow[ld, swap, "q_1'"] & \\
 (A'_{\C_p^\flat})^\sim_{\text{\'et}} & (q_1^{-1}(\tilde Y)_{\C_p^\flat})^\sim_{\text{\'et}} \arrow[dr] \arrow[u] \arrow[dl] &  (A_{\C_p})^\sim_{\text{\'et}}  \\
(\tilde Y'_{\C_p^\flat})^\sim_{\text{\'et}} \arrow[u] & & (\tilde Y_{\C_p})^\sim_{\text{\'et}} \arrow[u] \\
(Z_{\C_p^\flat})^\sim_{\text{\'et}} \arrow[u] & & (Y_{\C_p})^\sim_{\text{\'et}}. \arrow[u, "{\cong}"]

\end{tikzcd}
\end{center}
Here the parallelograms are fiber diagrams, and the diagonal maps are projection maps to the bottom of the tilde-limits constructed by iterating the maps $[p]$ for $A'$ and $A$. 

We first construct the maps $f^*_i$ from this diagram, then check that it satisfies the desired properties. In the above diagram, pullback induces isomorphisms 
\[H^i(\tilde Y_{\C_p,\et},\Z/\ell^m\Z)\rightarrow H^i(Y_{\C_p,\et},\Z/\ell^m\Z)\]
for all $i$ and $m$, and maps
 \[H^i(\tilde Y_{\C_p,\text{\'et}},\Z/\ell^m\Z)\rightarrow H^i(q_1^{-1}(\tilde Y)_{\C_p^\flat,\et},\Z/\ell^m\Z),\]
 \[H^i(\tilde Y'_{\C^\flat_p,\text{\'et}},\Z/\ell^m\Z)\rightarrow H^i(Z_{\C_p^\flat,\text{\'et}},\Z/\ell^m\Z),\]
It remains to construct a map \[H^i(q_1^{-1}(\tilde Y)_{\C_p^\flat},\Z/\ell^m\Z)\rightarrow H^i(\tilde Y'_{\C^\flat_p,\text{\'et}},\Z/\ell^m\Z)\]

To do this, we factor the morphism $[p]:A'\rightarrow A'$ as \[A'\xrightarrow{\varphi} A'/A'[p]^0\xrightarrow{\psi} A'.\] Here $A'[p]^0$ is the identity component of the finite group scheme $A'[p]$. The morphism $\varphi$ is radicial and induces a homeomorphism of topological spaces, so it induces an isomorphism of \'etale topoi as in \cite[Corollary 7.19]{PS}. The morphism $\psi$ is finite \'etale.

Restricting these morphisms to \[[p]^{-1}(\tilde Y'_{\C_p^\flat})\xrightarrow{\varphi} \psi^{-1}(\tilde Y'_{\C_p^\flat})\xrightarrow{\psi} \tilde Y'_{\C_p^\flat},\] we again have a radicial homeomorphism followed by a finite \'etale map. As the radicial map induces an isomorphism of \'etale topoi, we get isomorphisms \[H^i([p]^{-1}(\tilde Y)_{\C_p^\flat},\Z/\ell^m\Z)\xrightarrow{\cong} H^i(\psi^{-1}(\tilde Y_{\C_p^\flat}),\Z/\ell^m\Z)\] for all $i$ and $m$. As $\psi$ is finite \'etale, the trace map \cite[Tag 03SH]{Stacks} gives maps \[H^i(\psi^{-1}(\tilde Y_{\C_p^\flat}),\Z/\ell^m\Z)\rightarrow H^i(\tilde Y_{\C_p^\flat},\Z/\ell^m\Z)\] for all $i$ and $m$. These maps are induced by the adjunction maps $\psi_*\psi^*(\Z/\ell^m\Z)\rightarrow \Z/\ell^m\Z$ of sheaves on $\tilde Y_{\C^\flat_p}$, and are defined as the composition 
 \[H^i(\psi^{-1}(\tilde Y_{\C_p^\flat}),\Z/\ell^m\Z)\xrightarrow{\cong}H^i(\tilde Y_{\C_p^\flat}, \psi_*\psi^*(\Z/\ell^m\Z))\rightarrow H^i(\tilde Y_{\C_p^\flat},\Z/\ell^m\Z).\]

Composing these, we get a map
\[H^i([p]^{-1}(\tilde Y)_{\C_p^\flat},\Z/\ell^m\Z)\rightarrow H^i(\tilde Y_{\C_p^\flat},\Z/\ell^m\Z).\]
Iterating this process for the inverse system defining $q_1^{-1}(\tilde Y)_{\C_p^\flat}\sim\varprojlim_n [p^n]^{-1}(\tilde Y'_{\C_p}),$ we get the desired map 
 \[H^i(q_1^{-1}(\tilde Y)_{\C_p^\flat},\Z/\ell^m\Z)\rightarrow H^i(\tilde Y'_{\C^\flat_p,\text{\'et}},\Z/\ell^m\Z).\] 
 
All of the maps are induced by a combination of pullback along morphisms of adic spaces and (in the case of trace) homomorphisms of \'etale sheaves. By functoriality of these constructions, the compositions $H^i(Y_{\C_p,\et},\Z/\ell^m\Z)\rightarrow H^i(Z_{\C_p^\flat,\et},\Z/\ell^m\Z)$ are $G$-equivariant and compatible with cup product. Taking the inverse limit and tensoring with $\overline \Q_\ell$, the same is true of the maps $f^*_i:H^i(Y_{\C_p,\et},\overline \Q_\ell)\rightarrow H^i(Z_{\C_p^\flat,\et},\overline \Q_\ell).$ There is only one condition left to check for Proposition \ref{prop:Sch-framework} to apply.

\begin{lemma}
The map $f_{2d}^*:H^{2d}(Y_{\C_p,\et},\overline \Q_\ell)\rightarrow H^{2d}(Z_{\C_p^\flat,\et},\overline \Q_\ell)$ is an isomorphism.
\end{lemma}

\begin{proof}
Both cohomology groups are isomorphic to $\overline\Q_\ell$, so it is enough to show that the map is non-zero. It is enough to check this when the coefficients are $\Z/\ell^m\Z$, then pass to the inverse limit. In this situation, the above constructions give the following commutative diagram.

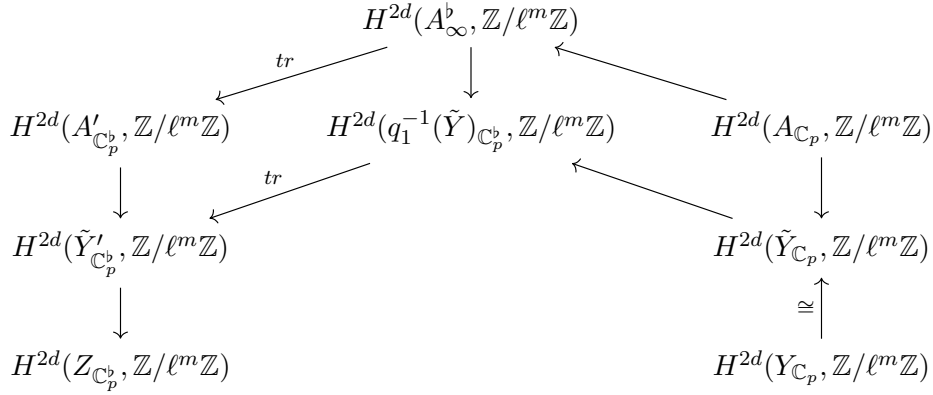
\begin{center}
\begin{tikzcd}
 & H^{2d}(A_\infty^\flat,\Z/\ell^m\Z) \arrow[ld, swap, "tr"]  \arrow[d] & \\
 H^{2d}(A'_{\C_p^\flat},\Z/\ell^m\Z) \arrow[d] & H^{2d}(q_1^{-1}(\tilde Y)_{\C_p^\flat},\Z/\ell^m\Z)  \arrow[dl,swap, "tr"] &  H^{2d}(A_{\C_p},\Z/\ell^m\Z) \arrow[lu] \arrow[d] \\
H^{2d}(\tilde Y'_{\C_p^\flat},\Z/\ell^m\Z) \arrow[d] & & H^{2d}(\tilde Y_{\C_p},\Z/\ell^m\Z) \arrow[ul] \\
H^{2d}(Z_{\C_p^\flat},\Z/\ell^m\Z)  & & H^{2d}(Y_{\C_p},\Z/\ell^m\Z) \arrow[u, "{\cong}"]

\end{tikzcd}
\end{center}

The argument is now essentially the same as \cite[Lemma 9.8]{PS}, with a slightly bigger diagram. As multiplication by $p$ induces an isomorphism on all $\ell$-adic cohomology groups of an abelian variety, the two diagonal maps at the top of the diagram are isomorphisms. It therefore suffices to show that the map $H^{2d}(A'_{\C^\flat_p},\Z/\ell^m\Z)\rightarrow H^{2d}(Z_{\C^\flat_p},\Z/\ell^m\Z)$ is non-zero. The $d$th power of the first Chern class of an ample line bundle on $A'$ will have non-zero image, so we are done. 

\end{proof}

This concludes the proof of Theorem \ref{thm:WMC abelian varieties}.

\begin{remark}\label{rem:lefschetz}
The words ``set-theoretic" in the statement of both Scholze's theorem and our version make it difficult to know what varieties the theorems apply to. For most varieties, it is unknown if it is possible to write them in this fashion. Ignoring the words set-theoretic then, most proper smooth complete intersections $Y$ in abelian varieties are not also proper smooth complete intersections in toric varieties. We learned the following argument from the MathOverflow answer \cite{FP}. When $\dim(Y)\geq 3$, the Lefschetz hyperplane theorem says that $\pi_1(Y_{\C_p})=\pi_1(A_{\C_p})$. If $Y$ were also a complete intersection in a proper smooth toric variety $X$, we would also have $\pi_1(Y_{\C_p})=\pi_1(X_{\C_p})$. But $\pi_1(X_{\C_p})$ is trivial and $\pi_1(A_{\C_p})$ is not, giving a contradiction.
\end{remark}

\bibliographystyle{alpha}  
\bibliography{thebib}

\end{document}